\documentclass[11pt]{amsart}
\usepackage{amsmath, amssymb, amsthm, amsfonts}
\usepackage{fullpage}
\usepackage{hyperref}
\usepackage{enumitem}
\usepackage{tkz-berge}
\usepackage[utf8]{inputenc}

\newcounter{dummy} \numberwithin{dummy}{section}
\newcounter{claimcounter}
\newcounter{casecounter}
\newtheorem{thm}[dummy]{Theorem}
\newtheorem{prop}[dummy]{Proposition}
\newtheorem{lemma}[dummy]{Lemma}
\newtheorem{cor}[dummy]{Corollary}
\newtheorem*{ldplowerboundthm}{Theorem \ref{ldplowerboundthm}}
\newtheorem*{highlowcor2}{Corollary \ref{highlowcor2}}
\newtheorem{problem}[dummy]{Problem}
\theoremstyle{definition}
\newtheorem{definition}[dummy]{Definition}
\newtheorem{setup}[dummy]{Setup}
\newtheorem{conditions}[dummy]{Conditions}
\newtheorem{example}[dummy]{Example}
\theoremstyle{remark}
\newtheorem{claim}[claimcounter]{Claim}
\newtheorem{step}[claimcounter]{Step}
\newtheorem{case}[casecounter]{Case}
\newtheorem{fact}{Fact}

\newtheorem*{remark}{Remark}
\newtheorem*{conventions}{Conventions}
\DeclareMathOperator*{\esssup}{ess\,sup}
\DeclareMathOperator{\Hom}{Hom}
\title{Upper Tails of Subgraph Counts in Sparse Regular Graphs}
\author{Benjamin Gunby}
\email{bgunby@g.harvard.edu}
\thanks{The author is supported by an NDSEG Graduate Fellowship}
\begin{document}
\begin{abstract}
What is the probability that a sparse $n$-vertex random $d$-regular graph $G_n^d$, $n^{1-c}<d=o(n)$ contains many more copies of a fixed graph $K$ than expected? We determine the behavior of this upper tail to within a logarithmic gap in the exponent. For most graphs $K$ (for instance, for any $K$ of average degree greater than $4$) we determine the upper tail up to a $1+o(1)$ factor in the exponent. However, we also provide an example of a graph, given by adding an edge to $K_{2,4}$, where the upper tail probability behaves differently from previously studied behavior in both the sparse random regular and sparse Erd\H{o}s-R\'{e}nyi models in this sparsity regime.
\end{abstract}
\maketitle
\section{Introduction}\label{introductionsection}
Suppose we have a random $d$-regular graph $G_n^d$ on $n$ vertices, such that each such graph is chosen with equal probability. What is the probability that the number of triangles in $G_n^d$ exceeds its expectation by a constant factor? What if triangles are replaced by, for example, copies of $K_{2,3}$?

\subsection{History of the Upper Tail Problem for $G(n,p)$}
When instead of $G_n^d$ we take the the Erd\H{o}s-R\'{e}nyi random graph $G(n,p)$, this upper tail question is well-studied. Arguments bounding the upper tail generally consist of two components. One component is to formulate a \emph{large deviation principle}, bounding the upper tail in terms of the solution to a certain \emph{variational problem}. The other component is obtaining good bounds on the solution to the variational problem. Generally, better bounds on the variational problem translate into better upper tail estimates, and improvements in the large deviation principle translate into larger ranges of parameters in which these estimates hold.

The work of Chatterjee and Varadhan \cite{CV} first introduced such a large deviation principle. This enabled them to address the case of $G(n,p)$ where $p$ is fixed and $n\to\infty$, showing for example that if $T_{n,p}$ is the number of triangles in $G(n,p)$, that
\[\displaystyle\lim_{n\to\infty}\frac{1}{n^2}\log\Pr\left[T_{n,p}\geq t\mathbb{E}(T_{n,p})\right]=-\phi(p,t),\]
for some function $\phi(p,t)$, given by the solution to a particular variational problem, that is nonzero as long as $t>1$.

More recently, there has been additional focus on the case of the sparse Erd\H{o}s-R\'{e}nyi random graph, where instead of a fixed $p$, $p=p(n)$ tends to $0$ as $n\to\infty$. An early result of Kim and Vu \cite{KV} shows that if $p\geq n^{-1}\log n$,
\[\exp\left(-\Theta\left(p^2n^2\log\frac{1}{p}\right)\right)\leq\Pr[T_{n,p}\geq (1+\delta)\mathbb{E}(T_{n,p})]\leq\exp(-\Theta(p^2n^2))\]
for any fixed $\delta>0$. Chatterjee \cite{C} and DeMarco and Kahn \cite{DK} independently eliminated this log gap, showing that
\[\Pr[T_{n,p}\geq (1+\delta)\mathbb{E}(T_{n,p})]=\exp\left(-\Theta\left(p^2n^2\log\frac{1}{p}\right)\right)\]
for all fixed $\delta>0$ and $p\geq C(\delta)n^{-1}\log n$. (DeMarco-Kahn actually showed the stronger result that one can take $C(\delta)=1$ in the earlier expression.)

Chatterjee and Dembo \cite{CDembo} managed to improve the large deviation principle so that it applies when $p\to 0$ polynomially with $n$. Lubetzky and Zhao \cite{LZ} were able to solve the resulting variational problem in the case of triangles, and were able to use the Chatterjee-Dembo result to show that
\[\Pr[T_{n,p}\geq (1+\delta)\mathbb{E}(T_{n,p})]=\exp\left(-(1+o(1))c(\delta)p^2n^2\log\frac{1}{p}\right),\]
where $c(\delta)=\min\left(\frac{1}{2}\delta^{\frac{2}{3}},\frac{1}{3}\delta\right)$, as long as $n^{-\frac{1}{42}}\log n\leq p\ll 1$.

Bhattacharya, Ganguly, Lubetzky, and Zhao \cite{BGLZ} generalized this result and were able to compute the upper tail probability $\Pr[\Hom(K,G(n,p))\geq (1+\delta)\mathbb{E}(\Hom(K,G(n,p)))]$ up to a factor of $1+o(1)$ in the exponent for any fixed graph $K$. With appropriate bounds on $p$, they were able to compute a constant $c(K,\delta)$ such that
\[\Pr[\Hom(K,G(n,p))\geq (1+\delta)\mathbb{E}(\Hom(K,G(n,p)))]=\exp\left(-(c(K,\delta)+o(1))p^{\Delta(K)}n^2\log\frac{1}{p}\right),\]
where $\Delta(K)$ is the maximum degree of $K$.

At the same time, progress has been made on the large deviations principle side, proving the upper tails results for larger ranges of $p$. Eldan \cite{E} was able to improve on the Chatterjee-Dembo large deviations argument, improving the $n^{-\frac{1}{42}}\log n\ll p\ll 1$ range for the triangle upper tail result to $n^{-\frac{1}{18}}\log n\ll p\ll 1$. More recently, Cook and Dembo \cite{CD} proved a stronger large deviations principle for all graphs $K$, which in the case of triangles extended the range further to $n^{-\frac{1}{3}}\ll p\ll 1$. Augeri \cite{Aug} was independently able to prove a more specific but stronger large deviations result, which in the case of triangles extended the range to $n^{-\frac{1}{2}}\ll p\ll 1$. Harel, Mousset, and Samotij \cite{HMS} were able to prove a slightly different type of large deviations principle, resulting in a different variational problem. Their work extended the valid parameter range for triangles further to $n^{-1}\log n\ll p\ll 1$, and successfully accounted for a transition that happens at approximately $p=n^{-\frac{1}{2}}$.  Basak and Basu \cite{BB} later extended the results of \cite{HMS} to all regular graphs $K$. However, unlike \cite{CD}, the results of \cite{Aug,HMS,BB} do not generalize to all graphs $K$.

\subsection{History of the Upper Tail Problem for $G_n^d$}
We now turn our attention to the upper tail problem on the random $d$-regular graph $G_n^d$. This problem is more delicate in many ways, as edges no longer appear independently. We will be considering the case where our graph is sparse; in general, we would like to solve the following problem.

\begin{problem}
Let $n\to\infty$ and $d=d(n)\ll n$. Let $G_n^d$ be a random $d$-regular graph on $[n]$. Compute
\[\Pr[\Hom(K,G_n^d)\geq (1+\delta)\mathbb{E}(\Hom(K,G_n^d))].\]
\end{problem}
A common construction for $G(n,p)$, that of giving a subset of vertices having high degree, breaks in this case due to the regularity of $G_n^d$. Because of this, the answers are often quite different in the $G(n,p)$ and $G_n^d$ setups.

Bhattacharya and Dembo \cite{BD} were able to compute the correct log-asymptotic of the probability in the Problem, for graphs $K$ such that the $2$-core of $K$ is regular, in a sparse range roughly of the form $n^{1-\epsilon(K)}\ll d\ll n$. (The $2$-core of $K$ is given by succesively removing all leaves from $K$ until the minimum degree of $K$ is at least $2$, and replacing $K$ by its $2$-core does not change the probability in the Problem.) In particular, \cite{BD} showed that if the $2$-core of $K$ is $\Delta$-regular, then
\[\Pr[\Hom(K,G_n^d)\geq (1+\delta)\mathbb{E}(\Hom(K,G_n^d))]=\exp\left(-(c(K,\delta)+o(1))p^{\Delta}n^2\log\frac{1}{p}\right)\]
for some nonzero constant $c(K,\delta)$ that they were able to compute.

However, \cite{BD} left open the question of what happens when the $2$-core of $K$ is not regular, only proving that $\exp\left(-\Theta\left(p^{\Delta}n^2\log\frac{1}{p}\right)\right)$ (where $p=\frac{d}{n}$ and $\Delta$ is the maximum degree of the $2$-core of $K$) is \emph{never} in fact the correct growth rate.

\subsection{New Results}
We find the correct growth rate for all graphs $K$ to within a `log gap', for $d$ in an appropriate sparse regime. In particular, we will show that with $p=\frac{d}{n}$,
\[\exp\left(-\Theta\left(n^2p^{2+\gamma(K)}\log\frac{1}{p}\right)\right)\leq\Pr\left[\Hom(K,G_n^d)\geq (1+\delta)\mathbb{E}(\Hom(K,G_n^d))\right]\leq\exp\left(-\Theta\left(n^2p^{2+\gamma(K)}\right)\right)\]
for $n^{-\epsilon(K)}\ll p\ll 1$, where $\gamma(K)$ is a certain invariant of $K$, as long as $K$ is not a forest. (See Corollary \ref{correctexponentcor}.) (If $K$ is a forest then $\Hom(K,G_n^d)$ is constant, so the probability is $0$.)

For `most' graphs $K$ (for example, for any graph $K$ of average degree greater than $4$), we will be able to do better, obtaining the correct exponent to within a $1+o(1)$ factor. For example, in the case of $K=K_{2,3}$ mentioned earlier, we will be able to prove that
\[\Pr\left[\Hom(K_{2,3},G_n^d)\geq (1+\delta)\mathbb{E}(\Hom(K_{2,3},G_n^d))\right]=\exp\left(-(\sqrt{\delta}+o(1))n^2p^{\frac{5}{2}}\log\frac{1}{p}\right),\]
as long as $p:=\frac{d}{n}$ satisfies $(n^{-1}\log n)^{\frac{2}{19}}\ll p\ll 1$.

Given these examples, and the precedent for $G(n,p)$, one might expect that $\exp\left(-\Theta\left(n^2p^{2+\gamma(K)}\log(1/p)\right)\right)$ is the correct order of growth for all $K$. Indeed, an earlier version of \cite{BD} conjectured a similar statement. However, we will show that the formula above does not generalize to all graphs $K$, by exhibiting a graph $K_0$ such that
\[\Pr\left[\Hom(K_0,G_n^d)\geq (1+\delta)\mathbb{E}(\Hom(K_0,G_n^d))\right]\neq\exp\left(-\Theta\left(n^2p^{2+\gamma(K_0)}\log\frac{1}{p}\right)\right).\]
This example differs from all known examples in either the $G(n,p)$ or the $G_n^d$ case. Generally, the construction for the lower bound of the upper tail probability is given by `planting' some specific subgraph; that is, conditioning on our random graph containing that subgraph. For example, when considering the upper tail problem for the triangle count in $G(n,p)$, one may obtain a lower bound of the correct order $\exp\left(-\Theta\left(n^2p^2\log\frac{1}{p}\right)\right)$ by noting that if $G(n,p)$ contains a clique of size $\Theta(np)$, it should on average have $\Theta(n^3p^3)$ more triangles than expected. However, in the case of $K_0$ above, the constructions do not simply arise from planting a subgraph, and we additionally must condition on a subgraph having high (but not 1) density. This demonstrates the difficulty of solving the sparse regular upper tail problem for general $K$.

\subsection{Ideas}
The solution technique for such upper tail problems generally has two major steps. First, one applies a large deviation framework, showing that the behavior of the upper tail is given by a solution to a certain variational problem. Second, one must solve this variational problem. We follow this outline in reverse order: Sections \ref{constructionssection} through \ref{K24section} will cover the appropriate variational problem's solution, and Sections \ref{ldpupperboundsection} through \ref{secondtermsection} will be dedicated to the reduction to the variational problem.

\subsubsection{Variational Problem}
Our variational problem will be, in essence, to minimize entropy over all $p$-regular graphons having enough homomorphisms from $K$. Upper bounds on the variational problem (given by constructions) will generally translate to lower bounds on the upper tail, and vice versa.

In proving our lower bound on the variational problem, we incorporate several techniques from previous works. In particular, the adaptive thresholding technique demonstrated in Section 5 of \cite{BGLZ} will be vital. Similarly to previous works such as \cite{BGLZ,BD,LZ}, we will apply a generalized form of H\"older's inequality.

However, the form of H\"older's inequality used in those works is not strong enough for our purposes, in the sense that it is not responsive to the restriction that our graphon must be regular. Our main new ideas as regards the variational problem will be a stronger generalized H\"older's inequality that is responsive to the regularity condition (Theorem \ref{graphholderthm}), and the systematic application of that inequality via edge weightings.

Another key idea in our solution to the variational problem is the use of the minimum fractional vertex cover linear program and its dual, maximum fractional matching. One principle throughout is that the upper bounds on the variational problem (given by constructions) use minimum fractional vertex cover, whereas the lower bounds use maximum fractional matching (largely as weights to use in our generalized H\"older's inequality).

\subsubsection{Bounding the Upper Tail}
The second half of our paper shows that the upper tail probability is given by the solution to the variational problem. This section has two main results (Theorems \ref{ldpupperboundthm} and \ref{ldplowerboundthm}), essentially providing an upper and lower bound on the upper tail probability based on the solution to the variational problem. The upper bound is essentially given by the large deviations argument of Cook and Dembo in \cite{CD}.

The lower bound, i.e. constructing many $d$-regular graphs with many homomorphisms from $K$, requires significantly more innovation than the upper bound, and is much more delicate than in previous work. We loosely follow Section 2.3 of \cite{BD}. However, there are several substantial complications, related to the fact that our constructions are no longer in general given by planting a certain subgraph, or equivalently the fact that the solution to the variational problem takes values substantially greater than $p$ but less than $1$. As such, we must prove a result that holds for more general graphons than those considered in \cite{BD}.

The argument of \cite{BD} roughly involves conditioning on our random graph containing the particular subgraph we are planting, and proving that the upper tail event is then almost certain. In particular, if $\mathcal{H}$ is the desired upper tail event, the argument of \cite{BD} goes by first choosing some event $\mathcal{B}$ (which in essence states that a graph contains the planted subgraph) and proving that $\mathcal{B}$ is fairly probable whereas $\mathcal{B}\cap\neg\mathcal{H}$ is highly unlikely.

As we are no longer simply planting a subgraph, finding the correct auxiliary event $\mathcal{B}$ to use is difficult. In Section \ref{ldplowerboundsection1}, we will define the event $\mathcal{A}_n^{\deg}$, which we will use for this purpose. A second difficulty will come in bounding $\Pr[\mathcal{A}_n^{\deg}\cap\neg\mathcal{H}]$, the subject of Section \ref{secondtermsection}. In \cite{BD}, after using the auxiliary event to assist in changing measures to a certain measure $\mathbb{P}_{\star}$, it is dropped entirely, but in our case, this will not be possible as the desired bound on $\mathbb{P}_{\star}(\neg\mathcal{H})$ (c.f. (2.46) of \cite{BD}) is not even true. Thus we must extract additional use out of our event $\mathcal{A}_n^{\deg}$ before dropping it.

We state some conventions that we will use throughout this paper.
\begin{conventions}
We use big-$O$ notation (including $O(\cdot)$, $o(\cdot)$, $\Omega(\cdot)$, $\omega(\cdot)$, $\Theta(\cdot)$) in the usual way. All uses of this notation will apply as $n\to\infty$, or if there is no $n$ appearing (as in nearly all of Sections \ref{constructionssection} through \ref{K24section}) as $p\to 0$. We will also use $f\lesssim g$ to mean $f=O(g)$ and $f\ll g$ to mean $f=o(g)$.

We will consider only graphs with no isolated vertices. As such, subgraphs of a graph $K$ correspond to subsets of the edge set $E(K)$, and we will use these interchangably throughout.

Whenever we consider the random $d$-regular graph $G_n^d$, we will assume $dn$ is even.
\end{conventions}
\section{Main Theorems}\label{mainresultssection}
Before stating our main result, we make several definitions.

For $d<n\in\mathbb{Z}^+$ with $dn$ even, let $G_n^d$ be a random graph given by selecting each $d$-regular-graph on $[n]$ with equal probability.

For a graph $K$, let $E(K)$ be the set of edges of $K$, and $V(K)$ be the set of non-isolated vertices of $K$. Let $e(K):=|E(K)|$ and $v(K):=|V(K)|$.

Call a tuple $(c_v)_{v\in V(K)}$ a \emph{fractional vertex cover} if $c_v\geq 0$ for all $v\in V(K)$ and $c_v+c_w\geq 1$ for all $vw\in E(K)$. Let $c(K)$ be the minimum value of $\sum_{v\in V(K)}c_v$ over all fractional vertex covers $(c_v)$, and call a fractional vertex cover \emph{minimum} if $\sum_{v\in V(K)}c_v=c(K)$.

\begin{definition}
Let
\[\gamma(K)=\displaystyle\max_{H\subseteq K}\frac{e(H)-v(H)}{c(H)}.\]
Call a subgraph $H\subseteq K$ \emph{contributing} if $H$ has minimum degree at least $2$ and $e(H)-v(H)=c(H)\gamma(K)$.
\end{definition}
For two graphs $K,G$, we define the homomorphism count $\Hom(K,G)$ to be the number of functions $V(K)\to V(G)$ such that every edge of $K$ is mapped into an edge of $G$.
\begin{remark}
Homomorphism count is closely related to subgraph count; notice that if $G$ is a graph on $n$ vertices, $\Hom(K,G)=|\text{Aut}(K)|\cdot(\#\text{ copies of }K\text{ in }G)+O(n^{v(K)-1})$, as there are $|\text{Aut}(K)|$ homomorphisms from $K$ to $G$ for every time $K$ appears as a subgraph of $G$, and this counts all homomorphisms except the $O(n^{v(K)-1})$ that are not injective on the vertices of $K$.
\end{remark}
\begin{definition}\label{validdef}
For a graph $H$, call a subset $A\subseteq V(H)$ \emph{valid} if there is a minimum fractional vertex cover $(c_v)_{v\in V(H)}$ such that $c_v=1$ if $v\in A$ and $c_v\in\left\{0,\frac{1}{2}\right\}$ if $v\in V(H)\backslash A$.

For a graph $K$, we also define a bivariate polynomial $P_K$, given by
\[P_K(z,w):=\displaystyle\sum_{H\subseteq K\text{ contributing}}\displaystyle\sum_{A\subseteq V(H)\text{ valid}}z^{|A|}w^{c(H)-|A|}.\]

We further define
\[\rho(K,\delta):=\displaystyle\min_{\substack{z,w\geq 0 \\ P_K(z,w)\geq 1+\delta}}\left(z+\frac{w}{2}\right).\]
\end{definition}
Our main results are Theorems \ref{correctexponent}, \ref{correctconstant}, and \ref{K24thm}. The first of these theorems bounds the desired upper tail probability to within a logarithmic factor in the exponent.
\begin{thm}\label{correctexponent}
Let $K$ be any nonforest graph whose $2$-core is not a disjoint union of cycles, and fix $\delta>0$. If $d=d(n)$ with $dn$ even and $p:=\frac{d}{n}$ satisfies $(n^{-1}\log n)^{\frac{1}{2e(K)-2-\gamma(K)}}\ll p\ll 1$, then
\begin{align*}
n^2p^{2+\gamma(K)} & \lesssim-\log\Pr\left[\Hom(K,G_n^d)\geq (1+\delta)p^{e(K)}n^{v(K)}\right] \\ & \leq (1+o(1))\rho(K,\delta)n^2p^{2+\gamma(K)}\log\frac{1}{p}
\end{align*}
as $n\to\infty$.
\end{thm}
\begin{remark}
The expression
\[\displaystyle\min_{\substack{z,w\geq 0 \\ P_K(z,w)\geq 1+\delta}}\left(z+\frac{w}{2}\right)\]
in the definition of $\rho(K,\delta)$ will come from our construction in Section \ref{constructionssection}. In essence, we will plant a subgraph consisting of a `hub' $S\subseteq [n]$ of vertices of $G_n^d$ all of which are connected to all the vertices in $[d]$, as well as a clique on some vertex set $T$. The size of the hub $S$ will be governed by the parameter $z$, and the size of the clique $T$ will be governed by the parameter $w$, normalized so that the clique and the hub are of the appropriate order in size.

The probability that $G_n^d$ contains this graph will be exponential in a quantity proportional to $z+\frac{w}{2}$. Generally, if $G_n^d$ contains this subgraph, we will be able to find $P_K(z,w)n^{v(K)}p^{e(K)}$ homomorphisms from $K$. So we should have enough homomorphisms with at least probability proportional to
\[\min_{\substack{z,w\geq 0 \\ P_K(z,w)\geq 1+\delta}}\left(z+\frac{w}{2}\right).\]

The valid subsets in the definition of $P_K$ will correspond to the sets of vertices that we send into the hub when counting homomorphisms from $K$ to $G_n^d$.
\end{remark}
In the case where the $2$-core of $K$ \emph{is} a disjoint union of cycles, the method of Bhattacharya-Dembo \cite{BD} easily extends to the following result.
\begin{thm}[Essentially as in \cite{BD}]\label{cycleunionthm}
Let $K$ be any nonforest graph whose $2$-core is a disjoint union of cycles of length $i_1,\ldots,i_{\ell}$, and fix $\delta>0$. If $d=d(n)$ with $dn$ even and $p:=\frac{d}{n}$ satisfies $n^{-\frac{1}{3}}\ll p\ll 1$, then
\[-\log\Pr\left[\Hom(K,G_n^d)\geq (1+\delta)p^{e(K)}n^{v(K)}\right]=\left(\frac{c(i_1,\ldots,i_{\ell};\delta)}{2}+o(1)\right)n^2p^2\log\frac{1}{p},\]
where $c(i_1,\ldots,i_{\ell};\delta)$ is the unique positive value of $c$ such that $\displaystyle\prod_{j=1}^{\ell}(1+\lfloor c\rfloor+\{c\}^{i_j/2})=1+\delta$. (Here $\lfloor c\rfloor$ and $\{c\}$ denote the integer part and fractional part of $c$, respectively.)
\end{thm}
For completeness, we include a proof of Theorem \ref{cycleunionthm} in Section \ref{correctexponentcorsection}. The previous two theorems easily imply our desired log gap.
\begin{cor}\label{correctexponentcor}
Let $K$ be any nonforest graph and fix $\delta>0$. If $d=d(n)$ with $dn$ even and $p:=\frac{d}{n}$ satisfies $(n^{-1}\log n)^{\frac{1}{2e(K)-2-\gamma(K)}}\ll p\ll 1$, then
\[n^2p^{2+\gamma(K)}\lesssim -\log\Pr\left[\Hom(K,G_n^d)\geq (1+\delta)p^{e(K)}n^{v(K)}\right]\lesssim n^2p^{2+\gamma(K)}\log\frac{1}{p}\]
as $n\to\infty$.
\end{cor}
\begin{remark}
Note that the upper tail question is trivial when $K$ is a forest, as a forest has the same number of homomorphisms into every regular graph. Thus Corollary \ref{correctexponentcor} solves the upper tail problem to within a log gap for all graphs.
\end{remark}
For most graphs, we obtain an improved result that bounds the probability to within a $1+o(1)$ factor in the exponent. First we must make some additional definitions.
\begin{definition}
For a graph $H$, call a tuple $(w_e)_{e\in E(H)}$ a \emph{fractional matching} if $w_e\geq 0$ for all $e\in E(H)$ and $\displaystyle\sum_{e\ni v}w_e\leq 1$ for all $v\in V(H)$.

Call a tuple $(w_e)_{e\in E(H)}$ a \emph{fractional edge cover} if $w_e\geq 0$ for all $e\in E(H)$ and $\displaystyle\sum_{e\ni v}w_e\geq 1$ for all $v\in V(H)$.

Call a tuple $(w_e)_{e\in E(H)}$ a \emph{fractional perfect matching} if it is both a fractional matching and a fractional edge cover; that is, if $w_e\geq 0$ for all $e\in E(H)$ and $\displaystyle\sum_{e\ni v}w_e=1$ for all $v\in V(H)$.

Call a fractional matching \emph{maximum} if $\displaystyle\sum_{e\in E(H)}w_e=c(H)$. Call a fractional edge cover \emph{minimum} if $\displaystyle\sum_{e\in E(H)}w_e=v(H)-c(H)$.

Call an edge $e_0\in H$ \emph{bad} if for every maximum fractional matching $(w_e)_{e\in E(H)}$, $w_{e_0}=1$.
\end{definition}
\begin{remark}
It will follow from linear programming duality (as we will prove in Lemma \ref{edgeweightslemma}) that $c(H)$ is indeed the maximum of $\sum w_e$ over all fractional matchings $(w_e)_{e\in E(H)}$, and $v(H)-c(H)$ is the minimum over all fractional edge covers, justifying our terminology.
\end{remark}
\begin{thm}\label{correctconstant}
Let $K$ be a fixed nonforest graph none of whose contributing subgraphs have bad edges, and whose $2$-core is not a disjoint union of cycles. Fix $\delta>0$. If $d=d(n)$ with $dn$ even and $p:=\frac{d}{n}$ satisfies $(n^{-1}\log n)^{\frac{1}{2e(K)-2-\gamma(K)}}\ll p\ll 1$, then
\[-\log\Pr\left[\Hom(K,G_n^d)\geq (1+\delta)p^{e(K)}n^{v(K)}\right]=(1+o(1))\rho(K,\delta)n^2p^{2+\gamma(K)}\log\frac{1}{p}\]
as $n\to\infty$.
\end{thm}
\begin{remark}
Although it is not obvious, the log probability in Theorem \ref{correctconstant} is quite similar to the one found in \cite{BGLZ} for Erd\H{o}s-R\'{e}nyi random graphs. Suppose in the definitions above we replace the expression $2+\gamma(K)=\displaystyle\max_{H\subseteq K}\frac{e(H)+2c(H)-v(H)}{c(H)}$ with $\displaystyle\max_{H\subseteq K}\frac{e(H)}{c(H)}$, change the definition of contributing subgraphs accordingly to instead maximize $\frac{e(H)}{v(H)}$, and redefine $P_K(z,w)$ and $\rho(K,\delta)$ correspondingly. Then $\displaystyle\max_{H\subseteq K}\frac{e(H)}{c(H)}$ turns out to simply be the maximum degree $\Delta(K)$, which is exactly the exponent in Corollary 1.6 of \cite{BGLZ}. Furthermore, if one works out the new definition of $\rho(K,\delta)$, it turns out to be exactly the constant in Corollary 1.6 of \cite{BGLZ}, so the expression $(1+o(1))\rho(K,\delta)n^2p^{2+\gamma(K)}\log(1/p)$ in fact still gives the correct log-probability (albeit with modified definitions). Indeed, there is a reason for this: we are able to modify our application of Hölder's inequality so that our degree condition will give us an extra factor of $p^{v(H)-2c(H)}$ (see for example Corollary \ref{simpleholdercor}). In fact, dropping the degree condition and carrying the proof of Theorem \ref{varcorrectconstant} in this paper through without that term, one would essentially reprove Theorem 1.5 of \cite{BGLZ}. This is because the condition `no contributing subgraphs have bad edges' is true for \emph{all graphs} (except the single-edge graph) under the modified definition of contributing subgraph. This highlights a way in which the random $d$-regular case presents new fundamental difficulties not present in the Erd\H{o}s-R\'{e}nyi case.
\end{remark}
We claimed above that Theorem \ref{correctconstant} applies to `most' graphs; the proposition below justifies this claim.
\begin{prop}\label{mostgraphsprop}
If $\gamma(K)>2$, then the conditions on $K$ in Theorem \ref{correctconstant} hold. In particular, if $K$ or any subgraph of $K$ has average degree greater than $4$, then the conditions on $K$ in Theorem \ref{correctconstant} hold.

The conditions in Theorem \ref{correctconstant} also hold for any nonforest $K$ with $v(K)\leq 5$.
\end{prop}
One might expect, given Theorem \ref{correctconstant}, that
\[-\log\Pr\left[\Hom(K,G_n^d)\geq (1+\delta)p^{e(K)}n^{v(K)}\right]=\Theta\left(n^2p^{2+\gamma(K)}\log\frac{1}{p}\right)\]
for all graphs $K$ and $d$ in the appropriate sparsity regime. Indeed, for the Erd\H{o}s-R\'{e}nyi graph $G(n,p)$, the result
\[-\log\Pr\left[\Hom(K,G(n,p))\geq (1+\delta)p^{e(K)}n^{v(K)}\right]=\Theta\left(n^2p^{\Delta(K)}\log\frac{1}{p}\right)\]
holds for appropriate values of $p$, per Corollary 1.6 of \cite{BGLZ}. However, in our case this does not turn out to be true, as the following theorem shows.
\begin{figure}
\scalebox{0.5}{
\begin{tikzpicture}
\GraphInit[vstyle=simple]
\grEmptyPath[Math,prefix=w,RA=2,RS=0]{4}
\begin{scope}[xshift=2cm]
\grEmptyPath[Math,prefix=v,RA=2,RS=3]{2}
\end{scope}
\EdgeInGraphSeq{w}{1}{1}
\EdgeFromOneToAll{v}{w}{0}{4}
\EdgeFromOneToAll{v}{w}{1}{4}
\end{tikzpicture}
}
\caption{The graph $K_0$}
\label{K24plusanedgefigure}
\end{figure}
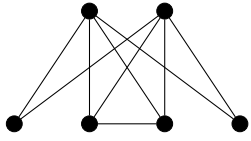

\begin{thm}\label{K24thm}
Let $K_0$ be the graph given by adding an edge to $K_{2,4}$ on the side with four vertices, as in Figure \ref{K24plusanedgefigure}. If $d=d(n)$ with $dn$ even and $p:=\frac{d}{n}$ satisfies $(n^{-1}\log n)^{\frac{1}{15}}\ll p\ll 1$, then
\[-\log\Pr\left[\Hom(K_0,G_n^d)\geq (1+\delta)p^9n^6\right]=(1+o(1))\frac{(18\delta)^{\frac{1}{3}}}{2}n^2p^3\left(\log\frac{1}{p}\right)^{\frac{2}{3}}\left(\log\log\frac{1}{p}\right)^{\frac{1}{3}}\]
as $n\to\infty$.
\end{thm}

Notice $K_0$ indeed has a bad edge--namely, the edge added to $K_{2,4}$ to obtain $K_0$.

\begin{remark}
The behavior of the upper tail in the case of $K_0$ is fundamentally different to in other known cases in the sparsity regime $n^{-\epsilon}\ll p\ll 1$. In previous work, the prototypical graph containing many copies of some $K$ is given by guaranteeing the existence of (`planting') some large substructure (such as a clique) that will force extra copies of $K$. However, the prototypical graph containing many copies of $K_0$ is given by simultaneously planting a subgraph and uniformly raising the density on a different subgraph.

This logarithmic gap between the true growth rate of the upper tail and its `expected' growth rate is similar to results of \v{S}ileikis and Warnke \cite{SW}, who found a similar logarithmic gap in the Erd\H{o}s-R\'{e}nyi case when the sparsity is very close to the appearance threshold of the graph $K$.
\end{remark}

Our method of proving Theorems \ref{correctexponent}, \ref{correctconstant}, and \ref{K24thm} will be to reduce to a variational problem. Specifically, we will (approximately) show that the upper tail probability is determined by the minimum-entropy graphon $W$ such that $\Hom(K,W)\geq 1+\delta$.

\begin{definition}\label{phidef}
For $x\in [0,1]$, let $I_p(x):=x\log\frac{x}{p}+(1-x)\log\frac{1-x}{1-p}$ be the $p$-entropy of $x$. (Define $I_p(0)$ and $I_p(1)$ to be the appropriate limiting values $\log(1/p)$ and $\log(1/(1-p))$, respectively.)

If $W$ is a graphon, let $I_p(W):=\displaystyle\int_0^1\displaystyle\int_0^1 I_p(W(x,y)) dx dy$ be the total entropy of $W$.

Say that a graphon $W$ is \emph{$p$-regular} if $\displaystyle\int_0^1 W(x_0,y) dy=p$ for all $x_0\in [0,1]$.

With $K$ any graph and $W$ any symmetric measurable from $[0,1]^2\to\mathbb{R}$ (e.g. a graphon), define
\[\Hom(K,W)=\displaystyle\int_{[0,1]^{v(K)}}\displaystyle\prod_{uv\in E(K)}W(x_u,x_v)\displaystyle\prod_{v\in V(K)}dx_v.\]

Further letting $t>1$, $d\leq n\in\mathbb{Z}^+$, $p=\frac{d}{n}$, we define
\[\Phi_n^d(K,t):=\frac{n^2}{2}\displaystyle\inf_{\substack{W\text{ }p\text{-regular} \\ \Hom(K,W)\geq tp^{e(K)}}}I_p(W),\]
where the minimum is taken over all graphons $W$ satisfying the desired properties.
\end{definition}
\begin{definition}\label{deltastardef}
For a graph $K$, let $\Delta_*(K)=\frac{1}{2}\displaystyle\max_{vw\in E(K)}(\deg_K(v)+\deg_K(W))$.
\end{definition}
The following two theorems (mostly) reduce the upper tail problem to that of determining $\Phi_n^d(K,t)$. The first result essentially follows from the argument of Cook-Dembo \cite{CD}.
\begin{thm}\label{ldpupperboundthm}
Let $K$ be any nonforest graph, and fix $t>1$. If $d=d(n)$ with $dn$ even and $p:=\frac{d}{n}$ satisfies $(n^{-1}\log n)^{-\frac{1}{2\Delta_{\star}(K)}}\ll p\ll 1$, then
\[-\log\left(\Pr\left[\Hom(K,G_n^d)\geq (t+o(1))p^{e(K)}n^{v(K)}\right]\right)\geq (1-o(1))\Phi_n^d(K,t)\]
as $n\to\infty$.
\end{thm}
To give a matching upper bound for the left side of Theorem \ref{ldpupperboundthm}, notice that $\Phi_n^d(K,t)$ is a minimum, so ideally we would like to have the upper bound $\left(\frac{1}{2}+o(1)\right)I_p(W)n^2$ for all $p$-regular $W$ satisfying $\Hom(K,W)\geq tp^{e(K)}$. This proves to be difficult in general, but it in fact suffices to prove this upper bound when $W$ satisfies several nice properties (which will be stated in Conditions \ref{lowerboundconditions}) that we expect the solution to the variational problem to have.
\begin{thm}\label{ldplowerboundthm}
Suppose $n^{-1}\log\log n\ll p\ll 1$, and let $W=W(n)$ be a block graphon on some constant $k=k(n)=O(1)$ number of blocks satisfying Conditions \ref{lowerboundconditions}. Then
\[-\log\left(\Pr\left[\Hom(K,G_n^d)\geq (1-o(1))\Hom(K,W)n^{v(K)}\right]\right)\leq \left(\frac{1}{2}+o(1)\right)I_p(W)n^2.\]
as $n\to\infty$.
\end{thm}
If $W$ is an approximate solution to the variational problem, such that $\Phi_n^d(K,t)=\left(\frac{1}{2}+o(1)\right)n^2 I_p(W)$, then the upper and lower bounds in Theorems \ref{ldpupperboundthm} and \ref{ldplowerboundthm} coincide. Thus these results reduce the upper tail problem to computing $\Phi_n^d(K,1+\delta)$, or in other words minimizing $I_p(W)$ over all graphons $W$ satisfying $\Hom(K,W)\geq (1+\delta)p^{e(K)}$ (and checking that the optimizer has the correct form). Thus together with Theorems \ref{ldpupperboundthm} and \ref{ldplowerboundthm}, the next three results (which bound the solutions to this variational problem) will easily respectively show our three main theorems.
\begin{thm}\label{varcorrectexponent}
Take $\delta>0$ and let $p\to 0$. Fix a nonforest graph $K$ whose $2$-core is not a disjoint union of cycles. Then
\[p^{2+\gamma(K)}\lesssim\displaystyle\inf_{\substack{\Hom(K,W)\geq (1+\delta)p^{e(K)} \\ W\text{ }p\text{-regular}}}I_p(W)\leq (2+o(1))\rho(K,\delta)p^{2+\gamma(K)}\log\frac{1}{p}.\]
Furthermore, the upper bound is attained by a graphon satisfying the conditions of Theorem \ref{ldplowerboundthm} for any $n$ satisfying $(n^{-1}\log n)^{\frac{1}{2e(K)-2-\gamma(K)}}\ll p\ll 1$.
\end{thm}
\begin{thm}\label{varcorrectconstant}
Take $\delta>0$ and let $p\to 0$. If the $2$-core of $K$ is not a disjoint union of cycles and none of the contributing subgraphs of $K$ have bad edges, then
\[\displaystyle\inf_{\substack{\Hom(K,W)\geq (1+\delta)p^{e(K)} \\ W\text{ }p\text{-regular}}}I_p(W)=(2+o(1))\rho(K,\delta)p^{2+\gamma(K)}\log\frac{1}{p}.\]
Furthermore, the upper bound is attained by a graphon satisfying the conditions of Theorem \ref{ldplowerboundthm} for any $n$ satisfying $(n^{-1}\log n)^{\frac{1}{2e(K)-2-\gamma(K)}}\ll p\ll 1$.
\end{thm}
\begin{thm}\label{K24correctconstant}
Take $\delta>0$ and let $p\to 0$. If $K_0$ is the graph from Figure \ref{K24plusanedgefigure}, then
\[\displaystyle\inf_{\substack{\Hom(K_0,W)\geq (1+\delta)p^{9} \\ W\text{ }p\text{-regular}}}I_p(W)=(1+o(1))(18\delta)^{\frac{1}{3}}p^3\left(\log\frac{1}{p}\right)^{\frac{2}{3}}\left(\log\log\frac{1}{p}\right)^{\frac{1}{3}}.\]
Furthermore, the upper bound is attained by a graphon satisfying the conditions of Theorem \ref{ldplowerboundthm} for any $n$ satisfying $(n^{-1}\log n)^{\frac{1}{15}}\ll p\ll 1$.
\end{thm}
\begin{remark}
In the language of the variational problem, the different behavior of the graph $K_0$ is due to the fact that the solution to the variational problem takes values less than $1$ but substantially higher than $p$. Informally (in a sense that we will describe in more detail in Section \ref{K24section}), the bad edge of $K_0$ forces us to `smooth' some of the values of $W$ inwards rather than outwards.
\end{remark}
The remainder of the paper will be structured as follows. In Section \ref{constructionssection}, we will discuss the constructions for the upper bounds of Theorems \ref{varcorrectexponent}, \ref{varcorrectconstant}, and \ref{K24correctconstant}, and prove those upper bounds. In Section \ref{preliminariessection}, we will prove some generally useful statements that we will reuse over the course of proving the lower bounds of Theorems \ref{varcorrectexponent}, \ref{varcorrectconstant}, and \ref{K24correctconstant}.

The next three sections primarily deal with the lower bound of Theorem \ref{varcorrectexponent}. In Section \ref{examplessection}, we will prove the lower bound of Theorem \ref{varcorrectexponent} for two illustrative examples. Section \ref{holdersection} will introduce and prove a generalized H\"older's inequality that will be a main technical engine in proving the lower bounds of Theorems \ref{varcorrectexponent} and \ref{varcorrectconstant}, and in Section \ref{simpleholderproofsection} we will use this to complete the proof of the lower bound of Theorem \ref{varcorrectexponent}.

Sections \ref{K23section} through \ref{constantsection2} deal with the lower bound of Theorem \ref{varcorrectconstant}. Section \ref{K23section} will prove this lower bound for the example graph $K=K_{2,3}$, and Sections \ref{constantsection1} and \ref{constantsection2} will contain the proof of the lower bound in general.

Section \ref{K24section} is dedicated to the proof of the lower bound of Theorem \ref{K24correctconstant}.

Section \ref{ldpupperboundsection} will describe how to modify the argument of \cite{CD} to prove Theorem \ref{ldpupperboundthm}.

The next four sections, namely sections \ref{ldplowerboundsection1}, \ref{ldplowerboundsection2}, \ref{firsttermsection}, and \ref{secondtermsection}, together prove Theorem \ref{ldplowerboundthm}.

The following three sections clean up leftover loose ends. Section \ref{thesegraphonsworksection} proves the promised statement that the constructions for Theorems \ref{varcorrectexponent}, \ref{varcorrectconstant}, and \ref{K24correctconstant} indeed satisfy the conditions of Theorem \ref{ldplowerboundthm} under the appropriate settings. Section \ref{correctexponentcorsection} will describe how Theorem \ref{cycleunionthm} follows from the methods of \cite{BD}, and will deduce Corollary \ref{correctexponentcor} from Theorems \ref{correctexponent} and \ref{cycleunionthm}. Section \ref{mostgraphssection} contains the proof of Proposition \ref{mostgraphsprop}.

Finally, Section \ref{maintheoremsection} proves the main Theorems \ref{correctexponent}, \ref{correctconstant}, and \ref{K24thm}.
\section{Constructions and Upper Bounds}\label{constructionssection}
The graphons that solve the variational problems given in Theorems \ref{varcorrectexponent}, \ref{varcorrectconstant}, and \ref{K24correctconstant}, and thus provide a template for graphs containing many copies of $K$, will be of several forms.
\begin{figure}
\begin{tikzpicture}[xscale=1.5,yscale=1.5]
\draw (0,4) -- (4,4) -- (4,0) -- (0,0) node[anchor=north east]{$1$} -- (0,4) node[anchor=south east]{$0$}; \draw (0.5,0) -- (0.5,4); \draw (4,3.5) -- (0,3.5) node[anchor=east]{$zp^{1+\gamma}$}; \draw (4,2.5) -- (0,2.5) node[anchor=east]{$p$}; \draw (1.5,0) -- (1.5,4);

\node at (0.25,3.75) {$1$}; \node at (0.25,3) {$1$}; \node at (0.25,1.25) {$0$}; \node at (1,3.75) {$1$}; \node at (1,3) {$p$}; \node[align=center] at (1,1.25) {$p+$ \\ $O(p^{1+\gamma})$}; \node at (2.75,3.75) {$0$}; \node at (2.75,3) {$p+O(p^{1+\gamma})$}; \node at (2.75,1.25) {$p+O(p^{2+\gamma})$};
\end{tikzpicture}
\hspace{12pt}
\begin{tikzpicture}[xscale=1.5,yscale=1.5]
\draw (0,4) -- (4,4) -- (4,0) -- (0,0) node[anchor=north east]{$1$} -- (0,4) node[anchor=south east]{$0$}; \draw (1.1,0) -- (1.1,4); \draw (4,2.9) -- (0,2.9) node[anchor=east]{$\sqrt{w}p^{1+\frac{\gamma}{2}}$}; 

\node at (0.55,3.45){$1$}; \node[align=center] at (0.55,1.45){$p+$ \\ $O(p^{1+\frac{\gamma}{2}})$}; \node at (2.55,3.45){$p+O(p^{1+\frac{\gamma}{2}})$}; \node at (2.55,1.45){$p+O(p^{2+\gamma})$};
\end{tikzpicture}

\begin{tikzpicture}[xscale=2.5, yscale=2.5]
\draw (0,4) -- (4,4) -- (4,0) -- (0,0) node[anchor=north east]{$1$} -- (0,4) node[anchor=south east]{$0$}; \draw (0.3,0) -- (0.3,4); \draw (4,3.7) -- (0,3.7) node[anchor=east]{$zp^{1+\gamma}$}; \draw (0.9,0) -- (0.9,4); \draw (4,3.1) -- (0,3.1) node[anchor=east]{$\sqrt{w}p^{1+\frac{\gamma}{2}}$}; \draw (4,2.3) -- (0,2.3) node[anchor=east]{$p$}; \draw (1.7,0) -- (1.7,4);

\node at (0.15,3.85) {$1$}; \node at (0.15,3.4) {$1$}; \node at (0.15,1.15) {$0$}; \node at (0.6,3.85) {$1$}; \node at (0.6,3.4) {$1$}; \node[align=center] at (0.6,1.15) {$p+$ \\ $O(p^{1+\frac{\gamma}{2}})$}; \node at (2.85,3.85) {$0$}; \node at (2.85,3.4) {$p+O(p^{1+\frac{\gamma}{2}})$}; \node at (2.85,1.15) {$p+O(p^{2+\gamma})$}; \node at (0.15,2.7) {$1$}; \node at (0.6,2.7) {$p$}; \node at (1.3,2.7) {$p$}; \node at (2.85,2.7) {$p+O(p^{1+\gamma})$}; \node at (1.3,3.85) {$1$}; \node at (1.3,3.4) {$p$}; \node[align=center] at (1.3,1.15) {$p+$ \\ $O(p^{1+\gamma})$};
\end{tikzpicture}
\caption{Constructions for the upper bounds of Theorems \ref{varcorrectexponent} and \ref{varcorrectconstant}, with $\gamma=\gamma(K)$. Planting respectively a modified hub, a clique, and both a modified hub and a clique. Note that the first two constructions can be considered specializations of the third by setting respectively $w=0$ or $z=0$. The entries that are only specified as $p+O(\cdot)$ are uniquely determined by the condition that the graphon be $p$-regular, and it is easy to check that the asymptotics are satisfied.}\label{constantgraphonfigure}
\end{figure}

\begin{figure}
\begin{tikzpicture}[xscale=1.5,yscale=1.5]
\draw (0,4) -- (4,4) -- (4,0) -- (0,0) node[anchor=north east]{$1$} -- (0,4) node[anchor=south east]{$0$}; \draw (0.5,0) -- (0.5,4); \draw (4,3.5) -- (0,3.5) node[anchor=east]{$a(p)$}; \draw (4,2.5) -- (0,2.5) node[anchor=east]{$p$}; \draw (1.5,0) -- (1.5,4);

\node at (0.25,3.75) {$1$}; \node at (0.25,3) {$1$}; \node at (0.25,1.25) {$0$}; \node at (1,3.75) {$1$}; \node at (1,3) {$b(p)$}; \node[align=center] at (1,1.25) {$p+$ \\ $O(pb(p))$}; \node at (2.75,3.75) {$0$}; \node at (2.75,3) {$p+O(pb(p))$}; \node at (2.75,1.25) {$p+O(p^2b(p))$};
\end{tikzpicture}
\caption{Construction for the upper bound of Theorem \ref{K24correctconstant}. Asymptotics hold as long as $a(p)\lesssim pb(p)$. In addition to planting a modified hub, part of the graphon takes on some value $b(p)$ with $p\ll b(p)\ll 1$.}\label{K24graphonfigure}
\end{figure}
These forms are given by the graphons appearing in Figures \ref{constantgraphonfigure} and \ref{K24graphonfigure}. Notice that in Figure \ref{constantgraphonfigure}, the first graphon has a `partial hub' of size $zp^{1+\gamma}$, the second has a `clique' of size $\sqrt{w}p^{1+\frac{\gamma}{2}}$, and the third has both.

In the graphon in Figure \ref{K24graphonfigure}, we again have a partial hub, but instead of a clique we have instead raised the value of the upper left $p$ by $p$ box to a value smaller than $1$ but significantly larger than $p$. This unusual behavior is the reason that the graph $K_0$ behaves differently than the graphs satisfying the conditions of Theorem \ref{varcorrectconstant}.

Some motivation for why this is the solution for $K_0$ may be had as follows. Let us take for granted that in the typical homomorphism from $K_0$ to $W$ we will send the two vertices of degree four into some partial hub $B$ such that $W=1$ on $B\times [0,p]$. Then all other vertices should be sent into $[0,p]$. The number of homomorphisms is then directly proportional to the average value of $W$ on $[0,p]\times [0,p]$. To maximize this while minimizing entropy, $W$ should be taken to be constant (or as close as possible) on $[0,p]\times [0,p]$.

As we would like, it turns out that these graphons satisfy the conditions of Theorem \ref{ldplowerboundthm} (indeed, most conditions of that theorem are chosen carefully to include these graphons as solution).

\begin{prop}\label{thesegraphonsworkprop}
The graphons in Figure \ref{constantgraphonfigure} satisfy the conditions of Theorem \ref{ldplowerboundthm} with any nonforest graph $K$, taking $\gamma:=\gamma(K)$, as long as $\gamma>0$ and $(n^{-1}\log n)^{\frac{1}{2e(K)-2-\gamma(K)}}\ll p\ll 1$, as does the graphon in Figure \ref{K24graphonfigure} with the graph $K_0$ from Figure \ref{K24plusanedgefigure} with $a(p)=\Theta\left(p^2\left(\log\frac{1}{p}\right)^{-\frac{1}{3}}\left(\log\log\frac{1}{p}\right)^{\frac{1}{3}}\right)$ and $b(p)=\Theta\left(p\left(\log\frac{1}{p}\right)^{\frac{2}{3}}\left(\log\log\frac{1}{p}\right)^{-\frac{2}{3}}\right)$ as long as $(n^{-1}\log n)^{\frac{1}{15}}\ll p\ll 1$.
\end{prop}
We will defer the proof for later, after we fully state the conditions of Theorem \ref{ldplowerboundthm}.

The upper bound of Theorem \ref{varcorrectexponent} (and thus Theorem \ref{varcorrectconstant}, which has the same upper bound) will easily follow from the following two lemmas.

\begin{lemma}\label{constructionenoughhoms}
Let $K$ be a nonforest graph with $2$-core not a disjoint union of cycles. With $\gamma=\gamma(K)$, let $W_0^{z,w}$ be the third graphon from Figure \ref{constantgraphonfigure}. Then
\[\Hom(K,W_0^{z,w})\geq (1-o(1))P_K(z,w) p^{e(K)}.\]
\end{lemma}

\begin{lemma}\label{constructionlowentropy}
Under the conditions of the previous lemma,
\[I_p(K,W_0^{z,w})=(1+o(1))(2z+w)p^{2+\gamma(K)}\log\frac{1}{p}.\]
\end{lemma}

\begin{proof}[Proof of Lemma \ref{constructionenoughhoms}]
Note that $\gamma=\gamma(K)>0$, as if $H\subseteq K$ is the $2$-core of $K$, then $\delta(H)\geq 2$ and $H$ is not a disjoint union of cycles, so $e(H)>v(H)$. Thus $W_0^{z,w}$ indeed takes values in $[0,1]$, as $p^{1+\frac{\gamma}{2}}=o(p)$.

Given the definition of $P_K(z,w)$, it suffices to, for every contributing subgraph $H\subseteq K$ and every valid $A\subseteq V(H)$, to find $(1-o(1))z^{|A|}w^{c(H)-|A|}p^{e(H)}$ homorphisms from $K$ to $W_0^{z,w}$, and then to show that these classes of homomorphisms are distinct.

Take a contributing subgraph $H\subseteq K$ and valid $A\subseteq V(H)$. Since $A$ is valid, there is some $B,C\subseteq V(H)$ such that $A\cup B\cup C$ is a partition of $V(H)$ and
\[c_v=\begin{cases}1 & v\in A \\ \frac{1}{2} & v\in B \\ 0 & v\in C\end{cases}\]
is a minimum fractional vertex cover of $H$.

Let $T=N_K(A)\backslash V(H)$. We count the homomorphisms that send the vertices in $A$ into $[0,zp^{1+\gamma}]$, the vertices in $B$ into $(zp^{1+\gamma},\sqrt{w}p^{1+\frac{\gamma}{2}}]$, the vertices in $C\cup T$ into $(\sqrt{w}p^{1+\frac{\gamma}{2}},p]$, and the vertices of $V(K)\backslash (V(H)\cup T)$ into $(p,1]$.

We show two useful facts.

\begin{fact}\label{contributingfact1}
For all $v_0\in T$, $|N_K(v_0)\cup A|=1$.
\end{fact}

\begin{fact}\label{contributingfact2}
The set of edges sent by our homomorphisms above into blocks of $W_0^{z,w}$ with value $1$ is exactly $E(H)\cup E(A,T)$.
\end{fact}

\begin{proof}[Proof of Fact \ref{contributingfact1}]
Since $v_0\in T$, $v_0$ is adjacent to at least one element of $A$. Suppose for the sake of contradiction that $v_0$ is adjacent to $w,w'\in A$. Let $H'$ be a graph with $V(H')=V(H)\cup\{v_0\}$ and $E(H')=E(H)\cup\{v_0w,v_0w'\}$. Then $c(H')=c(H)$, as we may extend our minimum fractional matching $(c_v)$ to $H'$ by weighting $c_{v_0}=0$. Therefore, $\frac{e(H')-v(H')}{c(H')}>\frac{e(H)-v(H)}{c(H)}$, contradicting the fact that $H$ is contributing. So $v_0$ is adjacent to exactly one element of $A$.
\end{proof}

\begin{proof}[Proof of Fact \ref{contributingfact2}]
All edges in $H$ are sent into blocks with value $1$, as $e_H(B,C)=e_H(C)=0$ by the fact that $(c_v)$ is a fractional vertex cover.

To show the other direction, we must show that
\[E_K(A,V(H))\cup E_K(B)\subseteq E(H).\]
But if there is any edge $e\in E_K(A,V(H))\cup E_K(B)\backslash E(H)$, then our fractional vertex cover $(c_v)$ is also a fractional vertex cover $H'=H\cup e$, so $c(H')=c(H)$ and so $\frac{e(H')-v(H')}{c(H')}>\frac{e(H)-v(H)}{c(H)}$, a contradiction.
\end{proof}

By our construction, no edges are sent into the blocks of the graphon with value $0$ (this is the point of separating out $T$). By Fact \ref{contributingfact2}, all edges in $E(H)\cup E_K(A,T)$ are sent into blocks of $W_0^{z,w}$ with value $1$, and all other edges in $E(K)$ are sent into blocks with value at least $p-o(p)$. Thus (since $p\ll 1$) the homomorphism count of this form is
\[(1-o(1))z^{|A|}w^{|B|/2}p^{v(H)+|T|+\gamma(|A|+|B|/2)+e(K)-e(H)-e_K(A,T)}.\]
Now, $|A|+|B|/2=c(H)$, by the minimality of our vertex cover, and $e(H)=v(H)+\gamma c(H)$ by the fact that $H$ is contributing. Since $e_K(A,T)=|T|$ by Fact \ref{contributingfact1}, we can simplify our expression to
\[(1-o(1))z^{|A|}w^{c(H)-|A|}p^{e(K)},\]
which is exactly the desired expression.

We now show that for all $H,A$, these homomorphisms are distinct; that is, given a homomorphism $\varphi$ of this form from $K$ to $W_0^{z,w}$, we can recover $H$ and $A$.

Recovering $A$ is simple; $A$ is just the set of vertices of $K$ that $\varphi$ sends into $[0,zp^{1+\gamma}]$.

Let $H'$ be the subgraph of $K$ consisting of the edges that $\varphi$ sends into blocks of $W_0^{z,w}$ of weight $1$. By Fact \ref{contributingfact2}, $E(H')=E(H)\cup E_K(A,T)$. By Fact \ref{contributingfact1}, each vertex in $T$ has degree $1$ in $H'$. Since $\delta(H)\geq 2$ by assumption, this means that we may recover $H$ as the $2$-core of $H'$, finishing the proof of Lemma \ref{constructionenoughhoms}.
\end{proof}
\begin{proof}[Proof of Lemma \ref{constructionlowentropy}]
First notice that the blocks of $W_0^{z,w}$ with value $1$ contribute $(1+o(1))(2z+w)p^{2+\gamma(K)}\log\frac{1}{p}$, because $I_p(1)=\log\frac{1}{p}$. Thus we must just show that the other parts contribute $o\left(p^{2+\gamma(K)}\log\frac{1}{p}\right)$.

We will use the first part of Lemma \ref{entropyapproxlemma}, stated in the next section. In particular, we use that if $x=O(p)$, $I_p(p+x)=\Theta\left(\frac{x^2}{p}\right)$. One may check using this that the blocks of $W_0^{z,w}$ with value $p+o(p)$ contribute $O\left(p^{2+\frac{3\gamma}{2}}\right)$. As $K$ is not a forest we must have $\gamma(K)\geq 0$ (as a cycle has $(e-v)/c=0$), so this contribution is sufficiently small.

Finally, the blocks with value $0$ contribute $O(p^{1+\gamma}I_p(0))=O(p^{2+\gamma})$. This contribution is also sufficiently small, proving the Lemma.
\end{proof}
\begin{proof}[Proof of Upper Bound of Theorems \ref{varcorrectexponent} and \ref{varcorrectconstant}]
By the previous two lemmas, we have that
\begin{align*}
\displaystyle\inf_{\substack{\Hom(K,W)\geq (1+\delta-o(1))p^{e(K)} \\ W\text{ }p\text{-regular}}}I_p(W) & \leq (1+o(1))\left(\min_{\substack{z,w\geq 0 \\ P_K(z,w)\geq 1+\delta}}\left(2z+w\right)\right)p^{2+\gamma(K)}\log\frac{1}{p} \\ & =(2+o(1))\rho(K,\delta)p^{2+\gamma(K)}\log\frac{1}{p}.
\end{align*}
The only difference between this and the desired upper bound of Theorems \ref{varcorrectexponent} and \ref{varcorrectconstant} is the $o(1)$ in the infimum on the left hand side. We may fix this by increasing $\delta$ by a negligible amount, but we must show that $\rho(K,\delta)$ also only increases by a negligible amount; that is, that
\[\rho(K,(1+o(1))\delta)=(1+o(1))\rho(K,\delta).\]

But if $k$ is the lowest degree of a nonconstant monomial in $P_K$, presuming such a monomial exists, since $P_K$ has constant term $1$ and only nonnegative coefficients, for all $z,w\geq 0$ and any $\epsilon>0$.
\[P_K((1+\epsilon)z,(1+\epsilon)w)-1\geq (1+\epsilon)^k (P_K(z,w)-1).\]
Therefore, if $P_K(z,w)\geq 1+\delta$, then $P_K((1+\epsilon)z,(1+\epsilon)w)\geq 1+(1+\epsilon)^k\delta$ for any $\epsilon>0$. By the definition of $\rho$, this immediately implies that
\[\rho(K,(1+\epsilon)^k\delta)\leq (1+\epsilon)\rho(K,\delta)\]
for all $\epsilon>0$. Since clearly $f$ is increasing, taking $\epsilon\to 0$ yields $\rho(K,(1+o(1))\delta)=(1+o(1))\rho(K,\delta)$ and we have proven the upper bound. Furthermore, by Proposition \ref{thesegraphonsworkprop}, a graphon attaining this upper bound satisfies the conditions of Theorem \ref{ldplowerboundthm} as long as we take $n,p$ with $(n^{-1}\log n)^{\frac{1}{2e(K)-2-\gamma(K)}}\ll p\ll 1$

If instead $P_K$ is constant, then $\rho(K,\delta)=\infty$, so we are also done in this case.
\end{proof}
We prove the upper bound of Theorem \ref{K24correctconstant} in a similar way.
\begin{proof}[Proof of Upper Bound of Theorem \ref{K24correctconstant}]
We will use the graphon given in Figure \ref{K24graphonfigure}, with
\[a(p)=d_1p^2\left(\log\frac{1}{p}\right)^{-\frac{1}{3}}\left(\log\log\frac{1}{p}\right)^{\frac{1}{3}}\]
and
\[b(p)=d_2p\left(\log\frac{1}{p}\right)^{\frac{2}{3}}\left(\log\log\frac{1}{p}\right)^{-\frac{2}{3}}\]
for constants $d_1,d_2$.
Call this graphon $W_1^{d_1,d_2}$. Note that since $a(p)\lesssim pb(p)$, the asymptotics in Figure \ref{K24graphonfigure} hold. We prove two claims, analogous to Lemma \ref{constructionenoughhoms} and \ref{constructionlowentropy}.
\setcounter{claimcounter}{0}
\begin{claim}
\[\Hom(K_0,W_1^{d_1,d_2})\geq (1+d_1^2d_2-o(1))p^9\]
\end{claim}
Consider homomorphisms where the two vertices of degree $4$ in $K_0$ are sent into the hub $[0,a(p)]$ and all four other vertices are sent into $(a(p),p]$. Since $a(p)=o(p)$, and all edges are sent into blocks of $W_1^{c_1,c_2}$ with value $1$ except for the single edge between the two vertices of degree $3$, we have that the number of homomorphisms of this form is at least
\[(1-o(1)) a(p)^2b(p)p^4=(1-o(1))d_1^2d_2p^9.\]
Since there are $(1-o(1))p^9$ homomorphisms where each vertex is sent into $(p,1]$, we have that
\[\Hom(K_0,W_1^{d_1,d_2})\geq (1+d_1^2d_2-o(1))p^9,\]
proving the claim.
\begin{claim}
\[I_p(W_1^{d_1,d_2})=\left(2d_1+\frac{2d_2}{3}+o(1)\right)p^3\left(\log\frac{1}{p}\right)^{\frac{2}{3}}\left(\log\log\frac{1}{p}\right)^{\frac{1}{3}},\]
\end{claim}
Now we compute $I_p(W_1^{d_1,d_2})$. The blocks of value $1$ contribute at most $(2+o(1))pa(p)\log\frac{1}{p}$. The blocks of value $0$ contribute $(2+o(1))pa(p)$, which is asymptotically smaller, so we may ignore them.

The blocks of value $p+O(pb(p))$ (again using that $I_p(p+x)=\Theta\left(\frac{x^2}{p}\right)$ for $x=O(p)$) contribute $O(p^2b(p)^2)\ll pa(p)\log\frac{1}{p}$ (since the former is $p^{4+o(1)}$ and the latter is $p^{3+o(1)}$), and similarly, the lower right block contributes $O(p^3b(p)^2)$, which is also negligible.

Thus we have
\begin{align*}
I_p(W_1^{d_1,d_2}) & =(1+o(1))\left(2pa(p)\log\frac{1}{p}+p^2I_p(b(p))\right) \\ & =\left(2d_1+\frac{2d_2}{3}+o(1)\right)p^3\left(\log\frac{1}{p}\right)^{\frac{2}{3}}\left(\log\log\frac{1}{p}\right)^{\frac{1}{3}},
\end{align*}
where we have used our expressions for $a(p)$ and $b(p)$ and the fact that by Lemma \ref{entropyapproxlemma},
\[I_p(b(p))=(1+o(1))b(p)\log\frac{b(p)}{p}=\left(\frac{2d_2}{3}+o(1)\right)p\left(\log\frac{1}{p}\right)^{\frac{2}{3}}\left(\log\log\frac{1}{p}\right)^{\frac{1}{3}}.\]
This proves the second claim.

Given the claims, we would like to minimize $2d_1+\frac{2d_2}{3}$ while maximizing $d_1^2d_2$. This occurs when $d_2=\frac{3d_1}{2}$. In particular, for any fixed $\epsilon>0$, let $d_1=(1+\epsilon)\left(\frac{2\delta}{3}\right)^{\frac{1}{3}}$ and $d_2=(1+\epsilon)\left(\frac{9\delta}{4}\right)^{\frac{1}{3}}$.
Then
\[\Hom(K_0,W_1^{d_1,d_2})\geq (1+(1+\epsilon)\delta-o(1))p^9\geq (1+\delta)p^9,\]
and
\[I_p(W_1^{d_1,d_2})\leq (1+\epsilon+o(1))(18\delta)^{\frac{1}{3}}p^3\left(\log\frac{1}{p}\right)^{\frac{2}{3}}\left(\log\log\frac{1}{p}\right)^{\frac{1}{3}}.\]
Therefore,
\[\displaystyle\inf_{\substack{\Hom(K_0,W)\geq (1+\delta)p^{9} \\ W\text{ }p\text{-regular}}}I_p(W)\leq (1+\epsilon+o(1))(18\delta)^{\frac{1}{3}}p^3\left(\log\frac{1}{p}\right)^{\frac{2}{3}}\left(\log\log\frac{1}{p}\right)^{\frac{1}{3}}\]
for all fixed $\epsilon>0$. The conclusion follows. Again, by Proposition \ref{thesegraphonsworkprop}, a graphon attaining this upper bound satisfies the conditions of Theorem \ref{ldplowerboundthm} as long as we take $n,p$ with $(n^{-1}\log n)^{\frac{1}{15}}\ll p\ll 1$.
\end{proof}
\section{Lower Bound Preliminaries}\label{preliminariessection}
The lower bounds in Theorems \ref{varcorrectexponent}, \ref{varcorrectconstant}, and \ref{K24correctconstant} will be proven over the next several sections. We will start by introducing the ideas behind and proving the lower bound of Theorem \ref{varcorrectexponent}, and then strengthen those ideas to prove the other two lower bounds. We first prove some preliminary identities and bounds that will be useful for manipulating commonly occurring expressions in our proof.

\begin{lemma}\label{subgraphbreakdown}
If $W:[0,1]^2\to\mathbb{R}$ is symmetric and measurable and $U=W-p$, then
\[\Hom(K,W)=\displaystyle\sum_{H\subseteq K}p^{e(K)-e(H)}\Hom(H,U).\]
\end{lemma}
\begin{proof}[Proof of Lemma \ref{subgraphbreakdown}]
We simply expand out
\begin{align*}
\Hom(K,W) & =\displaystyle\int_{[0,1]^{v(K)}}\displaystyle\prod_{vw\in E(K)}W(x_v,x_w)\displaystyle\prod_{v\in V(K)}dx_v \\ & =\displaystyle\int_{[0,1]^{v(K)}}\displaystyle\prod_{vw\in E(K)}(p+U(x_v,x_w))\displaystyle\prod_{v\in V(K)}dx_v \\ & =\displaystyle\sum_{S\subseteq E(K)}p^{e(K)-|S|}\displaystyle\int_{[0,1]^{v(K)}}\displaystyle\prod_{vw\in S}U(x_v,x_w)\displaystyle\prod_{v\in V(K)}dx_v \\ & =\displaystyle\sum_{H\subseteq K}p^{e(K)-e(H)}\Hom(H,U).
\end{align*}
(Since we ignore isolated vertices, subsets of $E(K)$ simply correspond to subgraphs of $K$.)
\end{proof}
It will be important to have bounds on the entropy function $I_p(W)$, especially in terms of the moments of $U:=W-p$. To that end, we prove the following three lemmas.
\begin{lemma}\label{entropyapproxlemma}
Let $p\to 0$ and $x=x(p)$. If $|x|=O(p)$, then $I_p(p+x)=\Theta\left(\frac{x^2}{p}\right)$. If $x\gg p$, then $I_p(p+x)=(1+o(1))x\log\frac{x}{p}$.
\end{lemma}
\begin{proof}
We may split $|x|=O(p)$ into the cases when $x=o(p)$ and when $x=\Theta(p)$. This is because if the result does not hold for some $x=O(p)$, we may find a subsequence $(p_i)$ with either $x=o(p)$ or $x=\Theta(p)$ on that subsequence where the result also does not hold.

Lemma 3.3 of \cite{LZ} implies the Lemma both when $|x|\ll p$ and when $x\gg p$. (The lemma does not explicitly address when $x<0$, but the exact same argument applies.)

We must just deal with the case when $x=\Theta(p)$. We seek to show that $I_p(p+x)=\Theta(p)$. Say $x=ap$, $a>-1$, $a\neq 0$ with $a=\Theta(1)$. (The case $a=-1$ is easily computed separately.) Then $I_p(p+x)=(1+a)(\log (1+a))p+(1-(1+a)p)\log\frac{1-(1+a)p}{1-p}$. The second term is $-(1+o(1))ap$, so $I_p(p+x)=((1+a)\log(1+a)-a+o(1))p$. Since $f(a):=(1+a)\log (1+a)-a$ and its first derivative are $0$ at $a=0$ and $f$ has positive second derivative on $(-1,\infty)$, $f(a)$ cannot be $0$ unless $a=0$. Thus $f(a)=\Theta(1)$ if $a=\Theta(1)$, and so $I_p(p+x)=\Theta(p)$. This finishes this case and completes the proof.
\end{proof}
\begin{lemma}\label{momentlemma}
For all $c>0$ there exists $\epsilon>0$ such that if $p\to 0$ and $x=x(p)$ with $|x|\geq p^{1+\epsilon}$, then
\[I_p(p+x)\geq (1-o(1))|x|^{1+c}\log\frac{1}{p}.\]
\end{lemma}
\begin{lemma}\label{firstmomentlemma}
Let $p\to 0$. Take $x=x(p)\in [-p,1-p]$ and $\epsilon=\epsilon(p)\in (0,1]$. If $|x|\geq\epsilon p$, then
\[I_p(p+x)=\Omega(\epsilon |x|).\]
\end{lemma}
The intuition behind these lemmas is that by Lemma \ref{entropyapproxlemma}, we are trying to bound $x\log(x/p)$ from below. $x\log(x/p)$ is convex, but grows slower than $x^c$ for any $c>1$, so when comparing it to $x$ we should `smooth inward' but when comparing it to $x^c$ for $c>1$ we should `smooth outward'.
\begin{proof}[Proof of Lemma \ref{momentlemma}]
We split into $x=O(p)$ and $x\gg p$ in the same way as the last lemma.

If $x=O(p)$, then by Lemma \ref{entropyapproxlemma} it suffices to show that $\frac{x^2}{p}\gg |x|^{1+c}\log\frac{1}{p}$, or in other words that
\[|x|^{1-c}\gg p\log\frac{1}{p}.\]
If $c>1$, this holds for all $x=O(p)$, and if $c\leq 1$ this holds for all $|x|\geq p^{1+\epsilon}$ as long as we choose $\epsilon$ such that $(1+\epsilon)(1-c)<1$.

The other case is when $x\gg p$. In this case, by Lemma \ref{entropyapproxlemma}, it suffices to show that
\[x\log\frac{x}{p}\geq (1-o(1))x^{1+c}\log\frac{1}{p}.\]
Take $y$ such that $x=p^y$ (so $0<y<1$ since $x\gg p$). Then after cancelling a factor of $x\log\frac{1}{p}$, the expression above reduces to
\[1-y\geq (1-o(1))p^{cy}.\]
Since $\log (1-y)=-y-\frac{y^2}{2}-\frac{y^3}{3}-\cdots>-\frac{y}{1-y}$, it suffices to show that
\[\frac{y}{1-y}\leq cy\log\frac{1}{p}+o(1).\]
As long as $y\leq 1-\frac{1}{c\log\frac{1}{p}}$, this holds (without the $o(1)$, in fact).

If $y\geq 1-\frac{1}{c\log\frac{1}{p}}$, then $x\leq p\cdot p^{-\frac{1}{c\log\frac{1}{p}}}$. But $p^{-\frac{1}{c\log\frac{1}{p}}}=e^{\frac{1}{c}}=O(1)$, so in this case $x=O(p)$, a contradiction, completing the proof.
\end{proof}
\begin{proof}[Proof of Lemma \ref{firstmomentlemma}]
When $x\gg p$, we use Lemma \ref{entropyapproxlemma}. We have that
\[I_p(p+x)=(1+o(1))x\log\frac{x}{p}\gg x\geq\epsilon |x|,\]
since $x\gg p$ and $\epsilon\leq 1$.

When $x=O(p)$, we may again use Lemma \ref{entropyapproxlemma}. In particular, we have that
\[I_p(p+x)=\Theta\left(\frac{x^2}{p}\right)=\Omega(\epsilon |x|),\]
since $|x|\geq\epsilon p$ by assumption.
\end{proof}
Lemmas \ref{momentlemma} and \ref{firstmomentlemma} bound the function $I_p$ effectively as long as we are not evaluating $I_p$ on a number very close to $p$. Thus we would like to only consider graphons taking no values that are $p+o(p)$ (except for $p$ itself, where $I_p(p)=0$ and thus our bounds will hold anyway).

For this reason, we define a class of graphons that have several desirable properties, including taking no values close to $p$ except $p$ itself.
\begin{definition}\label{nicegraphonsdef}
For any graph $K$, $\delta\geq 0$, $p\geq 0$, and $0\leq\epsilon\leq 1$, let $\Gamma_{\epsilon}(K,\delta,p)$ be the set of all graphons $W$ satisfying
\begin{enumerate}
\item $(1-\epsilon)p\leq\displaystyle\int_0^1 W(x_0,y)dy\leq (1+\epsilon)p$ for all $x_0\in [0,1]$
\item $\Hom(K,W)\geq (1-\epsilon)^{e(K)}(1+\delta)p^{e(K)}$
\item $W$ takes no values in $[(1-\epsilon)p,(1+\epsilon)p]\backslash\{p\}$.
\item $\Hom(H,|W-p|)\leq 2^{e(K)}(1+\delta)p^{e(H)}$ for all $H\subseteq K$.
\end{enumerate}
We will in general write $\Gamma_{\epsilon}(K,\delta)$ when the $p$ is implicit.
\end{definition}
\begin{remark}
Conditions (1) and (2) are simply approximate versions of the $p$-regularity and homomorphism count conditions from the variational problem (when $\epsilon=0$ they are exactly the same). Condition (3) will allow the use of our bounds in Lemmas \ref{momentlemma} and \ref{firstmomentlemma}, as discussed earlier. Condition (4) is a technical condition that will be useful later.

This idea of eliminating values of $W$ close to $p$ in order to ensure the accuracy of polynomial approximations of entropy also appears in the work of Liu and Zhao \cite{LiuZhao}.
\end{remark}
We would like to show that the solution to the variational problem is approximately the same when we work over $\Gamma_{\epsilon}(K,\delta)$. This is accomplished by the following lemma.
\begin{lemma}\label{replacebyplemma}
Take any graph $K$ and constant $\delta>0$. If $p\to 0$ and $\epsilon=\epsilon(p)$ with $0\leq\epsilon\leq 1$, then
\[\displaystyle\inf_{W\in\Gamma_{\epsilon}(K,\delta,p)}I_p(W)\leq (1+o(1))\displaystyle\inf_{\substack{\Hom(K,W')\geq (1+\delta)p^{e(K)} \\ W'\text{ }p\text{-regular}}}I_p(W').\]
\end{lemma}
\begin{proof}[Proof of Lemma \ref{replacebyplemma}]
The idea is to replace values close to $p$ by $p$ itself.

First note that if the set of $p$-regular graphons $W'$ with $\Hom(K,W')\geq (1+\delta)p^{e(K)}$ is empty, the right side is $\infty$ and we are done.

Otherwise, take a $W'=W'(p)$ that approximately minimizes $I_p(W')$ subject to $\Hom(K,W')\geq (1+\delta)p^{e(K)}$ and $W'$ $p$-regular, so that $I_p(W')$ is within a $1+o(1)$ factor of the infimum on the right side of the lemma. It then suffices to find $W\in\Gamma_{\epsilon}(K,\delta)$ with $I_p(W)\leq I_p(W')$.

Now, $\frac{W'+p}{2}$ is $p$-regular. Since $I_p$ is convex and $I_p(p)=0$, we have that $I_p\left(\frac{x+p}{2}\right)\leq\frac{I_p(x)}{2}$ for all $x$. Thus $I_p\left(\frac{W'+p}{2}\right)\leq\frac{I_p(W')}{2}$. But by assumption, $I_p(W')$ is at most $1+o(1)$ times the entropy of any $p$-regular graphon with at least $(1+\delta)p^{e(K)}$ homomorphisms from $K$. Thus
\[\Hom\left(K,\frac{W'+p}{2}\right)<(1+\delta)p^{e(K)}.\]
We can cancel the $2$ in the denominator by multiplying by $2^{e(K)}$. Furthermore,
\[\Hom(K,W'+p)\geq p^{e(K)-e(H)}\Hom(H,W'+p)\]
for all $H\subseteq K$, because the right hand side can be obtained from the left by writing out the integral form of $\Hom(K,W'+p)$ and replacing each occurrence of $W'+p$ with $p$ in factors that do not correspond to edges of $H$. Therefore,
\[\Hom(H,W'+p)\leq (1+\delta)2^{e(K)}p^{e(H)}\]
for all $H\subseteq K$, and since $|W'-p|\leq W'+p$,
\[\Hom(H,|W'-p|)\leq (1+\delta)2^{e(K)}p^{e(H)}.\]

We now construct $W$ by replacing all values of $W'$ in $[(1-\epsilon)p,(1+\epsilon)p]$ with $p$. $W$ clearly satisfies condition (3) of Definition \ref{nicegraphonsdef}. Furthermore, since pointwise we have $W\geq\frac{1}{1+\epsilon}W'\geq (1-\epsilon)W'$, condition (2) holds as well. Since we also have $|W'-W|\leq\epsilon p$ pointwise and $W'$ is $p$-regular, condition (1) holds.

Finally, condition (4) holds because $|W-p|\leq |W'-p|$ pointwise, so
\[\Hom(H,|W-p|)\leq\Hom(H,|W'-p|)\leq (1+\delta)2^{e(K)}p^{e(H)}.\]
Thus $W\in\Gamma_{\epsilon}(K,\delta)$. Since $I_p(W)\leq I_p(W')$ (since $I_p$ is minimized at $p$) and $W'$ minimizes the variational problem by assumption, we are done.
\end{proof}
The upshot of Lemma \ref{replacebyplemma} is that we may now apply Lemmas \ref{momentlemma} and \ref{firstmomentlemma} to the values of our graphon, with the only cost being that due to conditions (1) and (2) of Definition \ref{nicegraphonsdef}, we only have approximate $p$-regularity and slightly fewer homomorphisms.

After applying Lemma \ref{subgraphbreakdown}, we will need to bound the terms $\Hom(H,U)$. This will culminate in the following result.
\begin{thm}\label{simpleholderresult}
Let $H$ be a graph with no isolated vertices and let $W$ be a graphon satisfying $(1-\epsilon)p\leq\displaystyle\int_0^1 W(x_0,y)dy\leq (1+\epsilon)p$ for all $x_0\in [0,1]$ for some $\epsilon\leq 1$. If $U=W-p$, then

\[\Hom(H,|U|)\leq ((2+\epsilon)p)^{v(H)-2c(H)}\mathbb{E}(|U|)^{c(H)}.\]
\end{thm}
We will prove Theorem \ref{simpleholderresult} over the following two sections. For now, we complete the lower bound of Theorem \ref{varcorrectexponent}. We start with a corollary of Theorem \ref{simpleholderresult}.
\begin{cor}\label{simpleholdercor}
Take a graph $H$ with no isolated vertices. Further take $p\to 0$, $\epsilon=\epsilon(p)$, and a graphon $W$ collectively satisfying the constraints of Theorem 
\ref{simpleholderresult}. If $W$ takes no values in $[(1-\epsilon)p,(1+\epsilon)p]\backslash\{p\}$, then
\[\Hom(H,|U|)=O\left(\epsilon^{-c(H)}p^{v(H)-2c(H)}I_p(W)^{c(H)}\right).\]
\end{cor}
\begin{proof}[Proof of Corollary \ref{simpleholdercor}]
By Theorem \ref{simpleholderresult},
\[\Hom(H,|U|)=O\left(p^{v(H)-2c(H)}\mathbb{E}(|U|)^{c(H)}\right).\]
Thus it suffices to show that $I_p(W)=\Omega(\epsilon\mathbb{E}|U|)$. But for all $x,y\in [0,1]$,
\begin{align*}
I_p(W(x,y)) & =I_p(p+U(x,y)) \\ & =\Omega(\epsilon |U(x,y)|),
\end{align*}
as by our assumption either $U(x,y)=0$ (in which case $I_p(p+U(x,y))=|U(x,y)|=0$), or $|U(x,y)|>\epsilon p$, in which case Lemma \ref{firstmomentlemma} applies. Integrating both sides yields over $[0,1]^2$ yields the Corollary.
\end{proof}
\begin{proof}[Proof of Lower Bound of Theorem \ref{varcorrectexponent}]
Fix a nonforest graph $K$ whose $2$-core is not a disjoint union of cycles, and take $\delta>0$ and $p\to 0$. By Lemma \ref{replacebyplemma}, it suffices to show that there exists $\epsilon=\epsilon(p)$, $0\leq\epsilon\leq 1$ such that for all $W\in\Gamma_{\epsilon}(K,\delta)$,
\[I_p(W)\gtrsim p^{2+\gamma(K)}.\]

Take $\epsilon$ to be the constant such that $(1+\delta)(1-\epsilon)^{e(K)}=1+\frac{\delta}{2}$ and take $W$ to be any graphon in $\Gamma_{\epsilon}(K,\delta)$. Then by condition (2) of Definition \ref{nicegraphonsdef},
\[\Hom(K,W)\geq \left(1+\frac{\delta}{2}\right)p^{e(K)}.\]
Now, by Lemma \ref{subgraphbreakdown},
\[p^{-e(K)}\Hom(K,W)=\displaystyle\sum_{H\subseteq K}p^{-e(H)}\Hom(H,U),\]
where $U=W-p$. When $H=\emptyset$, $\Hom(H,U)=1$. Thus since the left hand side is at least $1+\frac{\delta}{2}$, there must be some $H_0\subseteq K$, $H_0\neq\emptyset$ (with $H_0$ possibly depending on $p$) such that
\[\Hom(H_0,|U|)\geq\Hom(H_0,U)=\Omega(p^{e(H_0)}).\]
By conditions (1) and (3) of Definition \ref{nicegraphonsdef}, the conditions of Corollary \ref{simpleholdercor} hold. Therefore, (since we chose $\epsilon=\Theta(1)$)
\[p^{e(H_0)}=O\left(p^{v(H_0)-2c(H_0)}I_p(W)^{c(H_0)}\right).\]
Rearranging,
\[I_p(W)^{c(H_0)}\gtrsim p^{e(H_0)-v(H_0)+2c(H_0)},\]
so
\begin{align*}
I_p(W) & \gtrsim p^{2+\frac{e(H_0)-v(H_0)}{c(H_0)}} \\ & \geq p^{2+\gamma(K)},
\end{align*}
as $\gamma(K)$ is the maximum of $\frac{e(H)-v(H)}{c(H)}$ over all subgraphs $H\subseteq K$. This completes the proof.
\end{proof}
\section{Examples}\label{examplessection}
We will demonstrate how to prove Theorem \ref{simpleholderresult} for two specific graphs $H$, in order to demonstrate the proof techniques we will use to prove the result in general, and in particular to demonstrate our specific use of H\"{o}lder's Inequality.

\begin{example}\label{butterflyexample}
Let $H_1$ be the butterfly graph; that is, two triangles joined at a vertex.
\end{example}
\begin{proof}[Proof of Theorem \ref{simpleholderresult} for $H=H_1$]
Label the unique degree-$4$ vertex of $H_1$ $v$, and let $w_1,w_2,w_3,w_4$ be the other vertices such that $w_1w_2$ and $w_3w_4$ are edges of $H_1$.

We compute $c(H_1)$. It is clear that $c(H_1)\leq\frac{5}{2}$, as weighting all vertices $\frac{1}{2}$ is a valid fractional vertex cover of $H_1$.

For any fractional vertex cover $(c_v,c_{w_1},c_{w_2},c_{w_3},c_{w_4})$, we have that $2(c_v+c_{w_1}+c_{w_2})=(c_v+c_{w_1})+(c_v+c_{w_2})+(c_{w_1}+c_{w_2})\geq 3$, as $vw_1w_2$ is a triangle. Since $w_3w_4$ is an edge, we also know that $c_{w_3}+c_{w_4}\geq 1$. Therefore, $c_v+c_{w_1}+c_{w_2}+c_{w_3}+c_{w_4}\geq\frac{5}{2}$, so $c(H_1)\geq\frac{5}{2}$. Since our upper and lower bounds agree, $c(H_1)=\frac{5}{2}$. Notice that $v(H)=2c(H)$.

Thus to prove Theorem \ref{simpleholderresult}, we must show that if $W$ is a graphon satisfying $(1-\epsilon)p\leq\displaystyle\int_0^1 W(x_0,y)dy\leq (1+\epsilon)p$ for all $x_0\in [0,1]$ and $U=W-p$,
\begin{equation}\label{H1eq}
\Hom(H_1,U)\leq\mathbb{E}(|U|)^{\frac{5}{2}}.
\end{equation}
It turns out that for this choice of $H=H_1$, the `approximate regularity' condition will not be necessary, and in fact we will show the stronger statement that \ref{H1eq} holds for any symmetric measurable function $U:[0,1]^2\to [-1,1]$. Take $U$ to be such a function. We may assume $U$ is positive-valued, as both sides of the desired inequality only depend on $|U|$. We first write out
\[\Hom(H_1,U)=\displaystyle\int_{[0,1]^5}U(x_{w_1},x_{w_2})U(x_{w_3},x_{w_4})\displaystyle\prod_{i=1}^4 U(x_v,x_{w_i})dx_v\displaystyle\prod_{i=1}^4 dx_{w_i}.\]
We will apply H\"{o}lder's inequality at one vertex at a time of $H_1$. Our weights will be `given' by the fractional perfect matching where the edges $w_1w_2$ and $w_3w_4$ have weight $\frac{3}{4}$ and the edges $vw_i$, $1\leq i\leq 4$ have weight $\frac{1}{4}$. (The exact way of turning an edge weighting into an application of H\"{o}lder's inequality will be given by Theorem \ref{graphholderthm}.)

It will be useful to define $S_a(x)=\left(\displaystyle\int_0^1 U(x,y)^a dy\right)^{\frac{1}{a}}$ for $a\geq 1$. Notice that $\| S_a\|_a=\| U\|_a$.

We first apply H\"{o}lder's inequality at vertices $w_1$ and $w_2$. We first break out those vertices into an inner integral, writing
\begin{align*}\Hom(H_1,U) & =\displaystyle\int_{[0,1]^3}U(x_v,x_{w_3})U(x_v,x_{w_4})U(x_{w_3},x_{w_4}) \\ &  \cdot\left(\displaystyle\int_0^1 U(x_v,x_{w_2})\left(\displaystyle\int_0^1 U(x_v,x_{w_1})U(x_{w_1},x_{w_2})dx_{w_1}\right) dx_{w_2}\right)dx_v dx_{w_3}dx_{w_4}.\end{align*}
By H\"{o}lder's inequality, for any fixed $x_v,x_{w_2}$,
\begin{align*}
\displaystyle\int_0^1 U(x_v,x_{w_1})U(x_{w_1},x_{w_2})dx_{w_1} & \leq\left(\displaystyle\int_0^1 U(x_v,x_{w_1})^4 dx_{w_1}\right)^{\frac{1}{4}}\left(\displaystyle\int_0^1 U(x_{w_2},x_{w_1})^{\frac{4}{3}} dx_{w_1}\right)^{\frac{3}{4}} \\ & =S_4(x_v)\cdot S_{\frac{4}{3}}(x_{w_2})
\end{align*}
Substituting, we have
\begin{align*}\Hom(H_1,U) & =\displaystyle\int_{[0,1]^3}U(x_v,x_{w_3})U(x_v,x_{w_4})U(x_{w_3},x_{w_4})S_4(x_v) \\ &  \cdot\left(\displaystyle\int_0^1 U(x_v,x_{w_2})S_{\frac{4}{3}}(x_{w_2}) dx_{w_2}\right)dx_v dx_{w_3}dx_{w_4}.\end{align*}
Applying H\"{o}lder again,
\begin{align*}
\displaystyle\int_0^1 U(x_v,x_{w_2})S_{\frac{4}{3}}(x_{w_2}) dx_{w_2} & \leq\left(\displaystyle\int_0^1 U(x_v,x_{w_2})^4\right)^{\frac{1}{4}}\left(\displaystyle\int_0^1 S_{\frac{4}{3}}(x_{w_2})^{\frac{4}{3}}\right)^{\frac{3}{4}} \\ & =S_4(x_v)\| S_{\frac{4}{3}}\|_{\frac{4}{3}} \\ & =S_4(x_v)\cdot \| U\|_{\frac{4}{3}}.
\end{align*}
Substituting again,
\[\Hom(H_1,U)\leq\| U\|_{\frac{4}{3}}\displaystyle\int_{[0,1]^3}U(x_v,x_{w_3})U(x_v,x_{w_4})U(x_{w_3},x_{w_4})S_4(x_v)^2dx_v dx_{w_3}dx_{w_4}.\]
Breaking out $w_3$ and $w_4$ and applying H\"{o}lder in the same way, we see that for all $x_v\in [0,1]$,
\begin{align*}
\displaystyle\int_{[0,1]^2}U(x_v,x_{w_3})U(x_v,x_{w_4})U(x_{w_3},x_{w_4})dx_{w_3}dx_{w_4} & =\displaystyle\int_0^1U(x_v,x_{w_4})\left(\displaystyle\int_0^1 U(x_v,x_{w_3})U(x_{w_3},x_{w_4})\right)dx_{w_3}dx_{w_4} \\ & \leq\displaystyle\int_0^1 U(x_v,x_{w_4})S_4(x_v)S_{\frac{4}{3}}(x_{w_4})dx_{w_4} \\ & =S_4(x_v)\displaystyle\int_0^1 U(x_v,x_{w_4})S_{\frac{4}{3}}(x_{w_4})dx_{w_4} \\ & \leq S_4(x_v)^2 \| S_{\frac{4}{3}}\|_{\frac{4}{3}} \\ & =S_4(x_v)^2\| U\|_{\frac{4}{3}}.
\end{align*}
Substituting in for the final time, and applying H\"{o}lder again at the vertex $v$,
\begin{align*}
\Hom(H_1,U) & \leq \| U\|_{\frac{4}{3}}^2\displaystyle\int_0^1 S_4(x_v)^4 dx_v \\ & \leq \| U\|_{\frac{4}{3}}^2 \| S_4\|_4^4 \\ & =\| U\|_{\frac{4}{3}}^2\| U\|_4^4.
\end{align*}
Note the correspondence between this and the fractional perfect matching we gave at the beginning of the argument. We started with two edges weighted $\frac{3}{4}$ and four edges weighed $\frac{1}{4}$ and we ended with two copies of $\| U\|_{\frac{4}{3}}$ and four copies of $\| U\|_4$.

All that is left is to note that for all $a\geq 1$, $\| U\|_a\leq\| U\|_1^{\frac{1}{a}}$, as $U$ is $1$-bounded. Therefore,
\begin{align*}
\Hom(H_1,U) & \leq\| U\|_{\frac{4}{3}}^2\| U\|_4^4 \\ & \leq \left(\|U\|_1^{\frac{3}{4}}\right)^2\left(\| U\|_1^{\frac{1}{4}}\right)^4 \\ & =\| U\|_1^{\frac{5}{2}},
\end{align*}
completing the proof. Notice that the power of $\| U\|_1$ we obtained was given by adding the edge weights from earlier, yielding $c(H_1)$.
\end{proof}
As discussed, we are implicitly using an edge weighting of the graph $H_1$ to determine our application of H\"{o}lder. It is useful that $H_1$ has a fractional perfect matching. We now consider a case where there is no such fractional perfect matching.
\begin{example}\label{K23example}
Let $H_2=K_{2,3}$.
\end{example}
\begin{proof}[Proof of Theorem \ref{simpleholderresult} for $H=H_2$]
Let the two vertices of degree $3$ be called $v_1,v_2$ and the three vertices of degree $2$ be called $w_1,w_2,w_3$. Since giving $v_1,v_2$ weight $1$ and $w_1,w_2,w_3$ weight $0$ yields a fractional vertex cover, $c(H_2)\leq 2$. For any fractional vertex cover $(c_v)$, $c_{v_1}+c_{v_2}+c_{w_1}+c_{w_2}\geq 2$, because of the constraints given by the two edges $v_1w_1$ and $v_2w_2$, so in fact $c(H_2)=2$. This implies that $v(H_2)-2c(H_2)=1$.

Take $W$ satisfying the conditions of Theorem \ref{simpleholderresult}; that is, for some $0\leq\epsilon\leq 1$, $(1-\epsilon)p\leq\displaystyle\int_0^1 W(x_0,y)dy\leq (1+\epsilon)p$ for all $x_0\in [0,1]$. Let $U=W-p$. We would like to show that
\[\Hom(H_2,|U|)\leq (2+\epsilon)p\mathbb{E}(|U|)^2.\]

We may write

\begin{align*}
\Hom(H_2,|U|) & =\displaystyle\int_{[0,1]^5}\left(\displaystyle\prod_{i=1}^2\displaystyle\prod_{j=1}^3 |U(x_{v_i},x_{w_j})|\right)dx_{v_1}dx_{v_2}dx_{w_1}dx_{w_2}dx_{w_3} \\ & =\displaystyle\int_{[0,1]^2}\displaystyle\prod_{j=1}^3\left(\displaystyle\int_0^1 |U(x_{v_1},x_{w_j})U(x_{v_2},x_{w_j})|dx_{w_j}\right)dx_{v_1}dx_{v_2}.
\end{align*}
Now, for each inner integral, we may apply Cauchy-Schwarz to say that
\[\displaystyle\int_0^1 |U(x_{v_1},x_{w_j})U(x_{v_2},x_{w_j})|dx_{w_j}\leq S_2(x_{v_1})S_2(x_{v_2}),\]
where we define (similarly to before) $S_a(x)=\left(\displaystyle\int_0^1 |U(x,y)|^a dy\right)^{\frac{1}{a}}$.

Here we are (in the language of the previous example) implicitly using the fractional edge cover with $w_e=\frac{1}{2}$ for all edges $e$, and applying H\"{o}lder (which is here just Cauchy-Schwarz) using those edge weights at the vertices $w_1$, $w_2$, and $w_3$. However, this edge cover is not a fractional perfect matching, which will pose an issue when we move to $v_1$ and $v_2$, where the edge weights do not sum to $1$.

Substituting our bound on the inner integral, we see that

\[\Hom(H_2,|U|)\leq\displaystyle\int_{[0,1]^2} S_2(x_{v_1})^3 S_2(x_{v_2})^3 dx_{v_1}dx_{v_2}=\displaystyle\int_0^1 S_2(x_{v_1})^3dx_{v_1}\displaystyle\int_0^1 S_2(x_{v_2})^3dx_{v_2}.\]

The difference between this example and the previous is that we now have an expression of the form $S_2^3$, whereas in the previous we only had expressions of the form $S_a^a$. This is equivalent to the fact the edge weights do not sum to $1$ at $v_1$ and $v_2$.

However, we solve this problem by noting that we in fact can bound $\| S_2\|_{\infty}$, so we can `pull out' a factor of $S_2$, replacing it with its upper bound. Specifically, for all $x\in [0,1]$, since $\| U\|_{\infty}\leq 1$,
\begin{align*}
S_2(x) & =\left(\displaystyle\int_0^1 |U(x,y)|^2 dy\right)^{\frac{1}{2}} \\ & \leq\left(\displaystyle\int_0^1 |U(x,y)| dy\right)^{\frac{1}{2}} \\ & =\left(\displaystyle\int_0^1 |W(x,y)-p| dy\right)^{\frac{1}{2}} \\ & \leq\left(\displaystyle\int_0^1 (W(x,y)+p)\right)^{\frac{1}{2}} \\ & \leq ((2+\epsilon)p)^{\frac{1}{2}},
\end{align*}
by our degree condition $\displaystyle\int_0^1 W(x,y)dy\leq (1+\epsilon)p$ for all $x\in [0,1]$.
Therefore,
\begin{align*}
\Hom(H_2,U) & \leq\displaystyle\int_0^1 S_2(x_{v_1})^3dx_{v_1}\displaystyle\int_0^1 S_2(x_{v_2})^3dx_{v_2} \\ & \leq (2+\epsilon)p\displaystyle\int_0^1 S_2(x_{v_1})^2dx_{v_1}\displaystyle\int_0^1 S_2(x_{v_2})^2dx_{v_2} \\ & =(2+\epsilon)p\| S_2\|_2^4 \\ & =(2+\epsilon)p\| U\|_2^4 \\ & \leq (2+\epsilon)p\| U\|_1^2,
\end{align*}
as $U$ is $1$-bounded. Since $\| U\|_1=\mathbb{E}(|U|)$, this completes the proof in this example.
\end{proof}
Notice that in this last example, at each $v_i$ we had one copy of $S_2$ coming from applying Cauchy-Schwarz at each of the three vertices $w_1,w_2,w_3$. We then eliminated one of those copies by replacing it with its upper bound and pulling it out of the integral; say, the copy coming from $w_3$. Thus when we are looking at the vertices $v_1$ and $v_2$, we are implicitly not using the weighting $w_e=\frac{1}{2}$ $\forall e$ from earlier, but instead using the fractional matching
\[w_e=\begin{cases}\frac{1}{2} & e\in\{v_1w_1,v_1w_2,v_2w_1,v_2w_2\} \\ 0 & e\in\{v_1w_3,v_2w_3\}\end{cases}.\]
This switching between a minimum fractional edge cover and a maximum fractional matching to make the weights sum to $1$ at the appropriate vertices by 	`pulling out' copies of $S_i$ is a key concept in the general case, as we will see in Theorem \ref{graphholderthm}.
\section{H\"older's Inequality}\label{holdersection}
The goal of this section will be to prove our main technical engine, which will be a generalized H\"older's inequality. We begin by citing a generalized H\"older's inequality that has appeared in several previous works.
\begin{thm}\label{genholder}[Theorem 2.1 of \cite{Holder}, restated as in Theorem 4.1 of \cite{BGLZ}]
Let $m,n\in\mathbb{Z}^+$. Take $A_1,\ldots,A_m\subseteq [n]$ and $p_1,\ldots,p_m\in\mathbb{R}^{\geq 1}\cup\{\infty\}$. Let $\Omega_1,\ldots,\Omega_n$ be spaces with associated probability measures $\mu_1,\ldots,\mu_n$.For each $i\in [m]$, let $\Omega_{A_i}=\displaystyle\prod_{j\in A_i}\Omega_j$ and $\mu_{A_i}=\displaystyle\prod_{j\in A_i}\mu_j$ for $i\in [m]$, and take some $f_i\in L^{p_i}(\Omega_{A_i},\mu_{A_i})$.

Suppose that $\displaystyle\sum_{i:A_i\ni j}\frac{1}{p_i}\leq 1$ for all $j\in [n]$ (where we take $\frac{1}{\infty}=0$). Then
\[\displaystyle\int_{\prod\Omega_j}\left(\displaystyle\prod_{i=1}^m f_i\right)d\mu_1\cdots d\mu_n\leq \displaystyle\prod_{i=1}^m \|f_i\|_{p_i},\]
where $\|\cdot\|_a$ denotes the $L^a$ norm.
\end{thm}
This result is along the lines of the statement we would like to obtain. Indeed, letting taking $[n]$ in Theorem \ref{genholder} to be the vertex set of $H$, the $A_i$ to be the edges of $H$, $f_i=U$ for all $i$, and finally taking the $\frac{1}{p_i}$ to be some fractional matching of $H$, we do obtain some upper bound on $\Hom(H,U)$.

However, note that if we apply Theorem \ref{genholder} to Example \ref{K23example}, and use the fractional matching mentioned there (with four edges of weight $\frac{1}{2}$ and two of weight $0$), we obtain
\[\Hom(H_2,U)\leq \|U\|_2^4,\]
weaker by about a factor of $p$ than the bound $\Hom(H_2,U)\leq (2+\epsilon)p\|U_2\|_2^4$ we obtained there. The reason is that Theorem \ref{genholder} does not apply the step where we `pulled out' one of the factors of $S_2$, and replaced it by its upper bound. Thus we need to bootstrap Theorem \ref{genholder} to a stronger result which does apply that step, which we now state and prove.
\begin{thm}\label{graphholderthm}
Let $H$ be a graph with no isolated vertices. For all $v\in V(H)$, let $B_v\subseteq [0,1]$ be a measurable subset, and for all $e\in E(H)$, $e=\{v,v'\}$, take $f_e:B_v\times B_{v'}\to\mathbb{R}$ bounded and measurable. Let $(w_e)_{e\in E(H)}$ be a maximum fractional matching and let $(w'_e)_{e\in E(H)}$ be a minimum fractional edge cover such that $w'_e\geq w_e$ for all $e$. Let $S'=\left\{v\in V(H):\displaystyle\sum_{e\ni v}w'_e>1\right\}$.

For $a\in\mathbb{R}^+\cup\{\infty\}$ and $e=\{v,v'\}$, let $S_v^af_e:B_v\to [0,1]$ be given by $S_v^af_e(x_{v})=\left(\displaystyle\int_{B_{v'}}\left|f_e(x_v,x_{v'})\right|^adx_{v'}\right)^{\frac{1}{a}}$ when $a\in\mathbb{R}^+$ and $\displaystyle\esssup_{x_{v'}\in B_{v'}}|f_e(x_v,x_{v'})|$ when $a=\infty$; that is, the outputs of $S_v^af_e$ are the $L^a$ norms of the $v$-columns of $f_e$.

Define $\|f_e\|_{v,a}:=\displaystyle\esssup_{x_v\in B_v}S_v^af_e(x_v)$ for all $a\in\mathbb{R}^+\cup\{\infty\}$. Then we have the following generalized H\"{o}lder inequality.

\begin{equation}\label{graphholder}
\displaystyle\int_{\prod B_v}\displaystyle\prod_{e=\{v,v'\}\in E(H)}f_e(x_v,x_{v'})\prod dx_v\leq\left(\displaystyle\prod_{v\in S'}\displaystyle\prod_{e\ni v}\|f_e\|_{v,\frac{1}{w'_e}}^{\frac{w'_e-w_e}{w'_e}}\right)\displaystyle\prod_{e\in E(H)}\|f_e\|_{\frac{1}{w'_e}}^{\frac{w_e}{w'_e}},
\end{equation}
where we take $\frac{1}{0}=\infty$ and when $w'_e=w_e=0$ we take $\frac{w'_e-w_e}{w'_e}=0$ and $\frac{w_e}{w'_e}=1$.
\end{thm}
\begin{remark}
We have stated the result above in the full generality we will eventually need for the proof of Theorem \ref{varcorrectconstant}. However, for the purposes of Theorem \ref{simpleholderresult} (and thus Theorem \ref{varcorrectexponent}) it suffices to consider $f_e=U$ for all $e\in E(H)$ and $B_v=[0,1]$ for all $v\in V(H)$. The left side of (\ref{graphholder}) then simply becomes $\Hom(H,U)$.

In relation to Example \ref{K23example}, the first term on the right side of (\ref{graphholder}) consists of the copies of $S_i$ (here $S_v^a$) we are pulling out of the integral, and the second term gives us the remaining norms after we have repeatedly applied H\"{o}lder's inequality.

We now have the tools to motivate why graphs with bad edges will cause problems for us when we want to be more precise than logarithmic. In the substitution above ($f_e=U$ for all $e\in E(H)$), the right hand side of (\ref{graphholder}) will contain the product of many terms of the form $\|U\|_{1/w'_e}^{w_e/w'_e}=\left(\int |U|^{1/w'_e}\right)^{w_e}$. If $w'_e<1$, then we can use Lemma \ref{momentlemma} in order to bound this quantity using the entropy $I_p(U+p)=I_p(W)$. Since the equality case of Lemma \ref{momentlemma} is at $x=1$, this will imply that our graphons should take value $1$ in the relevant sections; that is, they should correspond to planting a graph. However, if $w'_e=1$, then we must instead use Lemma \ref{firstmomentlemma} to bound, which loses a logarithmic factor compared to Lemma \ref{momentlemma}, and its equality case occurs when $x$ is on the order of $p$, corresponding to a strategy of `raising the density' of a component of the graphon.

If $H$ has a bad edge $e$, then $w_e=w'_e=1$ by the definition of bad edge (and the fact that $w_e\leq w'_e$), which causes problems for the reasons just described. However, if $H$ has no bad edges, it turns out that a converse holds in the relevant cases (see Lemma \ref{edgeweightslemma2}), and so we will be able to assume $w'_e<1$ for all $e$, making our job much easier.
\end{remark}
We now proceed to prove the inequality.
\begin{proof}[Proof of Theorem \ref{graphholderthm}]
Let $S=\left\{v\in V(H):\displaystyle\sum_{e\ni v}w_e<1\right\}$. We first show three useful claims.
\setcounter{claimcounter}{0}
\begin{claim}\label{noSedges}
There are no edges of $H$ between vertices of $S$.
\end{claim}
\begin{claim}\label{SS'claim}
$S\cap S'=\emptyset$.
\end{claim}
\begin{claim}\label{weightssameclaim}
$w'_e=w_e$ unless $e$ has one vertex in $S$ and one in $S'$.
\end{claim}
\begin{proof}[Proof of Claims]
For Claim \ref{noSedges}, if there were an edge $e_0$ $H$ between two vertices of $S$, we could increase $w_{e_0}$ slightly while still having a fractional matching, contradicting the maximality of $(w_e)$. Thus the first claim holds.

For Claim \ref{SS'claim}, assume for the sake of contradiction that $v\in S\cap S'$. Then $\displaystyle\sum_{e\ni v}(w'_e-w_e)>0$, so there is some $e_0\ni v$ with $w'_{e_0}>w_{e_0}$, say $e_0=vv'$. Since there are no edges between two vertices of $S$ by Claim \ref{noSedges} and $v\in S$, we must have $v'\notin S$. Thus $\sum_{e\ni v'}w_e=1$, so
\[\displaystyle\sum_{e\ni v'}w'_e\geq (w'_{e_0}-w_{e_0})+\displaystyle\sum_{e\ni v}w_e\geq w'_{e_0}-w_{e_0}+1>1.\]
So $v'\in S'$. But by assumption $v\in S'$. So we may decrease $w'_{e_0}$ slightly and still maintain the $\{w'_e\}$ as a fractional edge cover, contradicting minimality. Thus $S\cap S'=\emptyset$.'

For Claim \ref{weightssameclaim}, take any edge $e_0$ with $w'_{e_0}>w_{e_0}$. If $e$ has no vertices in $S$, then taking the weight system that is equal to $w'_e$ when $e\neq e_0$ and $w_e$ when $e=e_0$ is a smaller fractional edge cover than the $w'_e$, contradicting minimality. (This is a fractional edge cover as for all vertices $v\notin S$, the edge cover condition was already met by the $w_e$, and no edges adjacent to vertices $v\in S$ were changed by replacing $w'_{e_0}$ with $w_{e_0}$.) So there must be $v\in e_0\cap S$. Let $e_0=\{v,v'\}$. Then we must show that $v'\in S'$. But this is true by the same chain of inequalities as in the previous paragraph, as $v'\notin S$ since it is adjacent to a vertex of $S$. Therefore, for all edges $e$ with $w'_e>w_e$, $e$ has exactly one vertex in $S'$ and exactly one vertex in $S$.
\end{proof}

Having proved the three claims, our strategy will be to first apply H\"{o}lder's inequality to the vertices in $S$, and then apply the generalized H\"{o}lder's inequality to the remaining graph. We proceed in three steps.
\setcounter{claimcounter}{0}
\begin{step}
H\"{o}lder's inequality with weights $w'_e$.
\end{step}
Since there are no edges between two vertices of $S$ by Claim \ref{noSedges}, we may start by rewriting the left side of (\ref{graphholder}) in the form
\begin{equation}\label{Sseparate}
\displaystyle\int_{\prod_{v\notin S} B_v}\displaystyle\prod_{\substack{e=\{v,v'\}\in E(H) \\ v,v'\notin S}}f_e(x_v,x_{v'})\displaystyle\prod_{v\in S}\left(\displaystyle\int_{B_v}\displaystyle\prod_{\substack{v'\in N(v)\\ e=\{v,v'\}}}f_e(x_v,x_{v'})dx_v\right)\displaystyle\prod_{v\notin S} dx_v,
\end{equation}
by breaking vertices $v\in S$ into their own separate integrals.

Now, since $S\cap S'=\emptyset$ by Claim \ref{SS'claim}, for every $v\in S$ we must have $\displaystyle\sum_{e\ni v}w'_e=1$. Thus we may apply H\"{o}lder's inequality with weights $\frac{1}{w'_e}$ to the inner integral of the expression above. Namely, for all $v\in S$ and fixing $x_{v'}$ for all $v'\in N(v)$,
\[\displaystyle\int_{B_v}\displaystyle\prod_{\substack{v'\in N(v)\\ e=\{v,v'\}}}f_e(x_v,x_{v'})dx_v\leq\displaystyle\prod_{\substack{v'\in N(v)\\ e=\{v,v'\}}}S_{v'}^{\frac{1}{w'_e}}f_e(x_{v'}).\]
\begin{step}
Pulling out copies of $S_v^a$.
\end{step}
We will further bound
\[S_{v'}^{\frac{1}{w'_e}}f_e(x_{v'})=\left(S_{v'}^{\frac{1}{w'_e}}f_e(x_{v'})\right)^{\frac{w_e}{w'_e}}\left(S_{v'}^{\frac{1}{w'_e}}f_e(x_{v'})\right)^{\frac{w'_e-w_e}{w'_e}}\leq\left(S_{v'}^{\frac{1}{w'_e}}f_e(x_{v'})\right)^{\frac{w_e}{w'_e}}\|f_e\|_{v',\frac{1}{w'_e}}^{\frac{w'_e-w_e}{w'_e}},\]
simply replacing $S_{v'}^{\frac{1}{w'_e}}f_e$ by its maximum $\|f_e\|_{v',\frac{1}{w'_e}}$.

Substituting, we may bound (\ref{Sseparate}) by
\begin{equation}\label{Sholder}
\displaystyle\int_{\prod_{v\notin S} B_v}\displaystyle\prod_{\substack{e=\{v,v'\}\in E(H) \\ v,v'\notin S}}f_e(x_v,x_{v'})\displaystyle\prod_{v\in S}\left(\displaystyle\prod_{\substack{v'\in N(v)\\ e=\{v,v'\}}}\left(S_{v'}^{\frac{1}{w'_e}}f_e(x_{v'})\right)^{\frac{w_e}{w'_e}}\|f_e\|_{v',\frac{1}{w'_e}}^{\frac{w'_e-w_e}{w'_e}}\right)\displaystyle\prod_{v\notin S} dx_v.
\end{equation}
Now, consider the subproduct
\[\displaystyle\prod_{v\in S}\displaystyle\prod_{\substack{v'\in N(v)\\ e=\{v,v'\}}}\|f_e\|_{v',\frac{1}{w'_e}}^{\frac{w'_e-w_e}{w'_e}}.\]
Note that when $w'_e=w_e$, the factor is simply equal to $1$ (even when $w'_e=w_e=0$ by our convention chosen). Thus we only need consider the factors when $w'_e>w_e$. But by Claim \ref{weightssameclaim}, $w_e=w'_e$ except when $e$ has exactly one vertex in $S$ and one in $S'$, so we have that
\[\displaystyle\prod_{v\in S}\displaystyle\prod_{\substack{v'\in N(v)\\ e=\{v,v'\}}}\|f_e\|_{v',\frac{1}{w'_e}}^{\frac{w'_e-w_e}{w'_e}}=\displaystyle\prod_{v\in S}\displaystyle\prod_{\substack{v'\in N(v)\cap S'\\ e=\{v,v'\}}}\|f_e\|_{v',\frac{1}{w'_e}}^{\frac{w'_e-w_e}{w'_e}}=\displaystyle\prod_{v'\in S'}\displaystyle\prod_{\substack{v\in N(v')\cap S\\ e=\{v,v'\}}}\|f_e\|_{v',\frac{1}{w'_e}}^{\frac{w'_e-w_e}{w'_e}}=\displaystyle\prod_{v'\in S'}\displaystyle\prod_{\substack{v\in N(v')\\ e=\{v,v'\}}}\|f_e\|_{v',\frac{1}{w'_e}}^{\frac{w'_e-w_e}{w'_e}},\]
where in the middle equality we switched the order of the products (but still iterate over edges in $S\times S'$). 
Thus this product is equal to (replacing the dummy variable $v'$ with $v$)
\[\displaystyle\prod_{v\in S'}\displaystyle\prod_{e\ni v}\|f_e\|_{v,\frac{1}{w'_e}}^{\frac{w'_e-w_e}{w'_e}}.\]
Pulling this product out of the integral, we see that (\ref{Sholder}) is equal to
\[\left(\displaystyle\prod_{v\in S'}\displaystyle\prod_{e\ni v}\|f_e\|_{v,\frac{1}{w'_e}}^{\frac{w'_e-w_e}{w'_e}}\right)\displaystyle\int_{\prod_{v\notin S} B_v}\displaystyle\prod_{\substack{e=\{v,v'\}\in E(H) \\ v,v'\notin S}}f_e(x_v,x_{v'})\displaystyle\prod_{v\in S}\left(\displaystyle\prod_{\substack{v'\in N(v)\\ e=\{v,v'\}}}\left(S_{v'}^{\frac{1}{w'_e}}f_e(x_{v'})\right)^{\frac{w_e}{w'_e}}\right)\displaystyle\prod_{v\notin S} dx_v.\]
Note that we have obtained one of the two terms on the right side of (\ref{graphholder}). So it suffices to show that
\begin{equation}\label{genholderapplication}
\displaystyle\int_{\prod_{v\notin S} B_v}\displaystyle\prod_{\substack{e=\{v,v'\}\in E(H) \\ v,v'\notin S}}f_e(x_v,x_{v'})\displaystyle\prod_{v\in S}\left(\displaystyle\prod_{\substack{v'\in N(v)\\ e=\{v,v'\}}}\left(S_{v'}^{\frac{1}{w'_e}}f_e(x_{v'})\right)^{\frac{w_e}{w'_e}}\right)\displaystyle\prod_{v\notin S} dx_v\leq\displaystyle\prod_{e\in E(H)}\|f_e\|_{\frac{1}{w'_e}}^{\frac{w_e}{w'_e}}.
\end{equation}
\begin{step}
H\"{o}lder's inequality with weights $w_e$.
\end{step}
The proof of (\ref{genholderapplication}) is exactly an application of Theorem \ref{genholder}. Our base set $[n]$ will correspond to the vertices in $V(H)\backslash S$. Let our spaces $\Omega_v$ be exactly the intervals $B_v$ with the measure $\mu_v$ being the standard measure, $v\in V(H)\backslash S$. Our set $[m]$ will be correspond to the edges $e\in E(H)$.

Our subsets $A$ will be given by $A_e=e$ if $e\cap S=\emptyset$ and $A_e=e\backslash\{v\}$ if $e\cap S=\{v\}$. We will take $p_e=\frac{1}{w_e}$ (again with $\frac{1}{0}=\infty$) for all $e\in E(H)$.

The function corresponding to the edge $e$ will simply be the function $f_e$ when $e\cap S=\emptyset$. If $e=vv'$ with $v\in S$ (so that $A_e=\{v'\}$ consists of a single vertex), then we take our function to be $x_{v'}\mapsto\left(S_{v'}^{\frac{1}{w'_e}}f_e(x_{v'})\right)^{\frac{w_e}{w'_e}}$.

Note that for $v\in V(H)\backslash S$ and $e\in E(H)$, $v\in A_e$ if and only if $v\in e$. Thus the necessary condition $\displaystyle\sum_{e:A_e\ni v}\frac{1}{p_e}\leq 1$ simply of Theorem \ref{genholder} simply follows from the fractional matching condition on the $w_e$.

With these specifications, the left side of Theorem \ref{genholder} simply becomes the left side of (\ref{genholderapplication}). The right side becomes
\begin{align}
 & \displaystyle\prod_{\substack{e=\{v,v'\}\in E(H) \\ v,v'\notin S}}\|f_e\|_{\frac{1}{w_e}}\displaystyle\prod_{v\in S}\displaystyle\prod_{\substack{v'\in N(v)\\ e=\{v,v'\}}}\left(\displaystyle\int_{B_{v'}}\left(\left(S_{v'}^{\frac{1}{w'_e}}f_e(x_{v'})\right)^{\frac{w_e}{w'_e}}\right)^{\frac{1}{w_e}}dx_{v'}\right)^{w_e} \\ & \label{holderanswer}=\displaystyle\prod_{\substack{e=\{v,v'\}\in E(H) \\ v,v'\notin S}}\|f_e\|_{\frac{1}{w_e}}\displaystyle\prod_{v\in S}\displaystyle\prod_{\substack{v'\in N(v)\\ e=\{v,v'\}}}\left(\displaystyle\int_{B_{v'}}\left(S_{v'}^{\frac{1}{w'_e}}f_e(x_{v'})\right)^{\frac{1}{w'_e}}dx_{v'}\right)^{w_e}.
\end{align}
Now, substituting the definition of $S_{v'}^{\frac{1}{w'_e}}f_e$, we see that
\begin{align*}
\left(\displaystyle\int_{B_{v'}}\left(S_{v'}^{\frac{1}{w'_e}}f_e(x_{v'})\right)^{\frac{1}{w'_e}}dx_{v'}\right)^{w_e} & =\left(\displaystyle\int_{B_{v'}}\left(\left(\displaystyle\int_{B_v}|f_e(x_v,x_{v'})|^{\frac{1}{w'_e}}dx_v\right)^{w'_e}\right)^{\frac{1}{w'_e}}dx_{v'}\right)^{w_e} \\ & =\left(\displaystyle\int_{B_{v'}}\left(\displaystyle\int_{B_v}|f_e(x_v,x_{v'})|^{\frac{1}{w'_e}}dx_v\right)dx_{v'}\right)^{w_e} \\ & =\left(\displaystyle\int_{B_v\times B_{v'}}|f_e(x_v,x_{v'})|^{\frac{1}{w'_e}}dx_vdx_{v'}\right)^{w_e} \\ & =\|f_e\|_{\frac{1}{w'_e}}^{\frac{w_e}{w'_e}}.
\end{align*}
Thus (\ref{holderanswer}) is equal to
\begin{equation}\label{simplifiedholderanswer}
\displaystyle\prod_{\substack{e=\{v,v'\}\in E(H) \\ v,v'\notin S}}\|f_e\|_{\frac{1}{w_e}}\displaystyle\prod_{v\in S}\displaystyle\prod_{\substack{v'\in N(v)\\ e=\{v,v'\}}}\|f_e\|_{\frac{1}{w'_e}}^{\frac{w_e}{w'_e}}.
\end{equation}
But when $e\cap S=\emptyset$, $w_e=w'_e$ by Claim \ref{weightssameclaim}, so $\|f_e\|_{\frac{1}{w_e}}=\|f_e\|_{\frac{1}{w'_e}}^{\frac{w_e}{w'_e}}$. Since no edge can have two vertices in $S$ by Claim \ref{noSedges}, we may write (\ref{simplifiedholderanswer}) as
\[\displaystyle\prod_{\substack{e\in E(H) \\ e\cap S=\emptyset}}\|f_e\|_{\frac{1}{w'_e}}^{\frac{w_e}{w'_e}}\displaystyle\prod_{\substack{e\in E(H) \\ e\cap S\neq\emptyset}}\|f_e\|_{\frac{1}{w'_e}}^{\frac{w_e}{w'_e}}=\displaystyle\prod_{e\in E(H)}\|f_e\|_{\frac{1}{w'_e}}^{\frac{w_e}{w'_e}},\]
which is the second term of (\ref{graphholder}), so we have proven Theorem \ref{graphholderthm}.
\end{proof}
\section{Proof of Theorem \ref{simpleholderresult}}\label{simpleholderproofsection}
Take $p\to 0$, $\epsilon=\epsilon(p)$ with $0<\epsilon\leq 1$, and let $W$ be a graphon satisfying the conditions of Theorem \ref{simpleholderresult}. That is, for all $x_0\in [0,1]$, $(1-\epsilon)p\leq\displaystyle\int_0^1 W(x_0,y)\leq (1+\epsilon)p$. Let $U=W-p$.

We apply Theorem \ref{graphholderthm}, with the following parameters. Let $H$ be any graph with no isolated vertices, and for all $e\in E(H)$ we take the interval $B_v=[0,1]$. Let $f_e=|U|$ for all $e\in E(H)$. Finally, let $(w_e)_{e\in E(H)}$ and $(w'_e)_{e\in E(H)}$ be respectively any maximum fractional matching and any minimum fractional edge cover of $H$ such that $w_e\leq w'_e$ for all $e\in E(H)$. (The fact that these edge weightings exist will be shown later.)

The left side of (\ref{graphholder}) then simply becomes $\Hom(H,|U|)$. We bound the terms on the right side.

Unpacking the definition of $\| f_e\|_{v,\frac{1}{w'_e}}^{\frac{w'_e-w_e}{w'_e}}$ with our setting $f_e=|U|$, we see that it is equal to
\[\displaystyle\esssup_{x_0\in [0,1]}\left(\displaystyle\int_0^1 |U(x_0,y)|^{\frac{1}{w'_e}}dy\right)^{w'_e-w_e}\]
when $w'_e\neq 0$ and $1$ when $w'_e=w_e=0$. Now, since our fractional edge cover $(w'_e)$ is minimum, $w'_e\leq 1$ for all $e\in E(H)$. Since $U$ is $1$-bounded, we thus have
\begin{align*}
\displaystyle\esssup_{x_0\in [0,1]}\left(\displaystyle\int_0^1 |U(x_0,y)|^{\frac{1}{w'_e}}dy\right)^{w'_e-w_e} & \leq\displaystyle\esssup_{x_0\in [0,1]}\left(\displaystyle\int_0^1 |U(x_0,y)| dy\right)^{w'_e-w_e} \\ & =\displaystyle\esssup_{x_0\in [0,1]}\left(\displaystyle\int_0^1 |W(x_0,y)-p| dy\right)^{w'_e-w_e} \\ & \leq\displaystyle\esssup_{x_0\in [0,1]}\left(\displaystyle\int_0^1 (W(x_0,y)+p) dy\right)^{w'_e-w_e} \\ & \leq ((2+\epsilon)p)^{w'_e-w_e}
\end{align*}
by our conditions on $W$ (the same bound also trivially holds when $w'_e=w_e=0$).

The $\|f_e\|_{\frac{1}{w'_e}}^{\frac{w_e}{w'_e}}$ terms are easier, as they are equal to
\begin{align*}
\left(\displaystyle\int_{[0,1]^2}|U(x,y)|^{\frac{1}{w'_e}}dx dy\right)^{w_e} & \leq\left(\displaystyle\int_{[0,1]^2}|U(x,y)| dx dy\right)^{w_e} \\ & =\mathbb{E}(|U|)^{w_e}
\end{align*}
when $w'_e>0$ (using that $w_e\leq w'_e\leq 1$ for all $e\in E(H)$ and that $U$ is $1$-bounded), and the same bound again easily holds when $w'_e=w_e=0$. Substituting, Theorem \ref{graphholderthm} gives that
\[\Hom(H,|U|)\leq ((2+\epsilon)p)^{\displaystyle\sum_{v\in S'}\displaystyle\sum_{e\ni v}(w'_e-w_e)}\mathbb{E}(|U|)^{\displaystyle\sum_{e\in E(H)}w_e},\]
where $S'=\left\{v\in V(H):\displaystyle\sum_{e\ni v}w'_e>1\right\}$.

By Claims \ref{SS'claim} and \ref{weightssameclaim} in the proof of Theorem \ref{graphholderthm}, if $e\in E(H)$ does not have exactly one vertex in $S'$, then $w_e=w'_e$. Thus the sum $\sum_{v\in S'}\sum_{e\ni v}(w'_e-w_e)$ counts each edge where $w'_e\neq w_e$ exactly once, so this sum simply equals $\sum_{e\in E(H)}w'_e-\sum_{e\in E(H)} w_e$. Substituting,
\[\Hom(H,|U|)\leq ((2+\epsilon)p)^{\sum w'_e-\sum w_e}\mathbb{E}(|U|)^{\sum w_e},\]
where all sums are over $e\in E(H)$. Thus to finish the proof of Theorem \ref{simpleholderresult}, it suffices to show the following lemma.
\begin{lemma}\label{edgeweightslemma}
Let $H$ be a graph with no isolated vertices. Then the following statements hold.
\begin{enumerate}
\item Let $(w_e)_{e\in E(H)}$ be a maximum fractional matching. Then there is some minimum fractional edge cover $(w'_e)_{e\in E(H)}$ such that $w_e\leq w'_e$ for all $e\in E(H)$.
\item Let $(w'_e)_{e\in E(H)}$ be a minimum fractional edge cover. Then there is some maximum fractional matching $(w_e)_{e\in E(H)}$ such that $w_e\leq w'_e$ for all $e\in E(H)$.
\item The fractional matching number of $H$ (that is, the maximum total weight of a fractional matching) is equal to the fractional vertex cover number $c(H)$.
\item The fractional edge cover number of $H$ (that is, the minimum total weight of a fractional edge cover) is equal to $v(H)-c(H)$.
\end{enumerate}
\end{lemma}
\begin{proof}[Proof of Lemma \ref{edgeweightslemma}]
(3) is a consequence of the strong duality theorem for linear programming, as the dual system of fractional vertex cover is simply fractional matching.

We prove (1), (2), and (4) simultaneously. Let $w_e$ be a maximum fractional matching. Let $S$ be the set of vertices $v\in V(H)$ such that $\displaystyle\sum_{e\ni v}w_e<1$. Then there are no edges between vertices of $S$, as if there was such an edge we could increase its weight slightly to obtain a larger fractional matching.

For each $v\in S$, we modify the weights of the edges $e\ni v$ as follows. If $\displaystyle\sum_{e\ni v}w_e>0$, take all edges $e\ni v$ and increase each of their weights by $\frac{1}{\deg(v)}\left(1-\displaystyle\sum_{e\ni v}w_e\right)$.

Call this new set of weights $w'_e$. Note that now $\displaystyle\sum_{e\ni v}w'_e\geq 1$ for all $v$ (as if this sum was originally less than $1$ it was increased to $1$ by the process above), so the $w'_e$ form a fractional edge cover with $w_e\leq w'_e$ for all $e$. However, we do not know yet that $(w'_e)$ is minimum.

Since we increased weights of edges adjacent to a vertex $v\in S$ by a total of $1-\displaystyle\sum_{e\ni v}w_e$, we have that
\[\displaystyle\sum_{e\in E(H)}(w'_e-w_e)=\displaystyle\sum_{v\in V(H)}\left(1-\displaystyle\sum_{e\ni v}w_e\right)=v(H)-2\displaystyle\sum_{e\in E(H)}w_e=v(H)-2c(H),\]
as each edge $w_e$ occurs twice in the nested sum and applying (3). Thus $\displaystyle\sum_{e\in E(H)}w'_e=v(H)-c(H)$, so letting $w(H)$ be the fractional edge cover number, we have that $w(H)\leq v(H)-c(H)$.

\vspace{12pt}

In the reverse direction, let $w'_e$ be a minimum fractional edge cover. Let $S'$ be the set of vertices $v\in V(H)$ such that $\displaystyle\sum_{e\ni v}w_e>1$. Then all edges $e$ between two vertices of $S'$ have weight $w'_e=0$, as otherwise we could decrease $w'_e$ slightly and obtain a smaller fractional edge cover.

We again modify the edges adjacent to each $v\in S'$, although as we must be slightly more careful to ensure weights do not drop below $0$, we will scale multiplicatively. For each $v\in S'$ and $e\ni v$, replace $w'_e$ by $\frac{w'_e}{\displaystyle\sum_{e\ni v}w'_e}$. Now, each edge's weight is only modified once by this process, as any edge between two vertices of $S'$ has weight $0$ and thus is not modified at all. Let the new weights be $w_e$. Then it is easy to see that $\displaystyle\sum_{e\ni v}w_e\leq 1$ for all $v\in V(H)$, so the $w_e$ form a fractional matching with $w_e\leq w'_e$ for all $e$. However, we (similarly to before) do not yet know that $(w_e)$ is maximum.

Since at each vertex $v\in S'$ we replaced the weights of the edges containing $v$ with weights that summed to $1$, we have that
\[\displaystyle\sum_{e\in E(H)}(w'_e-w_e)=\displaystyle\sum_{v\in V(H)}\left(\displaystyle\sum_{e\ni v}w_e-1\right)=2\displaystyle\sum_{e\in E(H)}w_e-v(H)=2w(H)-v(H).\]
Thus $\displaystyle\sum_{e\in E(H)}w_e=v(H)-w(H)$, and so $c(H)\geq v(H)-w(H)$.

Combining this with the earlier statement that $w(H)\leq v(H)-c(H)$ gives the conclusion of (4); that is, that $w(H)=v(H)-c(H)$. Furthermore, this implies that the fractional edge cover and fractional matching we constructed by these processes are respectively minimum and maximum, proving (1) and (2) as well. This concludes the proof of the Lemma.
\end{proof}
This completes the proof of Theorem \ref{simpleholderresult} and thus the lower bound of Theorem \ref{varcorrectexponent}.
\section{$K_{2,3}$ Revisited}\label{K23section}

In this section, we will attempt to motivate some of the techniques that go into proving the stronger lower bounds given in Theorem \ref{varcorrectconstant}. We will let $K=K_{2,3}$ and prove Theorem \ref{varcorrectconstant} in this case.

It is not difficult to compute that $\gamma(K_{2,3})=\frac{1}{2}$ and $P_{K_{2,3}}(z,w)=1+z^2$. Thus
\[\rho(K_{2,3},\delta):=\min_{\substack{z,w\geq 0 \\ P_{K_{2,3}}(z,w)\geq 1+\delta}}(2z+w)=2\sqrt{\delta}.\]
So for the graph $K=K_{2,3}$, the lower bound of Theorem \ref{varcorrectconstant} becomes the following proposition.
\begin{prop}\label{K23prop}
Take a constant $\delta>0$ and take $p\to 0$. Then
\[\displaystyle\inf_{\substack{\Hom(K_{2,3},W)\geq (1+\delta)p^6 \\ W\text{ }p\text{-regular}}}I_p(W)\geq (2\sqrt{\delta}-o(1))p^{\frac{5}{2}}\log\frac{1}{p}.\]
\end{prop}
To prove this proposition, we will use the following setup.
\begin{setup}\label{K23setup}
Fix a constant $\delta>0$ and take $p\to 0$. Let $\epsilon=\epsilon(p)=1/\log(1/p)$. Recall from Definition \ref{nicegraphonsdef} the definition of $\Gamma_{\epsilon}(K_{2,3},\delta)$. Using Lemma \ref{replacebyplemma}, take some graphon $W\in\Gamma_{\epsilon}(K_{2,3},\delta)$ such that $I_p(W)\leq (1+o(1))\displaystyle\inf_{\substack{\Hom(K_{2,3},W')\geq (1+\delta)p^6 \\ W\text{ }p\text{-regular}}}I_p(W')$.
\end{setup}
Notice that since $I_p(W)$ is at most $1+o(1)$ times the infimum on the left side of Proposition \ref{K23prop}, to prove that proposition it suffices to show that under Setup \ref{K23setup}, $I_p(W)\geq (2\sqrt{\delta}-o(1))p^{\frac{5}{2}}\log\frac{1}{p}$.

We first prove some simple consequences of this setup.
\begin{lemma}\label{K23consequenceslemma}
Under Setup \ref{K23setup}, the following properties hold.
\begin{enumerate}
\item $\displaystyle\int_0^1 W(x_0,y) dy=(1+o(1))p$ for all $x_0\in [0,1]$
\item $\Hom(K_{2,3},W)\geq (1+\delta-o(1))p^{6}$.
\item $W$ takes no values in $[p-p^{1+o(1)},p+p^{1+o(1)}]\backslash p$.
\item $\Hom(H,|U|)\leq 64(1+\delta)p^{e(H)}$ for all $H\subseteq K_{2,3}$.
\item $I_p(W)=O\left(p^{\frac{5}{2}}\log\frac{1}{p}\right)$
\end{enumerate}
\end{lemma}
\begin{proof}
Since $W\in\Gamma_{1/\log(1/p)}(K_{2,3},\delta)$, (1) through (4) follow from recalling Definition \ref{nicegraphonsdef} and noting that $\epsilon:=1/\log(1/p)$ satisfies $\epsilon=o(1)$ and $\epsilon=p^{o(1)}$.

(5) follows from the fact that $I_p(W)\leq (1+o(1))\displaystyle\inf_{\substack{\Hom(K_{2,3},W')\geq (1+\delta)p^6 \\ W'\text{ }p\text{-regular}}}I_p(W')$ and the upper bound of Theorem \ref{varcorrectexponent}.
\end{proof}
We show the following lemma, eliminating the contributions from all subgraphs of $K_{2,3}$ that are not $\emptyset$ and $K_{2,3}$ itself. (Not coincidentally, one can check that these are exactly the contributing subgraphs of $K_{2,3}$.)
\begin{lemma}\label{K23subgraphbreakdown}
Under Setup \ref{K23setup},
\[\Hom(K_{2,3},U)\geq (\delta-o(1))p^6.\]
\end{lemma}
\begin{proof}
By Lemma \ref{subgraphbreakdown} and (2) of Lemma \ref{K23consequenceslemma},
\[1+\delta-o(1)\leq p^{-6}\Hom(K_{2,3},W)=\displaystyle\sum_{H\subseteq K_{2,3}}p^{-e(H)}\Hom(H,U).\]
Let $H\subseteq K_{2,3}$ be any subgraph. By (1) and (3) of Lemma \ref{K23consequenceslemma}, the conditions of Corollary \ref{simpleholdercor} are satisfied with $\epsilon=\left(\log\frac{1}{p}\right)^{-1}$. Since $\epsilon=p^{o(1)}$ and $I_p(W)=O(p^{\frac{5}{2}-o(1)})$  by (5) of Lemma \ref{K23consequenceslemma}, applying Corollary \ref{simpleholdercor} we obtain
\[\Hom(H,U)=O\left(p^{v(H)-2c(H)+\frac{5}{2}c(H)-o(1)}\right)=O\left(p^{v(H)+\frac{1}{2}c(H)-o(1)}\right).\]
Thus if $e(H)<v(H)+\frac{c(H)}{2}$, $p^{-e(H)}\Hom(H,U)=o(1)$. (Notice that $e(H)=v(H)+\frac{c(H)}{2}$ is exactly one of the conditions to be contributing!)

It is easy to see that $e(H)\leq v(H)$ for all $H\subsetneq K_{2,3}$. If $H\neq\emptyset$, $c(H)>0$, so if $H\not\in\{\emptyset,K_{2,3}\}$, $e(H)<v(H)+\frac{c(H)}{2}$. Therefore, $p^{-e(H)}\Hom(H,U)=o(1)$ unless $H=\emptyset$ or $H=K_{2,3}$. When $H=\emptyset$, clearly $p^{-e(H)}\Hom(H,U)=1$. Therefore,
\begin{align*}
1+\delta-o(1) & \leq\displaystyle\sum_{H\subseteq K_{2,3}}p^{-e(H)}\Hom(H,U) \\ & =1+o(1)+p^{-6}\Hom(K_{2,3},U).
\end{align*}
Rearranging, the lemma follows.
\end{proof}
Now that we have a lower bound for $\Hom(K_{2,3},U)$, we would like to improve on our argument in Example \ref{K23example}. There are some simple improvements we can make, but they will not be enough, as we will now demonstrate. In Example \ref{K23example} we showed (implicitly) that
\[\Hom(K_{2,3},U)\leq\| U\|_2^4\displaystyle\sup_{x\in [0,1]}\left(\displaystyle\int_0^1 U(x,y)^2 dy\right).\]
Now, instead of bounding $U^2\leq U$ (using the $1$-boundedness of $U$), as we did in that example, we may obtain better bounds if we simply keep the $U^2$. For example, instead of Lemma \ref{firstmomentlemma}, we may use the stronger Lemma \ref{momentlemma} to show that $I_p(W)\geq (1+o(1))\| U\|_2^2\log\frac{1}{p}$. Furthermore, for all $x\in [0,1]$,
\begin{align*}
\displaystyle\int_0^1 U(x,y)^2 dy & =\displaystyle\int_0^1 (W(x,y)-p)^2 dy \\ & =\displaystyle\int_0^1 W(x,y)^2 dy-2p\displaystyle\int_0^1 W(x,y) dy+p^2 \\ & \leq (1-2p)\displaystyle\int_0^1 W(x,y)+p^2=(1+o(1))p.
\end{align*}
Substituting into our bound yields
\[\Hom(H,U)\leq (1+o(1))p\left(\frac{I_p(W)}{\log\frac{1}{p}}\right)^2,\]
and using Lemma \ref{K23subgraphbreakdown} and rearranging we obtain
\[I_p(W)\geq (\sqrt{\delta}-o(1))p^{\frac{5}{2}}\log\frac{1}{p}.\]
As we can see, we are off by a factor of $2$ from the desired result.

The fundamental issue is that the symmetry of the graphon forces the `modified hub' in the first diagram of Figure \ref{constantgraphonfigure} to contribute \emph{twice} to the entropy, not just once, and we have not accounted for this fact anywhere in our argument. It therefore becomes very important to prove that most of the entropy comes from a small hub that will be forced to contribute twice.

We now define our `hub'.
\begin{definition}\label{K23hubdef}
Under Setup \ref{K23setup}, for any $b=b(p)$, let
\[B_b=\left\{x\in [0,1]:\displaystyle\int_0^1 U(x,y)^2 dy\geq b\right\}.\]
\end{definition}
Now, we would like to show that for almost all homomorphisms from $K_{2,3}$ to $U$, the two vertices $v_1,v_2$ of degree $3$ are sent into the hub $B_b$ for appropriate $b$.
\begin{lemma}\label{K23hublemma}
Assume Setup \ref{K23setup}. If $b\ll p$, then
\[\displaystyle\int_{\substack{x_{v_1}\in [0,1]\backslash B_b \\ x_{v_2},x_{w_1},x_{w_2},x_{w_3}\in [0,1]}}\displaystyle\prod_{i,j}\left|U(x_{v_i},x_{w_j})\right|dx_{v_1}dx_{v_2}dx_{w_1}dx_{w_2}dx_{w_3}=o(p^6),\]
where $B_b$ is as in Definition \ref{K23hubdef}.
\end{lemma}
Note that this integral is (modulo absolute value signs) the contribution to the homomorphism count when $v_1$ is not sent into the hub $B_b$.
\begin{proof}[Proof of Lemma \ref{K23hublemma}]
We may apply either the Cauchy-Schwarz argument from earlier or (equivalently) Theorem \ref{graphholderthm} with $H=K_{2,3}$, $f_e=|U|$ for all $e\in E(H)$,  $w'_e=\frac{1}{2}$ for all $e$, $w_e=\frac{1}{2}$ on a $4$-cycle and $0$ otherwise, and $B_{v_1}=[0,1]\backslash B_b$ and the other $B_v=[0,1]$. We obtain that the given integral is bounded above by
\begin{align*}
& \displaystyle\sup_{x_{v_1}\in [0,1]\backslash B_b}\left(\displaystyle\int_0^1 U(x_{v_1},y)^2 dy\right)^{\frac{1}{2}}\displaystyle\sup_{x_{v_2}\in [0,1]}\left(\displaystyle\int_0^1 U(x_{v_2},y)^2 dy\right)^{\frac{1}{2}} \\ & \cdot\displaystyle\int_{[0,1]\backslash B_b\times [0,1]} U(x,y)^2 dx dy\displaystyle\int_{[0,1]\times [0,1]} U(x,y)^2 dx dy.
\end{align*}
By the argument outlined earlier, $\displaystyle\int_0^1 U(x_0,y)^2 dy\leq p+o(p)$ for all $x_0\in [0,1]$, and by the definition of $B_b$ it is at most $b$ when $x_0\in [0,1]\backslash B_b$. Thus (replacing the third integral by one over $[0,1]^2$, we have an upper bound of
\begin{equation}\label{sqrtbp}
\sqrt{bp}\left(\displaystyle\int_{[0,1]^2}U(x,y)^2 dx dy\right)^2.
\end{equation}
By Lemma \ref{momentlemma} and (3) of Lemma \ref{K23consequenceslemma},
\[I_p(W)\geq (1-o(1))\displaystyle\int_{[0,1]^2}U(x,y)^2 dx dy\log\frac{1}{p}.\]
By (5) of Lemma \ref{K23consequenceslemma}, $I_p(W)=O\left(p^{\frac{5}{2}}\log\frac{1}{p}\right)$, so
\[\displaystyle\int_{[0,1]^2}U(x,y)^2 dx dy\leq (1+o(1))p^{\frac{5}{2}}.\]
Thus (\ref{sqrtbp}) is bounded above by $\sqrt{b}p^{\frac{11}{2}}$. Since $b\ll p$ by assumption, the Lemma follows.
\end{proof}
We will now show our lower bound on entropy from earlier, but this time restricted to the hub.
\begin{lemma}\label{hubentropylemma}
Assume Setup \ref{K23setup} and let $b\ll p$. Then
\[\displaystyle\int_{B_b\times [0,1]}I_p(W(x,y)) dx dy\geq (\sqrt{\delta}-o(1))p^{\frac{5}{2}}\log\frac{1}{p},\]
where $B_b$ is as in Definition \ref{K23hubdef}.
\end{lemma}
\begin{proof}
Since $v_1$ and $v_2$ play symmetric roles, Lemma \ref{K23hublemma} is true with $v_1$ and $v_2$ switched. Applying Lemmas \ref{K23subgraphbreakdown} and \ref{K23hublemma} thus yields
\begin{align*}
& \displaystyle\int_{\substack{x_{v_1},x_{v_2}\in B_b \\ x_{w_1},x_{w_2},x_{w_3}\in [0,1]}}\displaystyle\prod_{i,j} U(x_{v_i},x_{w_j})dx_{v_1}dx_{v_2}dx_{w_1}dx_{w_2}dx_{w_3} \\ & =\Hom(H,U)-\displaystyle\int_{\substack{x_{v_1}\in B_b, x_{v_2}\in [0,1]\backslash B_b \\ x_{w_1},x_{w_2},x_{w_3}\in [0,1]}}\displaystyle\prod_{i,j} U(x_{v_i},x_{w_j})dx_{v_1}dx_{v_2}dx_{w_1}dx_{w_2}dx_{w_3} \\ & -\displaystyle\int_{\substack{x_{v_1}\in[0,1]\backslash B_b,  \\ x_{v_2},x_{w_1},x_{w_2},x_{w_3}\in [0,1]}}\displaystyle\prod_{i,j} U(x_{v_i},x_{w_j})dx_{v_1}dx_{v_2}dx_{w_1}dx_{w_2}dx_{w_3} \\ & \geq (\delta-o(1))p^6-\displaystyle\int_{\substack{x_{v_2}\in [0,1]\backslash B_b \\ x_{v_1},x_{w_1},x_{w_2},x_{w_3}\in [0,1]}}\displaystyle\prod_{i,j} |U(x_{v_i},x_{w_j})|dx_{v_1}dx_{v_2}dx_{w_1}dx_{w_2}dx_{w_3} \\ & -\displaystyle\int_{\substack{x_{v_1}\in [0,1]\backslash B_b \\ x_{v_2},x_{w_1},x_{w_2},x_{w_3}\in [0,1]}}\displaystyle\prod_{i,j} |U(x_{v_i},x_{w_j})|dx_{v_1}dx_{v_2}dx_{w_1}dx_{w_2}dx_{w_3} \\ & \geq (\delta-o(1))p^6.
\end{align*}
So we have successfully proven a bound with $v_1$ and $v_2$ required to be sent into the hub $B_b$.

Let $S_2(x)=\left(\displaystyle\int_0^1 U(x,y)^2 dy\right)^{\frac{1}{2}}$. Since
\begin{align*}
\displaystyle\int_0^1 U(x,y)^2 dy & =\displaystyle\int_0^1 W(x,y)^2 dy-2p\displaystyle\int_0^1 W(x,y) dy+p^2 \\ & \leq (1-2p)\displaystyle\int_0^1 W(x,y) dy \\ & =p+o(p)
\end{align*}
by (1) of Lemma \ref{K23consequenceslemma}, we have that $S_2(x)\leq (1+o(1))\sqrt{p}$ for all $x\in [0,1]$.

Now we apply Cauchy-Schwarz (or Theorem \ref{graphholderthm}) in a similar way to Example \ref{K23example} and the previous lemma. In particular,
\begin{align*}
(\delta-o(1))p^6 & \leq\displaystyle\int_{\substack{x_{v_1},x_{v_2}\in B_b \\ x_{w_1},x_{w_2},x_{w_3}\in [0,1]}}\displaystyle\prod_{i,j} U(x_{v_i},x_{w_j})dx_{v_1}dx_{v_2}dx_{w_1}dx_{w_2}dx_{w_3} \\ & \leq\displaystyle\int_{B_b}S_2(x_{v_1})^3 dx_{v_1}\displaystyle\int_{B_b}S_2(x_{v_2})^3 dx_{v_2} \\ & \leq p\displaystyle\int_{B_b}S_2(x_{v_1})^2 dx_{v_1}\displaystyle\int_{B_b}S_2(x_{v_2})^2 dx_{v_2} \\ & =p\left(\displaystyle\int_{B_b\times [0,1]}U(x,y)^2 dx dy\right)^2,
\end{align*}
so
\[\displaystyle\int_{B_b\times [0,1]}U(x,y)^2 dx dy\geq (\sqrt{\delta}-o(1))p^{\frac{5}{2}}\]
By Lemma \ref{momentlemma} and (3) of Lemma \ref{K23consequenceslemma}, $I_p(W(x,y))\geq (1-o(1))U(x,y)^2\log\frac{1}{p}$ for all $x,y\in [0,1]$. Substituting,
\[\displaystyle\int_{B_b\times [0,1]}I_p(W(x,y)) dx dy\geq (\sqrt{\delta}-o(1))p^{\frac{5}{2}}\log\frac{1}{p},\]
as desired.
\end{proof}
Since $W$ is symmetric, the integral in Lemma \ref{hubentropylemma} contributes twice to the total entropy of $W$, once on $B_b\times [0,1]$ and once on $[0,1]\times B_b$. Thus we will be done if we can show that the contribution of the overlap $B_b\times B_b$ to the total entropy is small. We do this simply by showing that $B_b$ has small measure.
\begin{lemma}
Under Setup \ref{K23setup}, for any $b\geq 0$,
\[m(B_b)=O\left(\frac{p^{\frac{5}{2}}}{b}\right),\]
with $B_b$ as in Definition \ref{K23hubdef}.
\end{lemma}
\begin{proof}
Integrating the definition of $B_b$,
\[\displaystyle\int_{B_b\times [0,1]}U(x,y)^2 dx dy\geq m(B_b)b.\]
As shown in the previous lemma, for all $x,y$, $I_p(W(x,y))\geq (1-o(1))U(x,y)^2\log\frac{1}{p}$. Thus
\begin{align*}
m(B_b)b & \leq\displaystyle\int_{B_b\times [0,1]}U(x,y)^2 dx dy \\ & \leq(1-o(1))\left(\log\frac{1}{p}\right)^{-1}\displaystyle\int_{B_b\times [0,1]}I_p(W(x,y))dx dy \\ & \leq (1-o(1))\left(\log\frac{1}{p}\right)^{-1}I_p(W) \\ & =O(p^{\frac{5}{2}}),
\end{align*}
with the last step by (5) of Lemma \ref{K23consequenceslemma}, finishing the proof.
\end{proof}
\begin{proof}[Proof of Proposition \ref{K23prop}]
We are given $\delta$ and $p\to 0$. Take $\epsilon=\left(\log\frac{1}{p}\right)^{-1}$ and $W$ as in Setup \ref{K23setup}. We are now in the situation of Setup \ref{K23setup}.

Further take some $b$ with $p^{\frac{5}{4}}\ll b\ll p$ and recall from Definition \ref{K23hubdef} the definition of $B_b$.

Combining the previous two lemmas, and using the fact that $I_p(x)\leq\log\frac{1}{p}$ for all $x\in [0,1]$,
\begin{align*}
I_p(W) & =\displaystyle\int_{[0,1]\times [0,1]}I_p(W(x,y)) dx dy \\ & \geq\displaystyle\int_{B_b\times [0,1]}I_p(W(x,y)) dx dy+\displaystyle\int_{B_b\times [0,1]}I_p(W(x,y)) dx dy-\displaystyle\int_{B_b\times B_b}I_p(W(x,y)) dx dy \\ & \geq (2\sqrt\delta-o(1))p^{\frac{5}{2}}\log\frac{1}{p}-m(B_b)^2\log\frac{1}{p} \\ & \geq (2\sqrt\delta-o(1))p^{\frac{5}{2}}\log\frac{1}{p}-O\left(b^{-2}p^5\log\frac{1}{p}\right) \\ & \geq(2\sqrt\delta-o(1))p^{\frac{5}{2}}\log\frac{1}{p},
\end{align*}
where the last step is because we chose $b\gg p^{\frac{5}{4}}$. By our choice of $W$ (as per Setup \ref{K23setup}), we know that
\[I_p(W)\leq (1+o(1))\displaystyle\inf_{\substack{\Hom(K_{2,3},W')\geq (1+\delta)p^6 \\ W\text{ }p\text{-regular}}}I_p(W').\]
Thus $\displaystyle\inf_{\substack{\Hom(K_{2,3},W')\geq (1+\delta)p^6 \\ W\text{ }p\text{-regular}}}I_p(W')\geq (2\sqrt{\delta}-o(1))p^{\frac{5}{2}}\log\frac{1}{p}$, and we are done.
\end{proof}
Notice the importance in this example of determining which vertices are in our `hub' $B_b$.
\section{Separating Out the Hub}\label{constantsection1}
In this and the next section, we prove the lower bound of Theorem \ref{varcorrectconstant}. We have the following setup, which we will use for the next two sections.
\begin{setup}\label{constantsetup}
Let $K$ be some nonforest graph none of whose contributing subgraphs have bad edges and whose $2$-core is not a disjoint union of cycles. Let $\gamma:=\gamma(K)$ and notice that $\gamma>0$, as if $H$ is the $2$-core of $K$, $e(H)>v(H)$. Fix a constant $\delta>0$, and take $p\to 0$.

Recall from Definition \ref{nicegraphonsdef} the definition of $\Gamma_{\epsilon}(K,\delta)$. We will let our $\epsilon$ be $1/\log(1/p)$ (really, we just need $\epsilon=o(1)$ and $\epsilon=p^{o(1)}$). By Lemma \ref{replacebyplemma}, for any $p<\frac{1}{e}$ we may find $W=W(p)\in\Gamma_{1/\log(1/p)}(K,\delta)$ such that
\[I_p(W)\leq (1+o(1))\displaystyle\inf_{\substack{\Hom(K,W')\geq (1+\delta)p^{e(K)} \\ W'\text{ }p\text{-regular}}}I_p(W').\]
Let $U=W-p$.
\end{setup}
Note that under this setup, since $I_p(W)\leq (1+o(1))\displaystyle\inf_{\substack{\Hom(K,W')\geq (1+\delta)p^{e(K)} \\ W'\text{ }p\text{-regular}}}I_p(W')$, to prove the upper bound of Theorem \ref{varcorrectconstant} it suffices to show that
\begin{equation}\label{desiredconstantequation}
I_p(W)\geq (2-o(1))\rho(K,\delta)p^{2+\gamma}\log\frac{1}{p}.
\end{equation}

We state some of the simple consequences of Setup \ref{constantsetup} in the following lemma, exactly paralleling Lemma \ref{K23consequenceslemma}.
\begin{lemma}\label{constantconsequenceslemma}
Under Setup \ref{constantsetup}, the following properties hold.
\begin{enumerate}
\item $\displaystyle\int_0^1 W(x_0,y) dy=(1+o(1))p$ for all $x_0\in [0,1]$
\item $\Hom(K,W)\geq (1+\delta-o(1))p^{e(K)}$.
\item $W$ takes no values in $[p-p^{1+o(1)},p+p^{1+o(1)}]\backslash p$.
\item $\Hom(H,|U|)\leq 2^{e(K)}(1+\delta)p^{e(H)}$ for all $H\subseteq K$.
\item $I_p(W)=O\left(p^{2+\gamma}\log\frac{1}{p}\right)$.
\item $\|U\|_1=O\left(p^{2+\gamma}\left(\log\frac{1}{p}\right)^2\right)$.
\end{enumerate}
\end{lemma}
\begin{proof}
Since $W\in\Gamma_{1/\log(1/p)}(K,\delta)$, (1) through (4) follow from recalling Definition \ref{nicegraphonsdef} and noting that $\epsilon:=\log\frac{1}{p}$ satisfies $\epsilon=o(1)$ and $\epsilon=p^{o(1)}$.

(5) follows from the fact that $I_p(W)\leq (1+o(1))\displaystyle\inf_{\substack{\Hom(K,W')\geq (1+\delta)p^{e(K)} \\ W'\text{ }p\text{-regular}}}I_p(W')$ and the upper bound of Theorem \ref{varcorrectexponent}.

To prove (6), we note more specifically that we chose $\epsilon=1/\log(1/p)$. Thus $|U|$ takes no values in $(0,p/\log(1/p))$ by (3) of Definition \ref{nicegraphonsdef}. Thus by Lemma \ref{firstmomentlemma} with $\epsilon=1/\log(1/p)$, there is some absolute constant $C$ such that $|U(x,y)|=C\log(1/p)I_p(W(x,y)))$. Integrating over $[0,1]^2$ and applying (5),
\[\|U\|_1\leq C\log\frac{1}{p}I_p(W)=O\left(p^{2+\gamma}\left(\log\frac{1}{p}\right)^2\right),\]
as desired.
\end{proof}
We proceed with the following simple lemma.
\begin{lemma}\label{subgraphbreakdownresult}
Under Setup \ref{constantsetup},
\[\displaystyle\sum_{H\subseteq K}p^{-e(H)}\Hom(H,U)\geq 1+\delta-o(1).\]
\end{lemma}
\begin{proof}
Recalling that $U=W-p$, Lemma \ref{subgraphbreakdown} implies that
\[p^{-e(K)}\Hom(K,W)=\displaystyle\sum_{H\subseteq K}p^{-e(H)}\Hom(H,U).\]
Now, by (2) of Lemma \ref{constantconsequenceslemma},
\[\Hom(K,W)\geq (1+\delta-o(1))p^{e(K)}.\]
Substituting yields the desired result.
\end{proof}
We would like to determine which of the terms on the left side of Lemma \ref{subgraphbreakdownresult} contribute non-negligibly to the sum. The following two lemmas combined restrict us to the contributing subgraphs.
\begin{lemma}\label{degree1lemma}
Under Setup \ref{constantsetup}, if $H\subseteq K$ has some vertex $v_0$ of degree $1$, then
\[\Hom(H,U)=o\left(p^{e(H)}\right).\]
\end{lemma}
\begin{proof}
By (1) and (4) of Lemma \ref{constantconsequenceslemma},
\[\left|\displaystyle\int_0^1 U(x_0,y)dy\right|=o(p)\]
for all $x_0\in [0,1]$, and
\[\Hom(H',|U|)=O(p^{e(H)})\]
for all $H'\subseteq K$. Therefore, letting $w_0$ be the unique neighbor of $v_0$ and $H_0$ be the graph given by removing $v_0$ and the edge $v_0w_0$ from $H$,
\begin{align*}
\Hom(H,U) & =\displaystyle\int_{[0,1]^{v(H)}}\displaystyle\prod_{vw\in E(H)}U(x_v,x_w)\displaystyle\prod_{v\in V(H)}dx_v \\ & =\displaystyle\int_{[0,1]^{v(H)-1}}\left(\displaystyle\prod_{vw\in E(H_0)}U(x_v,x_w)\right)\left(\displaystyle\int_0^1 U(x_{v_0},x_{w_0})dx_{v_0}\right)\displaystyle\prod_{v\in V(H_0)}dx_v \\ & \leq\displaystyle\int_{[0,1]^{v(H)-1}}\left(\displaystyle\prod_{vw\in E(H_0)}|U(x_v,x_w)|\right)\left|\displaystyle\int_0^1 U(x_{v_0},x_{w_0})dx_{v_0}\right|\displaystyle\prod_{v\in V(H_0)}dx_v \\ & =o\left(p\cdot \displaystyle\int_{[0,1]^{v(H)-1}}\displaystyle\prod_{vw\in E(H_0)}|U(x_v,x_w)|\displaystyle\prod_{v\in V(H_0)}dx_v\right) \\ & =o(p\cdot \Hom(H_0,|U|)) \\ & =o\left(p^{e(H_0)+1}\right) \\ & =o\left(p^{e(H)}\right).
\end{align*}
\end{proof}
\begin{lemma}\label{contributingsubgraphslemma}
Under Setup \ref{constantsetup}, if $H\subseteq K$ with $e(H)-v(H)<\gamma c(H)$, then
\[\Hom(H,U)=o\left(p^{e(H)}\right).\]
\end{lemma}
\begin{proof}
By (1) and (3) of Lemma \ref{constantconsequenceslemma}, the conditions of Corollary \ref{simpleholdercor} are satisfied with $\epsilon=p^{o(1)}$. By (5) of Lemma \ref{constantconsequenceslemma}, $I_p(W)\leq p^{2+\gamma-o(1)}$. Thus by Corollary \ref{simpleholdercor}
\begin{align*}
\Hom(H,U) & =O\left(p^{v(H)-2c(H)-o(1)}I_p(W)^{c(H)}\right) \\ & =O\left(p^{v(H)-2c(H)+(2+\gamma)c(H)-o(1)}\right) \\ & =O\left(p^{v(H)+\gamma c(H)-o(1)}\right) \\ & =o\left(p^{e(H)}\right),
\end{align*}
as $v(H)+\gamma c(H)>e(H)$, finishing the proof.
\end{proof}
The previous two lemmas immediately imply the following corollary.
\begin{cor}\label{notcontributingcor}
Under Setup \ref{constantsetup}, if $H\subseteq K$ is not contributing, then
\[\Hom(H,U)=o(p^{e(H)}).\]
\end{cor}
Combining Corollary \ref{notcontributingcor} with Lemma \ref{subgraphbreakdownresult}, we obtain a breakdown into contributing graphs.
\begin{cor}\label{contributingbreakdowncor}
Under Setup \ref{constantsetup},
\[\displaystyle\sum_{\substack{H\subseteq K \\ H\text{ contributing}}}p^{-e(H)}\Hom(H,U)\geq 1+\delta-o(1).\]
\end{cor}
As in the example in the previous section, it is important to know which vertices are being sent into the `hub' of our graphon. We will define the hub slightly differently than in last section's definition, but in a functionally similar way.
\begin{definition}\label{hubdef}
Under Setup \ref{constantsetup}, we have the following further definitions. For any $b\geq 0$, let
\[B_b=B_b(W):=\{x\in [0,1]:\displaystyle\int_0^1 |U(x,y)|dy\geq b\}.\]
For $A\subseteq V(H)$, define
\[\Omega_b^A=\Omega_b^A(W):=\left\{(x_v)_{v\in V(H)}\in [0,1]^{v(H)}:x_v\in\begin{cases}B_b & v\in A \\ \overline{B_b} & v\notin A\end{cases}\right\}.\]
Let
\[\Hom_b^A(H,U):=\displaystyle\int_{\Omega_b^A}\displaystyle\prod_{vv'\in E(H)}U(x_v,x_{v'})dx\]
\end{definition}
\begin{remark}
If we think of $B_b$ as the hub of the graphon $U$, then $\Hom_b^A(H,U)$ counts those homomorphisms from $H$ to $U$ in which exactly the set $A\subseteq V(H)$ is sent into the hub $B_b$.
\end{remark}
Notice that since the $\Omega_b^A$ partition $[0,1]^{v(H)}$ as $A$ ranges over all subsets of $V(H)$, $\Hom(H,U)=\displaystyle\sum_{A\subseteq V(H)}\Hom_b^A(H,U)$ for any $b$. The following lemma shows that many of the terms in this sum are negligibly small.

Recall from Definition \ref{validdef} that a set $A\subseteq V(H)$ is \emph{valid} if there is some $\left\{0,\frac{1}{2},1\right\}$-valued minimum fractional vertex cover $(c_v)_{v\in V(H)}$ such that $c_v=1$ if and only if $v\in A$. The following proposition essentially says that the sets of vertices that we may send in to the hub are the valid subsets of $V(H)$.
\begin{prop}\label{multiplethresholdsprop}
Assume Setup \ref{constantsetup}. There exists $b=b(p)\geq p^{1+\frac{\gamma}{3e(K)}}$ such that for all contributing subgraphs $H\subseteq K$ and all invalid $A\subseteq V(H)$,
\[\Hom_b^A(H,U)=o\left(p^{e(H)}\right)\]
\end{prop}
The proof of Proposition \ref{multiplethresholdsprop} is quite technical and will be the focus of the next section. For now, we obtain a refinement of Lemma \ref{subgraphbreakdownresult}.
\begin{cor}\label{validbreakdowncor}
Assuming Setup \ref{constantsetup}, there exists $b=b(p)\geq p^{1+\frac{\gamma}{3e(K)}}$ such that
\[\displaystyle\sum_{\substack{H\subseteq K \\ H\text{ contributing}}}\displaystyle\sum_{\text{valid }A\subseteq V(H)}p^{-e(H)}\Hom_b^A(H,U)\geq 1+\delta-o(1).\]
\end{cor}
\begin{proof}

Since for any $b$ the $\Omega_b^A$ partition $[0,1]^{v(H)}$ as $A$ ranges over all subsets of $V(H)$,
\[\displaystyle\sum_{A\subseteq V(H)}\Hom_b^A(H,U)=\Hom(H,U)\]
for any $b$. Substituting into Corollary \ref{contributingbreakdowncor},
\[\displaystyle\sum_{\substack{H\subseteq K \\ H\text{ contributing}}}\displaystyle\sum_{A\subseteq V(H)}p^{-e(H)}\Hom_b^A(H,U)\geq 1+\delta-o(1)\]
for any $b$. Taking $b$ to be the value guaranteed by Proposition \ref{multiplethresholdsprop}, for all contributing $H\subseteq K$ and all invalid $A\subseteq V(H)$, $\Hom_b^A(H,U)=o(p^{e(H)})$. Thus all terms where $A$ is invalid are negligible and in fact
\[\displaystyle\sum_{\substack{H\subseteq K \\ H\text{ contributing}}}\displaystyle\sum_{\text{valid }A\subseteq V(H)}p^{-e(H)}\Hom_b^A(H,U)\geq 1+\delta-o(1),\]
as desired.
\end{proof}
We now bound the remaining terms of the sum in Corollary \ref{validbreakdowncor}.
\begin{definition}\label{zwdef}
Assume Setup \ref{constantsetup}.

For any $b$, define $z=z(b,p)$, $w=w(b,p)$, such that
\[zp^{2+\gamma}\log\frac{1}{p}=\displaystyle\int_{B_b\times\overline{B_b}}I_p(W(x,y))dx dy\]
and
\[wp^{2+\gamma}\log\frac{1}{p}=\displaystyle\int_{\overline{B_b}\times \overline{B_b}}I_p(W(x,y))dx dy,\]
where $B_b$ is defined as in Definition \ref{hubdef}.
\end{definition}
\begin{prop}\label{monomialboundprop}
Assume Setup \ref{constantsetup}. Take $H\subseteq K$ contributing and let $A\subseteq V(H)$ be valid. Then for any $b\geq 0$,
\[\Hom_b^A(H,U)\leq (1+o(1))p^{e(H)}z^{|A|}w^{c(H)-|A|},\]
where $\Hom_b^A$ is defined as in Definition \ref{hubdef} and $z=z(b,p)$ and $w=w(b,p)$ are defined as in Definition \ref{zwdef}.
\end{prop}
The proof of Proposition \ref{monomialboundprop} is where we will use the condition that contributing subgraphs of $K$ do not have bad edges. We will apply Theorem \ref{graphholderthm}, with weightings guaranteed by the following lemma (a refinement of Lemma \ref{edgeweightslemma}).

\begin{lemma}\label{edgeweightslemma2}
If $H$ is a graph with $\delta(H)\geq 2$ and no bad edges, it has is a maximum fractional matching $(w_e)_{e\in E(H)}$ and a minimum fractional edge cover $(w'_e)_{e\in E(H)}$ such that $w_e\leq w'_e<1$ for all $e\in E(H)$.
\end{lemma}
\begin{proof}[Proof of Lemma \ref{edgeweightslemma2}]
Since $H$ has no bad edges, for each $e_0\in E(H)$ there is some maximum fractional matching $(w_e)$ with $w_{e_0}<1$. Taking such a matching for each $e_0\in E(H)$ and averaging, we obtain a maximum fractional matching $w_e$ with $w_e<1$ for all $e\in E(H)$.

Now, by Lemma \ref{edgeweightslemma}, there is some minimum fractional edge cover $(w'_e)_{e\in E(H)}$ such that $w'_e\geq w_e$ for all $e\in E(H)$. Notice that in the proof of Lemma \ref{edgeweightslemma}, we constructed this fractional edge cover by taking vertices $v$ such that $\displaystyle\sum_{e\ni v}w_e<1$ and increasing the weights of all adjacent edges by $\frac{1}{\deg(v)}\left(1-\displaystyle\sum_{e\ni v}w_e\right)$. Since $\deg(v)\geq 2$ for all $v\in E(H)$ (since $H$ is contributing), notice that this preserves edge weights being less than $1$. So we may actually guarantee that $w'_e<1$ for all $e\in E(H)$, proving the lemma.
\end{proof}

\begin{proof}[Proof of Proposition \ref{monomialboundprop}]
We are given $H\subseteq K$ contributing and $A\subseteq V(H)$ valid. Since $H$ is contributing, $\delta(H)\geq 2$, and since contributing subgraphs of $K$ do not have bad edges by assumption, $H$ has no bad edges. Thus by Lemma \ref{edgeweightslemma2}, $H$ has a maximum fractional matching $(w_e)_{e\in E(H)}$ and a minimum fractional edge cover $(w'_e)_{e\in E(H)}$ such that $w_e\leq w'_e<1$ for all $e\in E(H)$.

We apply Theorem \ref{graphholderthm} with our graph $H$. We will set the remaining parameters as follows. For all $e\in E(H)$, set $f_e=U$. We set $B_v=B_b$ when $b\in A$ and $B_v=[0,1]\backslash B_b$ when $v\in V(H)\backslash A$. Finally, we set $(w_e)$ and $(w'_e)$ to be the maximum fractional matching and minimum fractional edge cover discussed in the last paragraph, so that $w_e\leq w'_e<1$ for all $e\in E(H)$.

Under these assumptions, the left side of (\ref{graphholder}) becomes $\Hom_b^A(H,U)$. We just must bound the terms on the right hand side.

For any $vv'=e\in E(H)$, $v\in S'$, since $f_e=U$, the $\| f_e\|_{v,\frac{1}{w'_e}}^{\frac{w'_e-w_e}{w'_e}}$ term becomes
\[\esssup_{x_v\in B_v}\left(\displaystyle\int_{B_{v'}}|U(x_v,x_{v'})|^{\frac{1}{w'_e}}dx_{v'}\right)^{w'_e-w_e},\]
which is simply upper bounded by
\[\left(\sup_{x_v\in [0,1]}\displaystyle\int_0^1 |U(x_v,x_{v'})|^{\frac{1}{w'_e}}dx_{v'}\right)^{w'_e-w_e}.\]
Now, since $|x|^{\frac{1}{w'_e}}$ is a convex function, if we fix $\displaystyle\int_0^1 U(x_v,x_{v'}) dx_{v'}$, then $\displaystyle\int_0^1 |U(x_v,x_{v'})|^{\frac{1}{w'_e}}dx_{v'}$ is maximized when the values of $U$ are smoothed outward as far as possible. Since $U$ only takes values in $[-p,1-p]$, and $\displaystyle\int_0^1 U(x_v,x_{v'}) dx_{v'}=o(p)$ by (1) of Lemma \ref{constantconsequenceslemma}, this means that our integral is maximized when $U(x_v,x_{v'})=1-p$ when $x_{v'}$ is on a set of measure $p+o(p)$ and $-p$ on a set of measure $1+o(1)$. Therefore,
\[\displaystyle\int_0^1 |U(x_v,x_{v'})|^{\frac{1}{w'_e}}dx_{v'}\leq (1-p)^{\frac{1}{w'_e}}(p+o(p))+p^{\frac{1}{w'_e}}(1+o(1))=(1+o(1))\left(p+p^{\frac{1}{w'_e}}\right).\]
Since $w'_e<1$, we simply obtain the upper bound $p+o(p)$. Thus our $\| f_e\|_{v,\frac{1}{w'_e}}^{\frac{w'_e-w_e}{w'_e}}$ term is upper bounded by
\[(1+o(1))p^{w'_e-w_e}.\]
(We implicitly assumed that $w'_e\neq 0$, but when $w'_e=w_e=0$ this bound is also trivially true.)

So the first product on the right hand side of (\ref{graphholder}) becomes
\[(1+o(1))p^{\displaystyle\sum_{v\in S'}\displaystyle\sum_{e\ni v}(w'_e-w_e)}.\]
Just as in Section \ref{simpleholderproofsection}, the sum in the exponent counts each edge where $w'_e>w_e$ exactly once, so this bound is equal to
\[(1+o(1))p^{\displaystyle\sum_{e\in E(H)}(w'_e-w_e)}=(1+o(1))p^{v(H)-2c(H)},\]
by (3) and (4) of Lemma \ref{edgeweightslemma}.

Now we must bound the $\| f_e\|_{\frac{1}{w'_e}}^{\frac{w_e}{w'_e}}$ terms. Set $e=vv'\in E(H)$. Again, since $f_e=U$, and noting again that these norms are taken over $B_v\times B_{v'}$, this term is equal to
\[\left(\displaystyle\int_{\substack{x_v\in B_v \\ x_{v'}\in B_{v'}}}|U(x_v,x_{v'})|^{\frac{1}{w'_e}}dx_vdx_{v'}\right)^{w_e}.\]
By (3) of Lemma \ref{constantconsequenceslemma}, for all $x,y\in [0,1]$, either $U(x,y)=0$ or $|U(x,y)|\geq p^{1+o(1)}$. Thus by Lemma \ref{momentlemma}, since $w'_e<1$ for all $e\in E(H)$,
\[|U(x,y)|^{\frac{1}{w'_e}}\leq (1+o(1))\left(\log\frac{1}{p}\right)^{-1}I_p(W(x,y)),\]
for all $x,y\in [0,1]$. Thus
\[\left(\displaystyle\int_{\substack{x_v\in B_v \\ x_{v'}\in B_{v'}}}|U(x_v,x_{v'})|^{\frac{1}{w'_e}}dx_vdx_{v'}\right)^{w_e}\leq (1+o(1))\left(\displaystyle\int_{\substack{x_v\in B_v \\ x_{v'}\in B_{v'}}}\frac{I_p(W(x_v,x_{v'}))}{\log\frac{1}{p}}dx_vdx_{v'}\right)^{w_e}.\]
When exactly one of $v,v'\in A$, the right hand integral is over $B_b\times\overline{B_b}$, and thus equals $zp^{2+\gamma}$. When both $v,v'$ are not in $A$, the right hand integral equals $wp^{2+\gamma}$. If $v,v'\in A$, we may simply bound the integral above by $1$, since $I_p$ is upper bounded by $\log\frac{1}{p}$.

Thus letting $E_1(H)$ be the set of edges in $H$ with exactly one vertex in $A$ and $E_0(H)$ be the set of edges in $H$ with no vertices in $A$, the $\displaystyle\prod_{e\in E(H)}\| f_e\|_{\frac{1}{w'_e}}^{\frac{w_e}{w'_e}}$ term in the right side of (\ref{graphholder}) is bounded above by
\[(1+o(1))(zp^{2+\gamma})^{\displaystyle\sum_{e\in E_1(H)}w_e}(wp^{2+\gamma})^{\displaystyle\sum_{e\in E_0(H)}w_e}.\]
Putting our bounds together,
\begin{align*}
\Hom_b^A(H,U) & \leq (1+o(1))p^{v(H)-2c(H)}(zp^{2+\gamma})^{\displaystyle\sum_{e\in E_1(H)}w_e}(wp^{2+\gamma})^{\displaystyle\sum_{e\in E_0(H)}w_e} \\ & =(1+o(1))z^{\displaystyle\sum_{e\in E_1(H)}w_e}w^{\displaystyle\sum_{e\in E_0(H)}w_e}p^{(2+\gamma)\left(\displaystyle\sum_{e\in E_0(H)\cup E_1(H)}w_e\right)+v(H)-2c(H)}.
\end{align*}
If we can show that $\displaystyle\sum_{e\in E_1(H)}w_e=|A|$ and $\displaystyle\sum_{e\in E_0(H)}w_e=c(H)-|A|$, then this bound will become
\[(1+o(1))z^{|A|}w^{c(H)-|A|}p^{(2+\gamma)c(H)+v(H)-2c(H)}=(1+o(1))z^{|A|}w^{c(H)-|A|}p^{e(H)},\]
as $H$ is contributing so $\gamma c(H)+v(H)=p(H)$. Thus we have reduced Proposition \ref{monomialboundprop} to the following lemma.

\begin{lemma}\label{aweightslemma}
Let $H$ be any graph and let $A\subseteq V(H)$ be a valid subset. Let $E_1(H)$ be the set of edges $e\in E(H)$ with exactly one vertex in $A$ and let $E_0(H)$ be the set of edges $e\in E(H)$ with no vertices in $A$. Then for any maximum fractional matching $(w_e)_{e\in E(H)}$,
\[\displaystyle\sum_{e\in E_1(H)}w_e=|A|\]
and
\[\displaystyle\sum_{e\in E_0(H)}w_e=c(H)-|A|\]
\end{lemma}
\begin{proof}[Proof of Lemma \ref{aweightslemma}]
Since maximum fractional matching and minimum fractional vertex cover are dual systems, by complementary slackness we know that if $e=uu'$ with $w_e>0$, then for any minimum fractional vertex cover $(c_v)_{v\in V(H)}$, $c_u+c_{u'}=1$. If $e=uu'$ with $u,u'\in A$, we know that since $A$ is valid there is some minimum fractional vertex cover $(c_v)$ with $c_u=c_{u'}=1$, so $c_u+c_{u'}=2$, and this implies that $w_e=0$. Therefore,
\[\displaystyle\sum_{e\in E_1(H)}w_e+\displaystyle\sum_{e\in E_0(H)}w_e=\displaystyle\sum_{e\in E(H)}w_e=c(H),\]
so it suffices to show just one of the two identities in the lemma.

We may write
\[\displaystyle\sum_{e\in E_1(H)}w_e=\displaystyle\sum_{v\in A}\displaystyle\sum_{e\ni v}w_e,\]
because for any $e\ni v$ with $e\not\in E_1(H)$, $e$ has two vertices in $A$ and thus $w_e=0$ by the argument earlier. Thus we just must show that
\[\displaystyle\sum_{e\ni v}w_e=1\]
for all $v\in A$. But this again follows by complementary slackness, as if equality does not hold in the inequality $\displaystyle\sum_{e\ni v}w_e\leq 1$, then for all minimum fractional vertex covers $(c_w)_{w\in V(H)}$ we must have $c_v=0$. But $c_v=1$ in at least one minimum fractional vertex cover for all $v\in A$ by the validity of $A$, a contradiction. This completes the proof of Lemma \ref{aweightslemma} and thus the proof of Proposition \ref{monomialboundprop}.
\end{proof}
\end{proof}
To bound the entropy of $W$ we will need the following lemma.
\begin{lemma}\label{2z+wlemma}
Assume Setup \ref{constantsetup}. Fix $b\geq 0$ and define $z=z(b,p)$, $w=w(b,p)$ as in Definition \ref{zwdef}.  Then
\[I_p(W)\geq (2z+w)p^{2+\gamma}\log\frac{1}{p}.\]
\end{lemma}
\begin{proof}
Breaking apart the integral, we have
\begin{align*}
I_p(W) & =\displaystyle\int_{[0,1]^2}I_p(W(x,y)) dx dy \\ & =\displaystyle\int_{B_b\times B_b}I_p(W(x,y)) dx dy+\displaystyle\int_{B_b\times\overline{B_b}}I_p(W(x,y)) dx dy+\displaystyle\int_{\overline{B_b}\times B_b}I_p(W(x,y)) dx dy+\displaystyle\int_{\overline{B_b}\times\overline{B_b}}I_p(W(x,y)) dx dy \\ & \geq\displaystyle\int_{B_b\times\overline{B_b}}I_p(W(x,y)) dx dy+\displaystyle\int_{\overline{B_b}\times B_b}I_p(W(x,y)) dx dy+\displaystyle\int_{\overline{B_b}\times\overline{B_b}}I_p(W(x,y)) dx dy \\ & =(2z+w)p^{2+\gamma}\log\frac{1}{p},
\end{align*}
by the symmetry of $W$.
\end{proof}
\begin{proof}[Proof of Lower Bound of Theorem \ref{varcorrectconstant}]
By the discussion preceding Lemma \ref{constantconsequenceslemma}, it suffices to show that under Setup \ref{constantsetup},
\[I_p(W)\geq (2-o(1))\rho(K,\delta)p^{2+\gamma}\log\frac{1}{p}.\]
Take $b$ to be the value guaranteed by Corollary \ref{validbreakdowncor}, so that
\[\displaystyle\sum_{\substack{H\subseteq K \\ H\text{ contributing}}}\displaystyle\sum_{\text{valid }A\subseteq V(H)}p^{-e(H)}\Hom_b^A(H,U)\geq 1+\delta-o(1).\]
Define $z=z(b,p)$ and $w=w(b,p)$ as in Definition \ref{zwdef}. By Proposition \ref{monomialboundprop}, $\Hom_b^A(H,U)\leq (1+o(1))p^{e(H)}z^{|A|}w^{c(H)-|A|}$ for all valid $A\subseteq V(H)$, so
\[\displaystyle\sum_{\substack{H\subseteq K \\ H\text{ contributing}}}\displaystyle\sum_{\text{valid }A\subseteq V(H)}z^{|A|}w^{c(H)-|A|}\geq 1+\delta-o(1).\]
The left side of the expression above is simply the definition of $P_K(z,w)$, so
\[P_K(z,w)\geq 1+\delta-o(1).\]
By Lemma \ref{2z+wlemma}, we know that $I_p(W)\geq (2z+w)p^{2+\gamma}\log\frac{1}{p}$, so
\begin{align*}
I_p(W) & \geq\left(\displaystyle\min_{\substack{z,w\geq 0 \\ P_K(z,w)\geq 1+\delta-o(1)}}(2z+w)\right)p^{2+\gamma}\log\frac{1}{p} \\ & =2\rho(K,\delta-o(1))p^{2+\gamma}\log\frac{1}{p}.
\end{align*}
By the same argument as in Section \ref{constructionssection}, $\rho(K,\delta-o(1))=(1-o(1))\rho(K,\delta)$, so
\[I_p(W)\geq (2-o(1))\rho(K,\delta)p^{2+\gamma}\log\frac{1}{p},\]
completing the proof of the lower bound and thus the proof of Theorem \ref{varcorrectconstant}.
\end{proof}
\section{Proof of Proposition \ref{multiplethresholdsprop}}\label{constantsection2}
The only remaining loose end in our proof of Theorem \ref{varcorrectconstant} (except for Proposition \ref{thesegraphonsworkprop}) is the proof of Proposition \ref{multiplethresholdsprop}, which we will prove here. Our argument is inspired by the adaptive-thresholding argument in Lemma 5.2 of \cite{BGLZ}. The idea is to have many slightly-separated possible thresholds $b_1,b_2,\ldots,b_m$, and prove that at least one has the desired properties.

To this end, we use the following technical lemma.
\begin{lemma}\label{multiplethresholdslemma}
Assume Setup \ref{constantsetup}, and take $H\subseteq K$ contributing and $A\subseteq V(H)$ invalid. Take $\epsilon=\epsilon(p)\geq\left(\log\frac{1}{p}\right)^{-O(1)}$ and $m=m(p)\in\mathbb{Z}^+$. Take any $b_1,\ldots,b_m$ (functions of $p$) with $p^{1+\frac{\gamma}{3v(H)}}\leq b_1<\cdots <b_m\leq p$ and $\frac{b_i}{p}\leq\left(\frac{b_{i+1}}{p}\right)^{v(K)}\left(\log\frac{1}{p}\right)^{-\omega(1)}$ for all $i\in [m-1]$. Then defining $\Hom_{b_i}^A(H,U)$ as in Definition \ref{hubdef},
\[\left|\Hom_{b_i}^A(H,U)\right|\leq (1+o(1))\epsilon p^{e(H)}\]
for all but at most $\frac{(1+\delta)2^{e(K)}}{\epsilon}+1$ values of $i\in [m]$.
\end{lemma}
We first prove Proposition \ref{multiplethresholdsprop} given this lemma, and then prove the lemma itself.
\begin{proof}[Deduction of Proposition \ref{multiplethresholdsprop} from Lemma \ref{multiplethresholdslemma}]
Assume Setup \ref{constantsetup}. Notice that it is possible to take $m=\log\log\log\frac{1}{p}$ and $b_1,\cdots,b_m$ satisfying $p^{1+\frac{\gamma}{3v(K)}}\leq b_1<\cdots<b_m\leq p$ and $\frac{b_i}{p}\leq\left(\frac{b_{i+1}}{p}\right)^{v(K)}\left(\log\frac{1}{p}\right)^{-\omega(1)}$ for all $i\in [m-1]$. For example, we can take
\[b_{m-i}=p\cdot\exp\left(-(v(K)+1)^i\sqrt{\log\frac{1}{p}}\right),\]
in which case $-\log\frac{b_i}{p}+v(K)\log\frac{b_{i+1}}{p}=(v(K)+1)^{m-i-1}\sqrt{\log\frac{1}{p}}\gg\log\log\frac{1}{p}$.
This is valid as long as $b_1\geq p^{1+\frac{\gamma}{3v(K)}}$, which will be true since $m\ll\log\log\frac{1}{p}$ and $\gamma>0$.

Now, take $\epsilon=\left(\log\log\log\frac{1}{p}\right)^{-\frac{1}{2}}$. By Lemma \ref{multiplethresholdslemma}, for any contributing $H\subseteq K$ and any invalid $A\subseteq V(H)$, we have $|\Hom_{b_i}^A(H,U)|\leq (1+o(1))\epsilon p^{e(H)}$ for all but at most $O\left(\frac{1}{\epsilon}\right)=O\left(\sqrt{\log\log\log\frac{1}{p}}\right)$ values of $b_i$. Since we have $\log\log\log\frac{1}{p}$ total values of $b_i$ and each $(H,A)$ pair (of which there are $O(1)$) only invalidates $O\left(\sqrt{\log\log\log\frac{1}{p}}\right)$ of them, we may find some $b_i\geq p^{1+\frac{\gamma}{3v(K)}}$ such that
\[|\Hom_{b_i}^A(H,U)|\leq (1+o(1))\epsilon p^{e(H)}\]
for all contributing $H$ and invalid $A\subseteq V(H)$. Since $\epsilon=o(1)$, Proposition \ref{multiplethresholdsprop} follows.
\end{proof}
The remainder of this section will be devoted to the proof of Lemma \ref{multiplethresholdslemma}. Suppose the conditions of Lemma \ref{multiplethresholdslemma} hold for some $H\subseteq K$ contributing, $A\subseteq V(H)$ invalid, and $b_i$, $1\leq i\leq m$. Now, for all $i$, $2\leq i\leq m$, we have the decomposition
\begin{align*}
|\Hom_{b_i}^A(H,U)| & =\left|\displaystyle\int_{\Omega_{b_i}^A}\displaystyle\prod_{vv'\in E(H)}U(x_v,x_{v'})dx\right| \\ & \leq \displaystyle\int_{\Omega_{b_i}^A}\displaystyle\prod_{vv'\in E(H)}|U(x_v,x_{v'})|dx \\ & \leq \displaystyle\int_{\Omega_{b_i}^A\backslash\Omega_{b_{i-1}}^A}\displaystyle\prod_{vv'\in E(H)}|U(x_v,x_{v'})|dx+\displaystyle\int_{\Omega_{b_i}^A\cap\Omega_{b_{i-1}}^A}\displaystyle\prod_{vv'\in E(H)}|U(x_v,x_{v'})|dx,
\end{align*}
recalling from Definition \ref{hubdef} the definition of $\Omega_b^A$. Thus (ignoring $b_1$ altogether) Lemma \ref{multiplethresholdslemma} will follow from the following two statements.
\begin{lemma}\label{disjointlemma}
Under the conditions of Lemma \ref{multiplethresholdslemma}, for all but at most $\frac{(1+\delta)2^{e(K)}}{\epsilon}$ values of $i$, $2\leq i\leq m$,
\[\displaystyle\int_{\Omega_{b_i}^A\backslash\Omega_{b_{i-1}}^A}\displaystyle\prod_{vv'\in E(H)}|U(x_v,x_{v'})|dx\leq\epsilon p^{e(H)}.\]
\end{lemma}
\begin{lemma}\label{intersectionlemma}
Under the conditions of Lemma \ref{multiplethresholdslemma}, for all $i$, $2\leq i\leq m$,
\[\displaystyle\int_{\Omega_{b_i}^A\cap\Omega_{b_{i-1}}^A}\displaystyle\prod_{vv'\in E(H)}|U(x_v,x_{v'})|dx=o(\epsilon p^{e(H)}).\]
\end{lemma}
\begin{proof}[Proof of Lemma \ref{disjointlemma}]
Noting that the $\Omega_{b_i}^A\backslash\Omega_{b_{i-1}}^A$, $2\leq i\leq m$ are disjoint, we have that
\begin{align*}
\displaystyle\sum_{i=2}^m\displaystyle\int_{\Omega_{b_i}^A\backslash\Omega_{b_{i-1}}^A}\displaystyle\prod_{vv'\in E(H)}|U(x_v,x_{v'})|dx & \leq\displaystyle\int_{[0,1]^{v(H)}}\displaystyle\prod_{vv'\in E(H)}|U(x_v,x_{v'})|dx \\ & =\Hom(H,|U|) \\ & \leq 2^{e(K)}(1+\delta)p^{e(H)},
\end{align*}
where the last step is by (4) of Lemma \ref{constantconsequenceslemma}. Thus at most $\frac{(1+\delta)2^{e(K)}}{\epsilon}$ terms of the sum are at least $\epsilon p^{e(H)}$, proving the lemma.
\end{proof}
\begin{proof}[Proof of Lemma \ref{intersectionlemma}]
Suppose the conditions of Lemma \ref{multiplethresholdslemma} hold, with some $H\subseteq K$ contributing, $A\subseteq V(H)$ invalid. Let $C(A)$ be the set of vertices in $V(H)\backslash A$ that are only adjacent to vertices in $A$, and let $H'$ be the restriction of $H$ to $V(H)\backslash (A\cup C(A))$. Dropping the edges between $A$ and $V(H')$ (and letting the integral over the vertices in $C(A)$ range from $0$ to $1$) only increases the value of the desired integral, and thus we have the bound
\begin{align}
\label{ACAbreakdown}\displaystyle\int_{\Omega_{b_i}^A\cap\Omega_{b_{i-1}}^A}\displaystyle\prod_{vv'\in E(H)}|U(x_v,x_{v'})|dx & \leq\left(\displaystyle\int_{B_{b_i}^{|A|}}\displaystyle\prod_{v\in C(A)}\left(\displaystyle\int_0^1 \displaystyle\prod_{v'\in N_H(v)}|U(x_v,x_{v'})|dx_v\right)\displaystyle\prod_{v\in A}dx_v\right) \\ & \nonumber\cdot\left(\displaystyle\int_{\overline{B_{b_{i-1}}}^{|V(H')|}}\displaystyle\prod_{vv'\in E(H')}|U(x_v,x_{v'})|\displaystyle\prod_{v\in V(H')}dx_v\right).
\end{align}
\setcounter{claimcounter}{0}
The following two claims will help bound the first factor of this product.
\begin{claim}\label{plemma}
For any $v\in C(A)$ and fixed $(x_{v'})_{v'\in N(v)}\in [0,1]^{\deg(v)}$,
\[\displaystyle\int_0^1 \displaystyle\prod_{v'\in N_H(v)}|U(x_v,x_{v'})|dx_v\leq p+o(p).\]
\end{claim}
\begin{proof}[Proof of Claim \ref{plemma}]
Since $H$ is contributing, $\deg_H v\geq 2$. Let $w,w'$ be two distinct neighbors of $v$. Then
\begin{align*}
\displaystyle\int_0^1 \displaystyle\prod_{v'\in N_H(v)}|U(x_v,x_{v'})|dx_v & \leq\displaystyle\int_0^1 |U(x_v,x_{w})U(x_v,x_{w'})|dx_v \\ & \leq\left(\displaystyle\int_0^1 U(x_v,x_w)^2 dx_v\right)^{\frac{1}{2}}\left(\displaystyle\int_0^1 U(x_v,x_{w'})^2 dx_v\right)^{\frac{1}{2}}.
\end{align*}
But since $W$ is $1$-bounded,
\begin{align*}
\displaystyle\int_0^1 U(x_v,x_w)^2 dx_v & =\displaystyle\int_0^1 (W(x_v,x_w)-p)^2 dx_v \\ & =\displaystyle\int_0^1 (W(x_v,x_w)^2-2pW(x_v,x_w)+p^2) dx_v \\ & \leq\displaystyle\int_0^1 (1-2p)W(x_v,x_w) dx_v+p^2 \\ & =p+o(p),
\end{align*}
by (1) of Lemma \ref{constantconsequenceslemma}. The same holds for $\displaystyle\int_0^1 U(x_v,x_{w'})^2 dx_v$, so
\[\left(\displaystyle\int_0^1 U(x_v,x_w)^2 dx_v\right)^{\frac{1}{2}}\left(\displaystyle\int_0^1 U(x_v,x_{w'})^2 dx_v\right)^{\frac{1}{2}}\leq p+o(p),\]
proving Claim \ref{plemma}.
\end{proof}
\begin{claim}\label{mbblemma}
\[m(B_{b_i})=O\left(\frac{p^{2+\gamma}\left(\log\frac{1}{p}\right)^2}{b_i}\right)\]
\end{claim}
\begin{proof}[Proof of Claim \ref{mbblemma}]
By the definition of $B_{b_i}$, for all $x\in B_{b_i}$, $\displaystyle\int_0^1 |U(x,y)|dy\geq b_i$. Therefore,
\begin{align*}
m(B_i)b_i & \leq\displaystyle\int_{B_b\times [0,1]}|U(x,y)|dx dy \\ & \leq\displaystyle\int_{[0,1]^2}|U(x,y)|dx dy \\ & =\|U\|_1 \\ & =O\left(p^{2+\gamma}\left(\log\frac{1}{p}\right)^2\right)
\end{align*}
by (6) of Lemma \ref{constantconsequenceslemma}, proving Claim \ref{mbblemma}.
\end{proof}
Using Claims \ref{plemma} and \ref{mbblemma},
\begin{align}
\label{ACAfirstbound}\left(\displaystyle\int_{B_{b_i}^{|A|}}\displaystyle\prod_{v\in C(A)}\left(\displaystyle\int_0^1 \displaystyle\prod_{v'\in N(v)}|U(x_v,x_{v'})|dx_v\right)\displaystyle\prod_{v\in A}dx_v\right) & \leq (1+o(1))\left(\displaystyle\int_{B_{b_i}^{|A|}}p^{|C(A)|}\displaystyle\prod_{v\in A}dx_v\right) \\ & \nonumber =(1+o(1))p^{|C(A)|}m(B_{b_i})^{|A|} \\ & \nonumber =O\left(\left(\log\frac{1}{p}\right)^{2|A|}b_i^{-|A|}p^{(2+\gamma)|A|+|C(A)|}\right)
\end{align}
We now bound the second term in the product (\ref{ACAbreakdown}), with the following claim.
\begin{claim}\label{ACAsecondbound}
\[\displaystyle\int_{\overline{B_{b_{i-1}}}^{|V(H')|}}\displaystyle\prod_{v,v'\in V(H')}|U(x_v,x_{v'})|\displaystyle\prod_{v\in V(H')}dx_v=O\left(b_{i-1}^{v(H')-2c(H')}\left(p^{2+\gamma}\left(\log\frac{1}{p}\right)^2\right)^{c(H')}\right).\]
\end{claim}
\begin{proof}[Proof of Claim \ref{ACAsecondbound}]
We apply Theorem \ref{graphholderthm}, with the following parameters. Our graph will be $H'$, and for each $e\in E(H')$ we will let $f_e=|U|$. For all $v\in V(H')$, let $B_v=\overline{B_{b_{i-1}}}$. Let $(w_e)_{e\in E(H')}$ be any maximum fractional matching and let $(w'_e)_{e\in E(H')}$ be a minimum fractional edge cover with $w_e\leq w'_e$ for all $e\in E(H)$, the existence of which is guaranteed by Lemma \ref{edgeweightslemma}.

Under these conditions, the left side of (\ref{graphholder}) becomes exactly the left side of Claim \ref{ACAsecondbound}. We now bound the right side of (\ref{graphholder}).

We may write the $\|f_e\|_{v,\frac{1}{w'_e}}^{\frac{w'_e-w_e}{w'_e}}$ term as
\[\displaystyle\esssup_{x\in\overline{B_{b_{i-1}}}}\left(\displaystyle\int_{\overline{B_{b_{i-1}}}}|U(x,y)|^{\frac{1}{w'_e}}\right)^{w'_e-w_e}.\]
Now, for any $x\in\overline{B_{b_{i-1}}}$, by the definition of $B_{b_{i-1}}$ and using the fact that $U$ is $1$-bounded and $w'_e\leq 1$
\begin{align*}
\displaystyle\int_{\overline{B_{b_{i-1}}}}|U(x,y)|^{\frac{1}{w'_e}}dy & \leq\displaystyle\int_{\overline{B_{b_{i-1}}}}|U(x,y)|dy \\ & \leq b_{i-1}.
\end{align*}
(Note that this still holds when $w'_e=w_e=0$.) So the first product on the right side of (\ref{graphholderthm}) is upper bounded by
\[b_{i-1}^{\displaystyle\sum_{v\in S'}\displaystyle\sum_{e\ni v}(w'_e-w_e)},\]
which by the same argument as in Section \ref{simpleholderproofsection} is equal to $b_{i-1}^{v(H')-2c(H')}$.

Furthermore, the $\|f_e\|_{\frac{1}{w'_e}}^{\frac{w_e}{w'_e}}$ term equals
\[\left(\displaystyle\int_{\overline{B_{b_{i-1}}}^2}|U(x,y)|^{\frac{1}{w'_e}} dx dy\right)^{w_e}\leq\left(\displaystyle\int_{[0,1]^2}|U(x,y)| dx dy\right)^{w_e}\]
By the same argument as in the proof of Claim \ref{mbblemma}, $|U(x,y)|=O\left(I_p(W(x,y))\log\frac{1}{p}\right)$. Thus we may bound this term by
\[O\left(\left(I_p(W)\log\frac{1}{p}\right)^{w_e}\right),\]
and so the second product on the right side of (\ref{graphholder}) is upper bounded by
\[O\left(\left(I_p(W)\log\frac{1}{p}\right)^{c(H')}\right)=O\left(\left(p^{2+\gamma}\left(\log\frac{1}{p}\right)^2\right)^{c(H')}\right).\]
Multiplying our two bounds yields the desired bound of
\[O\left(b_{i-1}^{v(H')-2c(H')}\left(p^{2+\gamma}\left(\log\frac{1}{p}\right)^2\right)^{c(H')}\right).\]
\end{proof}
By (\ref{ACAbreakdown}), (\ref{ACAfirstbound}), and Claim \ref{ACAsecondbound},
\begin{equation}\label{bigasymptoticbound}
\displaystyle\int_{\Omega_{b_i}^A\cap\Omega_{b_{i-1}}^A}\displaystyle\prod_{vv'\in E(H)}|U(x_v,x_{v'})|dx=O\left(\left(\log\frac{1}{p}\right)^{2|A|+2c(H)}b_{i-1}^{v(H')-2c(H')}b_i^{-|A|}p^{(2+\gamma)(|A|+c(H'))+|C(A)|}\right).
\end{equation}
To complete the proof of Lemma \ref{intersectionlemma} and thus the proof of Lemma \ref{multiplethresholdslemma}, we need only show that the right side of (\ref{bigasymptoticbound}) is $o(\epsilon p^{e(H)})$. Thus since $\epsilon\geq\left(\log\frac{1}{p}\right)^{-O(1)}$ by the conditions of Lemma \ref{multiplethresholdslemma}, it suffices to show the following claim.
\begin{claim}\label{bigasymptoticlemma}
Under the conditions of Lemma \ref{multiplethresholdslemma}, if $C(A)$ is the set of vertices in $V(H)\backslash A$ that are only adjacent to vertices in $A$, then
\[\left(\log\frac{1}{p}\right)^{O(1)}b_{i-1}^{v(H')-2c(H')}b_i^{-|A|}p^{(2+\gamma)(|A|+c(H'))+|C(A)|-e(H)}=o(1).\]
\end{claim}
\begin{proof}[Proof of Claim \ref{bigasymptoticlemma}]
We rewrite our expression as
\[\left(\log\frac{1}{p}\right)^{O(1)}\left(\frac{b_{i-1}}{p}\right)^{v(H')-2c(H')}\left(\frac{b_i}{p}\right)^{-|A|}p^{\gamma(|A|+c(H'))+|A|+|C(A)|+v(H')-e(H)}.\]
Now, since $H$ is contributing, $e(H)=v(H)+\gamma c(H)=|A|+|C(A)|+v(H')+\gamma c(H)$, so substituting, we obtain
\[\left(\log\frac{1}{p}\right)^{O(1)}\left(\frac{b_{i-1}}{p}\right)^{v(H')-2c(H')}\left(\frac{b_i}{p}\right)^{-|A|}p^{\gamma(|A|+c(H')-c(H))}.\]
The idea is that the exponential with base $p$ should dominate unless the exponent is $0$, in which case the exponential with base $\frac{b_{i-1}}{p}$ should dominate.

Suppose for the sake of contradiction that this expression is $\Omega(1)$.
Since $b_i\geq p^{1+\frac{\gamma}{3v(H)}}$ and $b_{i-1}\leq p$, and $v(H')\geq 2c(H')$ (since setting all weights $\frac{1}{2}$ is always a fractional vertex cover)
\[\left(\log\frac{1}{p}\right)^{O(1)}\left(\frac{b_{i-1}}{p}\right)^{v(H')-2c(H')}\left(\frac{b_i}{p}\right)^{-|A|}\leq p^{-\frac{\gamma|A|}{3v(H)}-o(1)}\leq p^{-\frac{\gamma}{3}-o(1)},\]
so 
$|A|+c(H')-c(H)\leq\frac{1}{3}$. Now, $|A|+c(H')-c(H)$ is a half integer, as the vertices of the fractional vertex cover polytope all have half-integer coordinates (see for example Theorem 64.11 of \cite{Sch}). Furthermore, $|A|+c(H')-c(H)\geq 0$, as any fractional vertex cover of $H'$ can be extended to a fractional vertex cover of $H$ by giving the vertices in $A$ weight $1$ and the vertices in $C(A)$ weight $0$. The only half-integer in $\left[0,\frac{1}{3}\right]$ is $0$, so we must have
\begin{equation}\label{cHcH'}
c(H)=|A|+c(H').
\end{equation}
Subsituting, we have that by assumption
\[\left(\log\frac{1}{p}\right)^{O(1)}\left(\frac{b_{i-1}}{p}\right)^{v(H')-2c(H')}\left(\frac{b_i}{p}\right)^{-|A|}=\Omega(1).\]
Now, $v(H')-2c(H')$ is an integer, so if it is not $0$ it is at least $1$ and our expression is at most
\begin{align*}
\left(\log\frac{1}{p}\right)^{O(1)}\frac{b_{i-1}}{p}\left(\frac{b_i}{p}\right)^{-|A|} & \leq\left(\log\frac{1}{p}\right)^{O(1)}\frac{b_{i-1}}{p}\left(\frac{b_i}{p}\right)^{-v(K)}\\ & \ll 1,
\end{align*}
as by the conditions of Lemma \ref{multiplethresholdslemma} $\frac{b_i}{p}\leq\left(\frac{b_{i+1}}{p}\right)^{v(K)}\left(\log\frac{1}{p}\right)^{-\omega(1)}$ for all $i$. This is a contradiction, and thus
\begin{equation}\label{vH'cH'}
v(H')=2c(H').
\end{equation}
Combining (\ref{cHcH'}) and (\ref{vH'cH'}),
\[c(H)=|A|+\frac{v(H')}{2}.\]
But this implies that the fractional vertex cover of $H$ given by giving all vertices in $A$ weight $1$, all vertices in $v(H')$ weight $\frac{1}{2}$, and all vertices in $C(A)$ weight $0$ is minimum. (This is a fractional vertex cover as vertices in $C(A)$ are only adjacent to vertices in $A$ by definition.) Thus $A$ is valid, a contradiction, so our assumption was false and the lemma follows.
\end{proof}
This completes the proof of Lemma \ref{intersectionlemma} and thus the proof of Lemma \ref{multiplethresholdslemma}.
\end{proof}
\section{$K_{2,4}$ Plus an Edge}\label{K24section}
In this section, we prove the lower bound of Theorem \ref{K24correctconstant}. Let $K=K_0$ be the graph in Figure \ref{K24plusanedgefigure}, given by adding an edge to $K_{2,4}$.
\subsection{Eliminating Subgraphs}
We first determine which subgraphs are contributing.

\begin{lemma}\label{K0contributingsubgraphslemma}
$\gamma(K_0)=1$, and the contributing subgraphs of $K_0$ are $\emptyset$, $K_{2,4}$, and $K_0$ itself.
\end{lemma}
\begin{proof}
If we can prove that $e(H)\leq c(H)+v(H)$ for all $H\subseteq K$ with equality only when $H\in\{\emptyset,K_{2,4},K_0\}$, then we will simultaneously prove that $\gamma(K_0)=1$ and that the only contributing subgraphs are the desired ones.

If $H\subseteq K_0$ with $c(H)\leq 1$ and $H$ having at least one edge, all weight in any minimum fractional vertex cover of $H$ must be concentrated along one edge, so all vertices not on that edge have weight $0$ and thus must only be connected to a single vertex with weight $1$, $H$ is a star. But then $e(H)<v(H)$, so it impossible that $e(H)=v(H)+c(H)$. Thus if $H\neq\emptyset$ and $e(H)=v(H)+c(H)$, we must have $e(H)\geq v(H)+2$. This immediately rules out all nonempty $H$ with at most $4$ vertices, as $K_0$ contains no copies of $K_4$.

Thus we only need consider $H$ with $5$ or $6$ vertices. If $v(H)=5$, then since we need $e(H)\geq v(H)+2$, $e(H)\geq 7$. But there are only $9$ edges in total in $K_0$ and each vertex has degree $2$, so we must obtain $H$ from $K_0$ by removing a vertex of degree $2$. Thus $H$ is $K_{2,3}$ with an extra edge on the side with $3$ vertices. However, in this case, $c(H)=\frac{5}{2}$, which we can see as $H$ has a fractional perfect matching given by taking a $5$-cycle (which does appear as a subgraph of $H$) and weighting all edges in it $\frac{1}{2}$. Since a fractional perfect matching is both a fractional matching and a fractional edge cover, if $H$ is any graph with a fractional perfect matching then $c(H)\geq v(H)-c(H)$, by Lemma \ref{edgeweightslemma}, so $c(H)\geq\frac{v(H)}{2}$, and the reverse equality holds since weighting all vertices $\frac{1}{2}$ is a fractional vertex cover. Therefore, for this choice of $H$, $e(H)<v(H)+c(H)$, and we may rule this option out as well.

So besides $H=\emptyset$, we need only consider when $v(H)=6$ and $e(H)\in\{8,9\}$. When $e(H)=9$ we must have $H=K_0$, and since $K_0$ has a perfect matching (and thus a fractional perfect matching) we must have $c(K_0)=\frac{v(K_0)}{2}=3$. Thus $e(K_0)=c(K_0)+v(K_0)$.

We only now need check when $v(H)=6$ and $e(H)=8$; that is, $H$ is obtained by removing a single edge from $K_0$. If $H$ contains a perfect matching, then $c(H)=3$, so $e(H)<v(H)+c(H)$. But the intersection of all perfect matchings of $K_0$ is a single edge, and removing this edge yields $K_{2,4}$. Thus the only option in this case is $H=K_{2,4}$. Since $H$ has a matching of size $2$, $c(H)\geq 2$, and since it has a vertex cover of size $2$, $c(H)=2$. So when $H=K_{2,4}$, $e(H)=v(H)+c(H)$ as well. This proves the lemma.
\end{proof}
We will use the following setup, which is simply Setup \ref{constantsetup} applied to our graph $K_0$.
\begin{setup}\label{K24setup}
Let $K=K_0$ be the graph in Figure \ref{K24plusanedgefigure} given by adding an edge to $K_{2,4}$. Fix a constant $\delta>0$, and take $p\to 0$.

Recall from Definition \ref{nicegraphonsdef} the definition of $\Gamma_{\epsilon}(K_0,\delta)$. By Lemma \ref{replacebyplemma}, for any $p<\frac{1}{e}$ we may find $W=W(p)\in\Gamma_{1/\log(1/p)}(K_0,\delta)$ such that
\[I_p(W)\leq (1+o(1))\displaystyle\inf_{\substack{\Hom(K_0,W')\geq (1+\delta)p^9 \\ W'\text{ }p\text{-regular}}}I_p(W').\]
Let $U=W-p$.
\end{setup}
The analogue of Lemma \ref{constantconsequenceslemma} is the following.
\begin{lemma}\label{K24consequenceslemma}
Under Setup \ref{K24setup}, the following properties hold.
\begin{enumerate}
\item $\displaystyle\int_0^1 W(x_0,y) dy=(1+o(1))p$ for all $x_0\in [0,1]$
\item $\Hom(K_0,W)\geq (1+\delta-o(1))p^9$.
\item $W$ takes no values in $[p-p^{1+o(1)},p+p^{1+o(1)}]\backslash p$.
\item $\Hom(H,|U|)\leq 512(1+\delta)p^{e(H)}$ for all $H\subseteq K_0$
\item $I_p(W)\ll p^3\log\frac{1}{p}$
\item $\|U\|_1\ll p^3\left(\log\frac{1}{p}\right)^2$
\item $\|U\|_2\ll p^{\frac{3}{2}}$
\end{enumerate}
\end{lemma}
\begin{proof}
Since Setup \ref{K24setup} implies Setup \ref{constantsetup} with $K=K_0$, noting that $e(K_0)=9$ we obtain (1) through (4) from Lemma \ref{constantconsequenceslemma}.

(5) comes from the fact that $I_p(W)\leq (1+o(1))\displaystyle\inf_{\substack{\Hom(K,W')\geq (1+\delta)p^{e(K)} \\ W'\text{ }p\text{-regular}}}I_p(W')$ and the upper bound of Theorem \ref{K24correctconstant} (proven in Section \ref{constructionssection}).

We may derive (6) from (5) in the same way as in Lemma \ref{constantconsequenceslemma}.

Note that by (3) and Lemma \ref{momentlemma}, $I_p(W(x,y))\geq (1-o(1))U(x,y)^2\log\frac{1}{p}$ for all $x,y$. Integrating and applying (5) yields (7).
\end{proof}
As in Section \ref{constantsection1}, we will break down into terms of the form $p^{-e(H)}\Hom(H,U)$.
\begin{lemma}\label{K0eliminatingsubgraphslemma}
Under Setup \ref{K24setup}, 
\[p^{-8}\Hom(K_{2,4},U)+p^{-9}\Hom(K_0,U)\geq\delta-o(1)\]
\end{lemma}
\begin{proof}
Since Setup \ref{K24setup} implies Setup \ref{constantsetup} with $K=K_0$, we may apply Corollary \ref{contributingbreakdowncor} to obtain that
\[\displaystyle\sum_{\substack{H\subseteq K_0 \\ H\text{ contributing}}}p^{-e(H)}\Hom(H,U)\geq 1+\delta-o(1).\]
By Lemma \ref{K0contributingsubgraphslemma}, the only contributing subgraphs of $K_0$ are $\emptyset$, $K_{2,4}$, and $K_0$ itself. Since $p^{-e(\emptyset)}\Hom(\emptyset,U)=1$, the lemma follows.
\end{proof}
It turns out that the $K_{2,4}$ term is also negligible.
\begin{lemma}\label{eliminatingK24lemma}
Under Setup \ref{K24setup},
\[\Hom(K_{2,4},U)=o(p^8).\]
\end{lemma}
\begin{proof}
By the triangle inequality and Cauchy-Schwarz,
\begin{align*}
|\Hom(K_{2,4},U)| & \leq\Hom(K_{2,4},|U|) \\ & =\displaystyle\int_{[0,1]^2}\left(\displaystyle\int_0^1 |U(x,z)U(y,z)|dz\right)^4 dx dy \\ & \leq\displaystyle\int_{[0,1]^2}\left(S_2(x)S_2(y)\right)^4 dx dy \\ & =\left(\displaystyle\int_0^1 S_2(x)^4dx\right)^2,
\end{align*}
where $S_2(x)=\left(\displaystyle\int_0^1 U(x,y)^2 dy\right)^{\frac{1}{2}}$.
Now, for all $x\in [0,1]$,
\begin{align*}
\displaystyle\int_0^1 U(x,y)^2 dy & =\displaystyle\int_0^1 (W(x,y)^2-2p W(x,y)+p^2) dy \\ & \leq (1+o(1))\displaystyle\int_0^1 W(x,y) dy+p^2 \\ & =p+o(p)
\end{align*}
for all $x$, by (1) of Lemma \ref{K24consequenceslemma}. Thus $S_2(x)\leq (1+o(1))\sqrt{p}$ for all $x$, so
\begin{align*}
|\Hom(K_{2,4},U)| & \leq\left(\displaystyle\int_0^1 S_2(x)^4dx\right)^2 \\ & \leq (1+o(1))p^2\left(\displaystyle\int_0^1 S_2(x)^2 dx\right)^2 \\ & =(1+o(1))p^2\|U\|_2^4.
\end{align*}
Since $\|U\|_2\ll p^{\frac{3}{2}}$ by (7) of Lemma \ref{K24consequenceslemma}, we are done.
\end{proof}
Lemmas \ref{K0eliminatingsubgraphslemma} and \ref{eliminatingK24lemma} together imply the following.
\begin{cor}\label{K0onlylemma}
Under Setup \ref{K24setup},
\[\Hom(K_0,U)\geq (\delta-o(1))p^9\]
\end{cor}
\subsection{Breaking into High and Low Values}
In our integral expression for $\Hom(K_0,U)$, we would like to show that some copies of U (corresponding to certain edges of $K_0$) generally take high values of $U$ and some generally take low values. We thus make the following definition.
\begin{definition}\label{Uidef}
Taking $U$ as in Setup \ref{K24setup}, define symmetric and measurable $U_1,U_2:[0,1]^2\to [0,1]$ as follows.
\[U_1(x,y)=\begin{cases}U(x,y) & |U(x,y)|\geq p^{\frac{7}{8}} \\ 0 & |U(x,y)|\leq p^{\frac{7}{8}}\end{cases}\]
\[U_2(x,y)=\begin{cases}0 & |U(x,y)|\geq p^{\frac{7}{8}} \\ U(x,y) & |U(x,y)|\leq p^{\frac{7}{8}}\end{cases}\]
\end{definition}
\begin{remark}
Any threshold that looks like $p^a$ for $a$ sufficiently close to $1$ will work for our purposes.
\end{remark}
\begin{definition}
Label the vertices of $K_0$ as $v_1,v_2,w_1,w_2,w_3,w_4$ such that $v_1,v_2$ are the two vertices of degree $4$ and $w_1,w_2$ the two vertices of degree $3$ (so $w_1w_2$ is the `extra' edge added to $K_{2,4}$ to obtain $K_0$).
\end{definition}
\begin{remark}
In essence, the goal of this subsection is to show that the main contribution to $\Hom(K_0,U)$ must come when the edges $v_1w_3$, $v_1w_4$, $v_2w_3$, and $v_2w_4$ are sent into high values of $U$ (appearing in $U_1$), the edge $w_1w_2$ is sent into low values of $U$ (appearing in $U_2$), and the vertices $v_1$ and $v_2$ are sent into the hub of $U$. The first two items will be accomplished by the next lemma.
\end{remark}

Notice that $U=U_1+U_2$. Thus in our expression
\[\Hom(K_0,U)=\displaystyle\int_{[0,1]^6}\displaystyle\prod_{vw\in E(K_0)}U(x_v,x_w)\displaystyle\prod_{v\in V(K_0)}dx_v,\]
we may substitute $U_1+U_2$ where each copy of $U$ appears, breaking the integral into $512$ terms. Most of these terms turn out to be negligible.
\begin{lemma}\label{highlowlemma}
Assume Setup \ref{K24setup}. Take $f:E(K_0)\to\{1,2\}$. Then
\[\displaystyle\int_{[0,1]^6}\displaystyle\prod_{vv'\in E(K_0)}U_{f(vv')}(x_v,x_v')\displaystyle\prod_{v\in V(K_0)}dx_v=o(p^9)\]
unless $f(v_1w_3)=f(v_1w_4)=f(v_2w_3)=f(v_2w_4)=1$ and $f(w_1w_2)=2$, where the $U_i$ are defined as in Definition \ref{Uidef}.
\end{lemma}
This lemma in essence says that we must only consider when the $4$-cycle $v_1w_3v_2w_4$ uses the high values of $U$, and $w_1w_2$ uses the low values.
\begin{proof}
We crudely bound by noting that $U$ is $1$-bounded, ignoring all edges except the ones appearing in the lemma, and applying Cauchy-Schwarz. Suppose our desired integral is $\Omega(p^9)$. Then, by dropping the edges $v_1w_1$, $v_1w_2$, $v_2w_1$, and $v_2w_2$ and applying Cauchy-Schwarz on the resulting $4$-cycle,
\begin{align*}
\Omega(p^9) & =\displaystyle\int_{[0,1]^6}\displaystyle\prod_{vv'\in E(K_0)}U_{f(vv')}(x_v,x_v')\displaystyle\prod_{v\in V(K_0)}dx_v \\ & \leq\displaystyle\int_{[0,1]^6}U_{f(w_1w_2)}(x_{w_1},x_{w_2})\displaystyle\prod_{\substack{i\in\{1,2\} \\ j\in \{3,4\}}}U_{f(v_iw_j)}(x_{v_i},x_{w_j})\displaystyle\prod_{v\in V(K_0)}dx_v \\ & =\left(\displaystyle\int_{[0,1]^2}U_{f(w_1w_2)}(x,y)dx dy\right)\left(\displaystyle\int_{[0,1]^4}\displaystyle\prod_{\substack{i\in\{1,2\} \\ j\in \{3,4\}}}U_{f(v_iw_j)}(x_{v_i},x_{w_j})dx_{v_1}dx_{v_2}dx_{w_3}dx_{w_4}\right) \\ & \leq \|U_{f(w_1w_2)}\|_1\|U_{f(v_1w_3)}\|_2\|U_{f(v_1w_4)}\|_2\|U_{f(v_2w_3)}\|_2\|U_{f(v_2w_4)}\|_2.
\end{align*}

The idea of our argument will be that the equality case of Lemma \ref{momentlemma} occurs for large values, and the equality case of Lemma \ref{firstmomentlemma} occurs for small values, so for the expression above to be large, the desired values of $f$ must be what we want.

Note that $\| U_i\|_1\leq \|U\|_1\leq p^{3-o(1)}$ and $\|U_i\|_2\leq \|U\|_2\ll p^{\frac{3}{2}}$ for $i\in\{1,2\}$, by (6) and (7) of Lemma \ref{K24consequenceslemma}.

Therefore,
\[\Omega(p^9)=\|U_{f(w_1w_2)}\|_1\|U_{f(v_1w_3)}\|_2\|U_{f(v_1w_4)}\|_2\|U_{f(v_2w_3)}\|_2\|U_{f(v_2w_4)}\|_2\leq p^{\frac{15}{2}-o(1)}\|U_{f(v_1w_3)}\|_2.\]
Thus $\|U_{f(v_1w_3)}\|_2\geq p^{\frac{3}{2}+o(1)}$, so
\begin{equation}\label{v1w3equation}
\displaystyle\int_{[0,1]^2}\left(U_{f(v_1w_3)}(x,y)\right)^2 dx dy\geq p^{3+o(1)}.
\end{equation}
But for all $a$ with $|a|\leq p^{\frac{7}{8}}$, $a^2=p^{\Omega(1)}I_p(p+a)$. This is because by Lemma \ref{entropyapproxlemma}, $I_p(p+a)=\Theta\left(\frac{a^2}{p}\right)$ when $a=O(p)$, and when $a\gg p$, $I_p(p+a)=(1+o(1))a\log\frac{a}{p}\geq a\geq p^{-\frac{7}{8}}a^2$, as $a\leq p^{\frac{7}{8}}$. Thus since $U_2$ is $p^{\frac{7}{8}}$-bounded,
\[\displaystyle\int_{[0,1]^2}U_2(x,y)^2 dx dy\leq p^{\Omega(1)}I_p(W)=p^{3+\Omega(1)}.\]
Combining with (\ref{v1w3equation}) we see that we cannot have $f(v_1w_3)=2$, so $f(v_1w_3)=1$. The same logic holds for the edges $v_1w_4,v_2w_3,v_2w_4$, so we have proven that
\[f(v_1w_3)=f(v_1w_4)=f(v_2w_3)=f(v_2w_4)=1.\]
We return to the statement
\[\|U_{f(w_1w_2)}\|_1\|U_{f(v_1w_3)}\|_2\|U_{f(v_1w_4)}\|_2\|U_{f(v_2w_3)}\|_2\|U_{f(v_2w_4)}\|_2=\Omega(p^9)\]
which now becomes
\[\|U_{f(w_1w_2)}\|_1\|U_1\|_2^4=\Omega(p^9)\]
Since $\|U_1\|_2\leq\|U\|_2\ll p^{\frac{3}{2}}$ by (7) of Lemma \ref{K24consequenceslemma},
\begin{equation}\label{w1w2equation}
\|U_{f(w_1w_2)}\|_1\gg p^3.
\end{equation}
But for all $a$ with $a\geq p^{\frac{7}{8}}$, by Lemma \ref{entropyapproxlemma},
\[I_p(p+a)=(1+o(1))a\log\frac{a}{p}\geq (1+o(1))a\log p^{-\frac{1}{8}}=\Omega\left(a\log\frac{1}{p}\right),\]
so since $U_1$ only takes values either $0$ or at least $p^{\frac{7}{8}}$,
\begin{align*}
\|U_1\|_1 & =\displaystyle\int_{[0,1]^2}U_1(x,y)dx dy \\ & =O\left(\left(\log\frac{1}{p}\right)^{-1}\displaystyle\int_{[0,1]^2}I_p(W(x,y))dx dy\right) \\ & =O\left(\left(\log\frac{1}{p}\right)^{-1}I_p(W)\right) \\ & \ll p^3
\end{align*}
by (5) of Lemma \ref{K24consequenceslemma}. Combining with (\ref{w1w2equation}) yields that $f(w_1w_2)\neq 1$, so $f(w_1w_2)=2$, as desired.
\end{proof}
We have the following corollary, where we use $dx$ as a shorthand for $\displaystyle\prod_{v\in V(K_0)}dx_v$.
\begin{cor}\label{highlowcor}
Under Setup \ref{K24setup},
\[\displaystyle\int_{[0,1]^6}U_2(x_{w_1},x_{w_2})\left(\displaystyle\prod_{\substack{i\in\{1,2\} \\ j\in\{1,2\}}}U(x_{v_i},x_{w_j})\right)\left(\displaystyle\prod_{\substack{i\in\{1,2\} \\ j\in\{3,4\}}}U_1(x_{v_i},x_{w_j})\right)dx\geq (\delta-o(1))p^9,\]
with $U_i$ as in Definition \ref{Uidef}.
\end{cor}
\begin{proof}
By Corollary \ref{K0onlylemma}, $\Hom(K_0,U)\geq (\delta-o(1))p^9$. Expanding out the integral $\Hom(K_0,U)$ by substituting $U=U_1+U_2$, we obtain $512$ terms. The terms appearing in the left hand side of Corollary \ref{highlowcor} are exactly those where we chose $U_2$ for the $w_1w_2$ edge and $U_1$ for the edges in the $4$-cycle $v_1w_3v_2w_4$. But all other terms are $o(p^9)$ by Lemma \ref{highlowlemma}, so the conclusion follows.
\end{proof}
We have now shown the first two specifications in the remark at the beginning of the section--that is, that $v_1w_3$, $v_1w_4$, $v_2w_3$, and $v_2w_4$ use high values of $U$ and that $w_1w_2$ uses low values of $U$. We now turn our attention to the remaining specification: that $v_1$ and $v_2$ generally are sent into the hub of $U$. We define a hub similarly to in Section \ref{K23section}.

\begin{definition}\label{K24hubdef}
Under Setup \ref{K24setup}, for any $b=b(p)\geq 0$, define
\[B_b:=\left\{x\in [0,1]:\displaystyle\int_0^1 U(x,y)^2 dy\geq b\right\}.\]
\end{definition}
We show that very little of the contribution to the left side of Corollary \ref{highlowcor} comes when $v_1\notin B_b$.
\begin{lemma}\label{highlowlemma2}
Under Setup \ref{K24setup}, if $b\ll p^{\frac{9}{8}}$, then
\[\displaystyle\int_{\substack{v_1\in\overline{B_b} \\ v_2,w_1,w_2,w_3,w_4\in [0,1]}}\left|U_2(x_{w_1},x_{w_2})\left(\displaystyle\prod_{\substack{i\in\{1,2\} \\ j\in\{1,2\}}}U(x_{v_i},x_{w_j})\right)\left(\displaystyle\prod_{\substack{i\in\{1,2\} \\ j\in\{3,4\}}}U_1(x_{v_i},x_{w_j})\right)\right|dx=o(p^9).\]
\end{lemma}
\begin{proof}
We replace the copy of $U_2$ by its maximum to yield
\begin{align*}
\displaystyle\int_{\substack{v_1\in\overline{B_b} \\ v_2,w_1,w_2,w_3,w_4\in [0,1]}}\left|U_2(x_{w_1},x_{w_2})\left(\displaystyle\prod_{\substack{i\in\{1,2\} \\ j\in\{1,2\}}}U(x_{v_i},x_{w_j})\right)\left(\displaystyle\prod_{\substack{i\in\{1,2\} \\ j\in\{3,4\}}}U_1(x_{v_i},x_{w_j})\right)\right|dx & \\ \leq\|U_2\|_{\infty}\displaystyle\int_{\substack{v_1\in\overline{B_b} \\ v_2,w_1,w_2,w_3,w_4\in [0,1]}}\left|\left(\displaystyle\prod_{\substack{i\in\{1,2\} \\ j\in\{1,2\}}}U(x_{v_i},x_{w_j})\right)\left(\displaystyle\prod_{\substack{i\in\{1,2\} \\ j\in\{3,4\}}}U_1(x_{v_i},x_{w_j})\right)\right|dx & \\ =\|U_2\|_{\infty}\displaystyle\int_{\substack{v_1\in\overline{B_b} \\ v_2,w_3,w_4\in [0,1]}}\left(\displaystyle\int_0^1|U(x_{v_1},y)U(x_{v_2},y)|dy\right)^2\left|\displaystyle\prod_{\substack{i\in\{1,2\} \\ j\in\{3,4\}}}U_1(x_{v_i},x_{w_j})\right|dx_{v_1}dx_{v_2}dx_{w_3}dx_{w_4}.&
\end{align*}
Now, by Cauchy-Schwarz and (1) of Lemma \ref{K24consequenceslemma}, whenever $v_1\in\overline{B_b}$,
\begin{align*}
\left(\displaystyle\int_0^1|U(x_{v_1},y)U(x_{v_2},y)|dy\right)^2 & \leq\left(\displaystyle\int_0^1 U(x_{v_1},y)^2 dy\right)\left(\displaystyle\int_0^1 U(x_{v_2},y)^2 dy\right) \\ & \leq b\displaystyle\int_0^1\left(W(x_{v_2},y)^2-2pW(x_{v_2},y)+p^2\right)dy \\ & \leq b\left((1-2p)\displaystyle\int_0^1 W(x_{v_2},y)dy+p^2\right) \\ & =(1+o(1))bp,
\end{align*}
as $\displaystyle\int_0^1 U(x_{v_1},y)^2\leq b$ by the definition of $B_b$.

Therefore our desired integral is bounded above by
\[(1+o(1))bp\|U_2\|_{\infty}\displaystyle\int_{\substack{v_1\in\overline{B_b} \\ v_2,w_3,w_4\in [0,1]}}\left|\displaystyle\prod_{\substack{i\in\{1,2\} \\ j\in\{3,4\}}}U_1(x_{v_i},x_{w_j})\right|dx_{v_1}dx_{v_2}dx_{w_3}dx_{w_4}.\]
Now allowing $v_1$ to range over $[0,1]$ and applying Cauchy-Schwarz repeatedly, we have the upper bound
\[(1+o(1))bp\|U_2\|_{\infty}\|U_1\|_2^4.\]
By (7) of Lemma \ref{K24consequenceslemma}, $\|U_1\|_2\leq \|U\|_2\ll p^{\frac{3}{2}}$, so we may bound the left side of the lemma as
\[o\left(bp^7\|U_2\|_{\infty}\right).\]
But $b\|U_2\|_{\infty}\ll p^{\frac{9}{8}}p^{\frac{7}{8}}=p^2$, so our bound is $o(p^9)$ and the lemma is proven.
\end{proof}
\begin{cor}\label{highlowcor3}
Assume Setup \ref{K24setup} and recall from Definitions \ref{Uidef} and \ref{K24hubdef} the definitions of $U_1$, $U_2$, and $B_b$. If $b\ll p^{\frac{9}{8}}$, then
\[\displaystyle\int_{\substack{v_1,v_2\in B_b \\ w_1,w_2,w_3,w_4\in [0,1]}}U_2(x_{w_1},x_{w_2})\left(\displaystyle\prod_{\substack{i\in\{1,2\} \\ j\in\{1,2\}}}U(x_{v_i},x_{w_j})\right)\left(\displaystyle\prod_{\substack{i\in\{1,2\} \\ j\in\{3,4\}}}U_1(x_{v_i},x_{w_j})\right)dx\geq (\delta-o(1))p^9.\]
\end{cor}
\begin{proof}
With $x=(x_v)_{v\in V(K_0)}$, let $F(x)=U_2(x_{w_1},x_{w_2})\left(\displaystyle\prod_{\substack{i\in\{1,2\} \\ j\in\{1,2\}}}U(x_{v_i},x_{w_j})\right)\left(\displaystyle\prod_{\substack{i\in\{1,2\} \\ j\in\{3,4\}}}U_1(x_{v_i},x_{w_j})\right)$. By Corollary \ref{highlowcor},
\[\displaystyle\int_{[0,1]^6}F(x)dx\geq (\delta-o(1))p^9,\]
and by Lemma \ref{highlowlemma2},
\[\displaystyle\int_{\substack{v_1\in\overline{B_b} \\ v_2,w_1,w_2,w_3,w_4\in [0,1]}}|F(x)|dx=o(p^9).\]
Thus (since $F$ is symmetric in $x_{v_1}$ and $x_{v_2}$),
\begin{align*}
\displaystyle\int_{\substack{v_1,v_2\in B_b \\ w_1,w_2,w_3,w_4\in [0,1]}}F(x)dx & =\displaystyle\int_{[0,1]^6}F(x)dx-\displaystyle\int_{\substack{v_1\in\overline{B_b} \\ v_2,w_1,w_2,w_3,w_4\in [0,1]}}F(x)dx-\displaystyle\int_{\substack{v_1\in B_b,v_2\in\overline{B_b} \\ w_1,w_2,w_3,w_4\in [0,1]}}F(x)dx \\ & \geq\displaystyle\int_{[0,1]^6}F(x)dx-\displaystyle\int_{\substack{v_1\in\overline{B_b} \\ v_2,w_1,w_2,w_3,w_4\in [0,1]}}|F(x)|dx-\displaystyle\int_{\substack{v_2\in\overline{B_b} \\ v_1,w_1,w_2,w_3,w_4\in [0,1]}}|F(x)|dx \\ & \geq (\delta-o(1))p^9-o(p^9)-o(p^9) \\ & =(\delta-o(1))p^9,
\end{align*}
as desired.
\end{proof}
We have in essence restricted to homomorphisms where $v_1$ and $v_2$ are sent into the hub $B_b$, $w_1w_2$ uses the low values of $U$, and the $4$-cycle $v_1w_3v_2w_4$ uses the high values of $U$. To conclude the subsection, we will manipulate the bound from the previous corollary into a more manageable form.
\begin{cor}\label{highlowcor2}
Assume Setup \ref{K24setup}. Take $0\leq b\ll p^{\frac{9}{8}}$. Recall from Definitions \ref{Uidef} and \ref{K24hubdef} the definitions of $U_1$, $U_2$, and $B_b$ and let $U_1^b$ be the restriction of $U_1$ to $B_b\times [0,1]$. Then
\[\|U_1^b\|_2^4\displaystyle\sup_{x_{v_1},x_{v_2}\in [0,1]}\left|\displaystyle\int_{[0,1]^2}U_2(x_{w_1},x_{w_2})\displaystyle\prod_{\substack{i\in\{1,2\} \\ j\in\{1,2\}}}U(x_{v_i},x_{w_j})dx_{w_1}dx_{w_2}\right|\geq (\delta-o(1))p^9.\]
\end{cor}
\begin{proof}
We will bound the left side of Corollary \ref{highlowcor3} above. We have
\begingroup
\allowdisplaybreaks
\begin{align*}
(\delta-o(1))p^9 & \leq\displaystyle\int_{\substack{v_1,v_2\in B_b \\ w_1,w_2,w_3,w_4\in [0,1]}}U_2(x_{w_1},x_{w_2})\left(\displaystyle\prod_{\substack{i\in\{1,2\} \\ j\in\{1,2\}}}U(x_{v_i},x_{w_j})\right)\left(\displaystyle\prod_{\substack{i\in\{1,2\} \\ j\in\{3,4\}}}U_1(x_{v_i},x_{w_j})\right)dx \\ & \leq\displaystyle\int_{B_b^2}\left|\displaystyle\int_{[0,1]^2}\displaystyle\prod_{\substack{i\in\{1,2\} \\ j\in\{3,4\}}}U_1(x_{v_i},x_{w_j})dx_{w_3}dx_{w_4}\right| \\ & \cdot\left|\displaystyle\int_{[0,1]^2}U_2(x_{w_1},x_{w_2})\displaystyle\prod_{\substack{i\in\{1,2\} \\ j\in\{1,2\}}}U(x_{v_i},x_{w_j})dx_{w_1}dx_{w_2}\right|dx_{v_1}dx_{v_2} \\ & \leq\displaystyle\int_{B_b^2}\left|\displaystyle\int_{[0,1]^2}\displaystyle\prod_{\substack{i\in\{1,2\} \\ j\in\{3,4\}}}U_1(x_{v_i},x_{w_j})dx_{w_3}dx_{w_4}\right|dx_{v_1}dx_{v_2}\\ & \cdot\displaystyle\sup_{x_{v_1},x_{v_2}\in [0,1]}\left|\displaystyle\int_{[0,1]^2}U_2(x_{w_1},x_{w_2})\displaystyle\prod_{\substack{i\in\{1,2\} \\ j\in\{1,2\}}}U(x_{v_i},x_{w_j})dx_{w_1}dx_{w_2}\right| \\ & \leq\|U_1^b\|_2^4\displaystyle\sup_{x_{v_1},x_{v_2}\in [0,1]}\left|\displaystyle\int_{[0,1]^2}U_2(x_{w_1},x_{w_2})\displaystyle\prod_{\substack{i\in\{1,2\} \\ j\in\{1,2\}}}U(x_{v_i},x_{w_j})dx_{w_1}dx_{w_2}\right|,
\end{align*}
\endgroup
where the first inequality is simply the triangle inequality, the second is by replacing the function $A(v_1,v_2):=\left|\displaystyle\int_{[0,1]^2}U_2(x_{w_1},x_{w_2})\displaystyle\prod_{\substack{i\in\{1,2\} \\ j\in\{1,2\}}}U(x_{v_i},x_{w_j})dx_{w_1}dx_{w_2}\right|$ by its supremum everywhere, and the third is by Cauchy-Schwarz.
\end{proof}
\subsection{Bounding by Entropy}
For the rest of Section \ref{K24section}, we use the following extended setup.
\begin{setup}\label{K24extendedsetup}
Take $K_0$, $\delta$, $p\to 0$, $W$, and $U=W-p$ as in Setup \ref{K24setup}. For $b=b(p)\geq 0$, define
\[B_b:=\left\{x\in [0,1]:\displaystyle\int_0^1 U(x,y)^2 dy\geq b\right\}\]
as in Definition \ref{K24hubdef}. Define $U_1,U_2$ with $U_1+U_2=U$ as in Definition \ref{Uidef}. Let $U_1^b$ be the restriction of $U$ to $B_b\times [0,1]$.

Finally, set
\[c_1=p^{-3}\|U_1^b\|_2^2\]
and
\[c_2=p^{-3}\displaystyle\sup_{x_{v_1},x_{v_2}\in [0,1]}\left|\displaystyle\int_{[0,1]^2}U_2(x_{w_1},x_{w_2})\displaystyle\prod_{\substack{i\in\{1,2\} \\ j\in\{1,2\}}}U(x_{v_i},x_{w_j})dx_{w_1}dx_{w_2}\right|.\]
Note that $c_1$ depends on $b$.
\end{setup}
We may restate Corollary \ref{highlowcor2} given these new definitions.
\begin{highlowcor2}[Restated]
Under Setup \ref{K24extendedsetup}, if $b\ll p^{\frac{9}{8}}$, then
\[c_1^2c_2\geq\delta-o(1).\]
\end{highlowcor2}
We will now bound $c_1$ and $c_2$ in terms of the respective entropies $I_p(p+U_1)$ and $I_p(p+U_2)$. To bound $c_2$, the idea is to rewrite
\[\displaystyle\int_{[0,1]^2}U_2(x_{w_1},x_{w_2})\displaystyle\prod_{\substack{i\in\{1,2\} \\ j\in\{1,2\}}}U(x_{v_i},x_{w_j})dx_{w_1}dx_{w_2}\]
as $\displaystyle\int U_2(x,y)a(x)a(y)dx dy$, where $a(x)=\displaystyle\int_0^1 U(x_{v_1},x)U(x_{v_2},x) dx$. Now, if $a$ is $\{0,1\}$-valued, we are integrating $U_2$ over some box, and since $I_p$ is convex, $I_p(p+U_2)$ should be minimized when $U_2$ is constant on that box and zero outside it by a smoothing argument. Of course, $a$ is not necessarily $\{0,1\}$-valued, but a similar argument will work nonetheless.

We apply this process via the following two lemmas, the first of which restricts the function $a$ and the second of which applies the smoothing argument.
\begin{lemma}\label{easycauchylemma}
Under Setup \ref{K24setup}, for all $x,y\in [0,1]$,
\[\displaystyle\int_0^1 |U(x,z)U(y,z)|dz\leq p+o(p).\]
\end{lemma}
\begin{proof}
The proof is exactly that given in the proof of Claim \ref{plemma} of Lemma \ref{intersectionlemma}. That is, by Cauchy-Schwarz,
\[\displaystyle\int_0^1 |U(x,z)U(y,z)|dz\leq\left(\displaystyle\int_0^1 U(x,z)^2 dz\right)^{\frac{1}{2}}\left(\displaystyle\int_0^1 U(y,z)^2 dz\right)^{\frac{1}{2}},\]
and for all $x\in [0,1]$,
\begin{align*}
\displaystyle\int_0^1 U(x,z)^2 dz & =\displaystyle\int_0^1 W(x,z)^2 dz-2p\displaystyle\int_0^1 W(x,z) dz+p^2 \\ & \leq (1-2p)\displaystyle\int_0^1 W(x,z) dz+p^2 \\ & =p+o(p),
\end{align*}
as $\displaystyle\int_0^1 W(x,z) dz=p+o(p)$ by (1) of Lemma \ref{K24consequenceslemma}.
\end{proof}
\begin{lemma}\label{smoothinglemma}
Assume Setup \ref{K24setup} and let $a,b:[0,1]\to [0,1]$ be measurable with $\|a\|_1,\|b\|_1\leq p+o(p)$.
\[I_p\left(p+p^{-2}\displaystyle\int_{[0,1]^2}a(x)b(y)|U_2(x,y)| dx dy\right)\leq (1+o(1))p^{-2}I_p(p+U_2)\]
\end{lemma}
\begin{proof}
We use a smoothing-type argument.

We will first choose $a,b:[0,1]\to [0,1]$ so that the left side is maximized. Since $I_p$ is increasing on $[p,1]$, this occurs when $\displaystyle\int_{[0,1]^2}a(x)b(y)|U_2(x,y)| dx dy$ is maximized.

Suppose we fix $\|a\|_1$, and fix $b$. When we set $c(x)=\displaystyle\int_0^1 b(y)|U_2(x,y)|dy$, 
\begin{align*}
\displaystyle\int_{[0,1]^2}a(x)b(y)|U_2(x,y)| dx dy & =\displaystyle\int_0^1 a(x)c(x) dx \\ & \leq\displaystyle\int_0^{\|a\|_1}c^*(x)dx
\end{align*}
by the Rearrangement Inequality, where $c^*$ is the symmetric decreasing rearrangement of $c$. Furthermore, this maximum is attainable with $a$ only taking values in $\{0,1\}$, as if $c^*(\|a\|_1)=t$, then $m((c^*)^{-1}((t,1]))\leq t\leq m((c^*)^{-1}([t,1]))$, so there is some $A$ of measure $t$ such that $(c^*)^{-1}((t,1])\subseteq A\subseteq (c^*)^{-1}([t,1])$, and letting $a=1$ on $A$ and $a=0$ otherwise, $a$ and $c$ are similarly sorted by definition, so
\begin{align*}
\displaystyle\int_{[0,1]^2}a(x)b(y)|U_2(x,y)| dx dy & =\displaystyle\int_0^1 a(x)c(x) dx \\ & =\displaystyle\int_0^1 a^*(x)c^*(x) dx \\ & =\displaystyle\int_0^{\|a\|_1}c^*(x)dx.
\end{align*}
Thus $\displaystyle\int_{[0,1]^2}a(x)b(y)|U_2(x,y)| dx dy$ is maximized when $a$ only takes values in $\{0,1\}$ and similarly for $b$. Take $a,b$ attaining this maximum with values in $\{0,1\}$ and let $A$ and $B$ be the sets on which $a=1$ and $b=1$, respectively. Then $m(A),m(B)\leq p+o(p)$ by the lemma conditions. Then the left hand side of the lemma becomes
\[I_p\left(p+p^{-2}\displaystyle\int_{A\times B}|U_2(x,y)| dx dy\right).\]

Now, notice that $I_p'''(x)=\frac{-1}{x^2}+\frac{1}{(1-x)^2}\leq 0$ on $[0,2p]$. Thus $I_p''(p-x)\geq I_p''(p+x)$ for all $x\in [0,p]$. Using that $I_p(p)=I_p'(p)=0$ and integrating twice, we see that $I_p(p-x)\geq I_p(p+x)$ for all $x\in [0,p]$. Thus $I_p(p+U_2)\geq I_p(p+|U_2|)$.

So to finish the proof of the lemma, it suffices to show that for all $A,B$ with $m(A),m(B)=p+o(p)$,
\[I_p\left(p+p^{-2}\displaystyle\int_{A\times B}|U_2(x,y)| dx dy\right)\leq (1+o(1))p^{-2}I_p(p+|U_2|).\]
But $I_p(p+x)$ is convex, so by Jensen's inequality,
\begin{align*}
I_p\left(p+p^{-2}\displaystyle\int_{A\times B}|U_2(x,y)| dx dy\right) & =I_p\left(p+\frac{\displaystyle\int_{A\times B}(1+o(1))|U_2(x,y)| dx dy}{|A|||B|}\right) \\ & \leq\frac{1}{|A||B|}\displaystyle\int_{A\times B}I_p(p+(1+o(1))|U_2(x,y)|)dx dy \\ & \leq (1+o(1))p^{-2}I_p(p+(1+o(1))|U_2|) \\ & =(1+o(1))p^{-2}I_p(p+|U_2|),
\end{align*}
since multiplying $x$ by $(1+o(1))$ only changes $I_p(p+x)$ by a factor of $1+o(1)$ (as $I_p$ is continuous and nonzero unless $x=0$). The lemma follows.
\end{proof}
We now may obtain our bound on $c_2$.
\begin{cor}\label{b2boundcor}
Under Setup \ref{K24extendedsetup},
\[I_p(p+U_2)\geq (1-o(1))p^2I_p(p+pc_2).\]
\end{cor}
\begin{proof}
By the definition of supremum, for any $\epsilon_0=\epsilon_0(p)>0$ we may take $x_{v_1},x_{v_2}$ that attain within a $1+\epsilon_0$ factor of the supremum in the definition of $c_2$, so that
\[c_2\leq (1+\epsilon_0)p^{-3}\left|\displaystyle\int_{[0,1]^2}U_2(x_{w_1},x_{w_2})\displaystyle\prod_{\substack{i\in\{1,2\} \\ j\in\{1,2\}}}U(x_{v_i},x_{w_j})dx_{w_1}dx_{w_2}\right|.\]
Let $a:[0,1]\to [0,1]$ be given by $a(x)=\sqrt{1+\epsilon_0}|U(x_{v_1},x)U(x_{v_2},x)|$. Then taking $\epsilon_0=o(1)$, $\|a\|_1\leq p+o(p)$ by Lemma \ref{easycauchylemma}. By the triangle inequality,
\begin{align*}
c_2 & \leq (1+\epsilon_0)p^{-3}\displaystyle\int_{[0,1]^2}|U_2(x_{w_1},x_{w_2})|\displaystyle\prod_{\substack{i\in\{1,2\} \\ j\in\{1,2\}}}|U(x_{v_i},x_{w_j})|dx_{w_1}dx_{w_2} \\ & =p^{-3}\displaystyle\int_{[0,1]^2}a(x_{w_1})a(x_{w_2})|U_2(x_{w_1},x_{w_2})|dx_{w_1}dx_{w_2}.
\end{align*}
Therefore, by Lemma \ref{smoothinglemma},
\begin{align*}
I_p(p+pc_2) & \leq I_p\left(p+p^{-2}\displaystyle\int_{[0,1]^2}a(x_{w_1})a(x_{w_2})|U_2(x_{w_1},x_{w_2})|dx_{w_1}dx_{w_2}\right) \\ & \leq (1+o(1))p^{-2}I_p(p+U_2),
\end{align*}
which easily rearranges to the corollary.
\end{proof}
The bound on $c_1$ in terms of entropy is much easier to show.
\begin{lemma}\label{b1boundlemma}
Under Setup \ref{K24extendedsetup}, if $b=p^{\frac{3}{2}-\Omega(1)}$, then
\[I_p(p+U_1)\geq (2-o(1))\left(p^3\log\frac{1}{p}\right)c_1-o(p^{3+\Omega(1)}).\]
\end{lemma}
\begin{proof}
By Lemma \ref{momentlemma} with $c=1$, since $U_1$ takes no values smaller than $p^{\frac{7}{8}}$, $I_p(p+U_1(x,y))\geq (1-o(1))U_1(x,y)^2\log\frac{1}{p}$ for all $x,y\in [0,1]$. Integrating,
\begin{align}
\nonumber I_p(p+U_1) & \geq (1-o(1))\log\frac{1}{p}\displaystyle\int_{[0,1]^2} U_1(x,y)^2 dx dy \\ & \nonumber\geq (1-o(1))\log\frac{1}{p}\left(\displaystyle\int_{B_b\times [0,1]} U_1(x,y)^2 dx dy+\displaystyle\int_{[0,1]\times B_b} U_1(x,y)^2 dx dy-\displaystyle\int_{B_b^2} U_1(x,y)^2 dx dy\right) \\ & \nonumber\geq (1-o(1))\log\frac{1}{p}\left(2\|U_1^b\|_2^2-m(B_b)^2\right) \\ & \label{finallyfactorof2} =(2-o(1))p^3\log\frac{1}{p}c_1-(1+o(1))\log\frac{1}{p}m(B_b)^2,
\end{align}
by the definition of $c_1$.

Now, by the definition of $B_b$ and (7) of Lemma \ref{K24consequenceslemma},
\begin{align*}
bm(B_b) & \leq\displaystyle\int_{B_b\times [0,1]}U(x,y)^2 dx dy \\ & \leq \|U\|_2^2 \\ & \ll p^3.
\end{align*}Thus $m(B_b)\ll\frac{p^3}{b}$, so substituting into (\ref{finallyfactorof2}) yields
\[I_p(p+U_1)\geq (2-o(1))p^3\log\frac{1}{p}c_1-o\left(\frac{p^6\log\frac{1}{p}}{b^2}\right).\]
Substituting $b=p^{\frac{3}{2}-\Omega(1)}$ and noting that $\log\frac{1}{p}=p^{-o(1)}$, the lemma follows.
\end{proof}
We combine our bounds into the following bound on the entropy of $W$.

\begin{lemma}\label{optimizationproblemprelemma}
Assume Setup \ref{K24extendedsetup}. If $b=p^{\frac{5}{4}}$, then
\[I_p(W)\geq (1-o(1))\left(\left(2p^3\log\frac{1}{p}\right)c_1+p^2I_p(p+pc_2)\right).\]
\end{lemma}
\begin{proof}
Notice that $I_p(W)=I_p(p+U_1)+I_p(p+U_2)$, since at each point either $I_p(p+U_1)$ or $I_p(p+U_2)$ is $0$ and the other is $I_p(W)$. Thus since $b=p^{\frac{5}{4}}=p^{\frac{3}{2}-\Omega(1)}$, Corollary \ref{b2boundcor} and Lemma \ref{b1boundlemma} together imply that
\[I_p(W)\geq (1-o(1))\left(\left(2p^3\log\frac{1}{p}\right)c_1+p^2I_p(p+pc_2)\right)-o\left(p^{3+\Omega(1)}\right).\]

By Corollary \ref{highlowcor2}, since $b=p^{\frac{5}{4}}\ll p^{\frac{9}{8}}$, $c_1^2c_2\geq\delta-o(1)$. Thus either $c_1=\Omega(1)$ or $c_2=\Omega(1)$. In the first case, $\left(2p^3\log\frac{1}{p}\right)c_1\geq p^{3+o(1)}$, and in the latter, $p^2I_p(p+pc_2)\geq p^{3+o(1)}$ by Lemma \ref{entropyapproxlemma}. Thus the $o(p^{3+\Omega(1)})$ term is negligible and we have proven the lemma.
\end{proof}
Since we chose $W$ in Setup \ref{K24extendedsetup} such that
\[I_p(W)\leq (1+o(1))\displaystyle\inf_{\substack{\Hom(K_0,W')\geq (1+\delta)p^9 \\ W'\text{ }p\text{-regular}}}I_p(W'),\]
to show Theorem \ref{K24correctconstant} it suffices to show that under the setup,
\[I_p(W)\geq (1-o(1))(18\delta)^{\frac{1}{3}}p^3\left(\log\frac{1}{p}\right)^{\frac{2}{3}}\left(\log\log\frac{1}{p}\right)^{\frac{1}{3}}.\]
We have, taking $b=p^{\frac{5}{4}}$ and using Corollary \ref{highlowcor2} and Lemma \ref{optimizationproblemprelemma}, constructed $c_1$ and $c_2$ such that $c_1^2c_2\geq\delta-o(1)$ and $I_p(W)\geq (1-o(1))\left(\left(2p^3\log\frac{1}{p}\right)c_1+p^2I_p(p+pc_2)\right)$. Thus (mulitplying $c_1$ by a $1+o(1)$ factor so that $c_1^2c_2\geq\delta$), to finish the proof of the lower bound of Theorem \ref{K24correctconstant}, it only remains to show the following lemma.
\begin{lemma}\label{optimizationproblemlemma}
Let $\delta>0$ be a constant and let $p\to 0$. Then
\[\displaystyle\min_{c_1^2c_2\geq\delta}\left(\left(2p^3\log\frac{1}{p}\right)c_1+p^2I_p(p+pc_2)\right)=(1-o(1))(18\delta)^{\frac{1}{3}}p^3\left(\log\frac{1}{p}\right)^{\frac{2}{3}}\left(\log\log\frac{1}{p}\right)^{\frac{1}{3}}.\]
\end{lemma}
\begin{remark}
Notice how the optimization problem in Lemma \ref{optimizationproblemlemma} tracks with the graphon in Figure \ref{K24graphonfigure}, where we set $a(p)=p^2c_1$ and $b(p)=p+pc_2$.
\end{remark}
\subsection{Optimizing}
We finish by proving Lemma \ref{optimizationproblemlemma}. Fix $\delta$ and let $p\to 0$.

If we take $c_1=\Theta\left(\left(\log\log\frac{1}{p}\right)^{-1}\right)$ and $c_2=\Theta\left(\left(\log\log\frac{1}{p}\right)^2\right)$ such that $c_1^2c_2\geq\delta$, then clearly $\left(2p^3\log\frac{1}{p}\right)c_1=o\left(p^3\log\frac{1}{p}\right)$. Furthermore, since $c_2\gg 1$, $I_p(p+pc_2)=\Theta(pc_2\log c_2)$ by Lemma \ref{entropyapproxlemma}, so $p^2I_p(p+pc_2)=\Theta(p^3c_2\log c_2)\ll p^3\log\frac{1}{p}$ as well for our choice of $c_2$. Thus our minimum
\begin{equation}\label{optimizeequation}
\displaystyle\min_{c_1^2c_2\geq\delta}\left(\left(2p^3\log\frac{1}{p}\right)c_1+p^2I_p(p+pc_2)\right)
\end{equation}
is $o\left(p^3\log\frac{1}{p}\right)$.

So for any $c_1,c_2$ that optimize (\ref{optimizeequation}), we must have $c_1\ll 1$, so $c_2\gg 1$. Thus we may apply Lemma \ref{entropyapproxlemma} to show that
\begin{equation}\label{optimizeequation2}
\displaystyle\min_{c_1^2c_2\geq\delta}\left(\left(2p^3\log\frac{1}{p}\right)c_1+p^2I_p(p+pc_2)\right)\geq (1-o(1))\displaystyle\min_{c_1^2c_2\geq\delta}\left(\left(2p^3\log\frac{1}{p}\right)c_1+p^3c_2\log c_2\right).
\end{equation}
We would like to normalize $c_1$ and $c_2$ to be constants. To that end, let
\[c_1=\left(\frac{\log\frac{1}{p}}{\log\log\frac{1}{p}}\right)^{-\frac{1}{3}}d_1\]
and
\[c_2=\left(\frac{\log\frac{1}{p}}{\log\log\frac{1}{p}}\right)^{\frac{2}{3}}d_2\]
Then the right side of (\ref{optimizeequation2}) becomes
\begin{align}
& \nonumber(1-o(1))\displaystyle\min_{d_1^2d_2\geq\delta}\left(2p^3\left(\log\frac{1}{p}\right)^{\frac{2}{3}}\left(\log\log\frac{1}{p}\right)^{\frac{1}{3}}d_1\right. \\ & \nonumber \left.+p^3\left(\log\frac{1}{p}\right)^{\frac{2}{3}}\left(\log\log\frac{1}{p}\right)^{-\frac{2}{3}}\left(\frac{2}{3}\log\log\frac{1}{p}-\frac{2}{3}\log\log\log\frac{1}{p}+\log d_2\right)d_2\right) \\ & \label{optimizeequation3}=(1-o(1))p^3\left(\log\frac{1}{p}\right)^{\frac{2}{3}}\left(\log\log\frac{1}{p}\right)^{\frac{1}{3}}\displaystyle\min_{d_1^2d_2\geq\delta}\left(2d_1+\frac{2}{3}d_2+d_2\frac{\log d_2}{\log\log\frac{1}{p}}\right),
\end{align}
as the $\log\log\log\frac{1}{p}$ term is negligible.

Since $d_2\log d_2\geq -\frac{1}{e}$ for all $d_2\in (0,\infty]$ and $2d_1+\frac{2}{3}d_2=\Omega(1)$ whenever $d_1^2d_2\geq\delta$, we may drop the $d_2\frac{d_2}{\log\log\frac{1}{p}}$ term altogether, so (\ref{optimizeequation3}) is lower bounded by
\[(1-o(1))p^3\left(\log\frac{1}{p}\right)^{\frac{2}{3}}\left(\log\log\frac{1}{p}\right)^{\frac{1}{3}}\displaystyle\min_{d_1^2d_2\geq\delta}\left(2d_1+\frac{2}{3}d_2\right).\]
Thus to finish the proof of Lemma \ref{optimizationproblemlemma}, it suffices to show that
\begin{equation}\label{amgmequation}
\displaystyle\min_{d_1^2d_2\geq\delta}\left(2d_1+\frac{2}{3}d_2\right)=(18\delta)^{\frac{1}{3}}.
\end{equation}
But $2d_1+\frac{2}{3}d_2=\frac{3d_1+3d_1+2d_2}{3}\geq (18d_1^2d_2)^{\frac{1}{3}}\geq (18\delta)^{\frac{1}{3}}$ by the AM-GM inequality, and when $d_1=\left(\frac{2\delta}{3}\right)^{\frac{1}{3}}$ and $d_2=\left(\frac{9\delta}{4}\right)^{\frac{1}{3}}$ we obtain
\[2d_1+\frac{2}{3}d_2=2\left(\frac{2\delta}{3}\right)^{\frac{1}{3}}+\left(\frac{2\delta}{3}\right)^{\frac{1}{3}}=3\left(\frac{2\delta}{3}\right)^{\frac{1}{3}}=(18\delta)^{\frac{1}{3}}.\]
So (\ref{amgmequation}) holds, and we have proven Lemma \ref{optimizationproblemlemma} and thus the lower bound of Theorem \ref{K24correctconstant} as well.
\section{Proof of Theorem \ref{ldpupperboundthm}}\label{ldpupperboundsection}
We modify the proof from Section 2.2 of \cite{BD}, itself a slight modification of the proof of Theorem 1.1 from \cite{CD}. The main reason why the proof does not apply verbatim is that $a_{n,p}:=p^{\Delta(K)}n^2$ is no longer the order of $\Phi_n^d(K,t)$. Thus we must verify that in the use of Corollary 2.2 of \cite{CD}, we may modify our construction to have $\exp\left(o\left(\Phi_n^d(K,t)\right)\right)$ convex sets and exceptional set $\mathcal{E}$ satisfying $\mu_p(\mathcal{E})\leq\exp\left(-(1+o(1))\Phi_n^d(K,t)\right)$.

In applying Theorem 3.4 of \cite{CD} to do this, we may take the same choice $\delta_0=\frac{\epsilon_0}{4C_{\star}}p^{\Delta_{\star}(K)}$. Choosing $L,\Delta$ so that $Lp^{\Delta}=\frac{\Phi_n^d(K,t)}{n^2\log\frac{1}{p}}$ makes the exceptional set $\mathcal{E}_0$ sufficiently small, and taking $k=\left\lceil L(p^{\Delta}/\delta_0^2)\log\frac{1}{p}\right\rceil$, we have that $k=\frac{\Phi_n^d(K,t)}{n^2\delta_0^2}+O(1)$, so
\[\log N=O\left(\left(\frac{\Phi_n^d(K,t)}{n\delta_0^2}+n\right)\log\frac{3n}{\delta_0}\right)=O\left(\frac{\Phi_n^d(K,t)\log n}{n\delta_0^2}+n\log n\right).\]
(We use the letter $L$ instead of $K$ as in \cite{CD} as we have already used $K$ as the name of our graph.) Thus we have few enough convex sets as long as $\delta_0^2\gg n^{-1}\log n$ and $\Phi_n^d(K,t)\gg n\log n$. The former is true by our bound on $p$. To prove the latter, note that$\Phi_n^d(K,t)=\Omega(n^2p^{\delta(K)}\log\frac{1}{p})$ by Lemma 6.1 of \cite{CD} and the fact that our variational problem takes a minimum over a smaller set (over only regular weighted graphs instead of graphs) and thus must have a larger solution. Since $\delta(K)< 2\Delta_{\star}(K)$ by definition (take an edge adjacent to a vertex of maximum degree) and $p\ll 1$, $\phi_n^d(K,t)\gg n^2p^{2\Delta_{\star}(K)}\log\frac{1}{p}\gg n\log n$ by our bound on $p$.

The only remaining thing that must be verified from the proof in \cite{CD} is that the induction preserves the exceptional set (6.16) being sufficiently small. We have shown that $\mathcal{E}_0$ is sufficiently small. Paralleling the argument up to (6.23), it suffices to show that for all subgraphs $F\subseteq K$ and for all $t>1$, there exists $L>1$ such that
\[\Phi_n^d(F,L)\geq\Phi_n^d(K,t).\]
Take $L>2^{e(K)}t$ and suppose we have a weighted graph $X\in\mathcal{X}_n^d$ with $\Hom(F,X)\geq Kn^{v(F)}p^{e(F)}$.

Note that $\frac{X+p}{2}\in\mathcal{X}_n^d$, where by $p$ we mean the weighted graph with weight $p$ on each edge. But
\[\Hom\left(K,\frac{X+p}{2}\right)\geq\frac{(1-o(1))}{2^{e(K)}}\Hom(F,X)n^{e(K)-e(F)}p^{e(K)-e(F)},\]
as after choosing the positions of the vertices and edges in $F$, there are $(1-o(1))n$ ways to choose the positions of the remaining vertices of $K$ so that they do not coincide with the previous choices, and then each edge not in $F$ has weight at least $\frac{p}{2}$. Therefore, $\Hom\left(K,\frac{X+p}{2}\right)\geq tn^{v(K)}p^{e(K)}$. But $\frac{X+p}{2}$ has at most the same entropy as $X$, so we have proven that $\Phi_n^d(F,L)\geq\Phi_n^d(K,t)$, finishing the proof of Theorem \ref{ldpupperboundthm}.
\section{Introduction to Theorem \ref{ldplowerboundthm}}\label{ldplowerboundsection1}
This section and the next four will together prove Theorem \ref{ldplowerboundthm}.

A \emph{block graphon} is a graphon $W:[0,1]^2\to [0,1]$ such that there exists $k\in\mathbb{Z}^+$ and a partition $[0,1]=S_1\cup\cdots\cup S_k$ into intervals $S_i$ such that for all $i,j\in [k]$, $W$ is constant on $S_i\times S_j$.

For convenience, we now recall the statement of Theorem \ref{ldplowerboundthm}.
\begin{ldplowerboundthm}
Suppose $n^{-1}\log\log n\ll p\ll 1$, and let $W=W(n)$ be a block graphon on some constant $k=k(n)=O(1)$ number of blocks satisfying Conditions \ref{lowerboundconditions}. Then
\[-\log\left(\Pr\left[\Hom(K,G_n^d)\geq (1-o(1))\Hom(K,W)n^{v(K)}\right]\right)\leq \left(\frac{1}{2}+o(1)\right)I_p(W)n^2.\]
\end{ldplowerboundthm}
To state the conditions of Theorem \ref{ldplowerboundthm} efficiently, we will need the following definition.
\begin{definition}
For a block graphon $W$ on the partition $S_1\cup\cdots S_k$, a $K$-block of $W$ is an assignment $B:V(K)\to [k]$, which we may think of as assigning to each vertex $v\in V(K)$ one of the intervals $S_i$, $1\leq i\leq k$.

If $B:V(K)\to [k]$ is a $K$-block of $W$, we define
\[\Hom_B(K,W)=\displaystyle\int_{\substack{x_v\in S_{B(v)} \\ \forall v\in V(K)}}\displaystyle\prod_{vw\in E(K)}W(x_v,x_w)\displaystyle\prod_{v\in V(K)}dx_v.\]

Call a $K$-block $B$ negligible if $\Hom_B(K,W)=o(p^{e(K)})$, and non-negligible otherwise.
\end{definition}
\begin{remark}
Notice that $\Hom_B(K,W)$ can be thought of as the contribution to $\Hom(K,W)$ from homomorphisms where the interval that each $v\in V(K)$ is sent into is determined by $B(v)$. Thus $\displaystyle\sum_{B}\Hom_B(K,W)=\Hom(K,W)$, where the sum ranges over all functions $B:V(K)\to [k]$.
\end{remark}
We are now able to state the conditions of Theorem \ref{ldplowerboundthm}.
\begin{conditions}\label{lowerboundconditions}
The following conditions are for $W$ a block graphon on the partition $[0,1]=S_1\cup\cdots\cup S_k$ into intervals, taking value $w_{ij}$ on $S_i\times S_j$.
\begin{enumerate}
\item (Regularity) $W$ is $p$-regular; that is, $\displaystyle\int_0^1 W(x_0,y)dy=p$ for all $x_0$.
\item (One Block Dominates in Size) $m(S_k)\geq 1-p$.
\item (Many Copies of $K$) $\Hom(K,W)\geq (1+\Omega(1))p^{e(K)}$.
\item (Bounded Entropy) $n^{-1}\log n\ll I_p(W)\ll p^{2e(K)}n$.
\item (Blocks Are Not Too Small) $m(S_i)\gg n^{-1}$ for all $i$.
\item (Dichotomy on Small Blocks) For all $i,j\leq k-1$, $w_{ij}\geq p$ and at least one of the following holds.
\begin{itemize}
\item $w_{ij}=p$
\item $m(S_i)m(S_j)w_{ij}\ll \left(\log\log\log\frac{1}{p}\right)^{-1}I_p(W)$.
\end{itemize}
Call a block $(i,j)$ \emph{important} if the latter case holds and \emph{unimportant} otherwise; that is, if $m(S_i)m(S_j)w_{ij}\not\ll\left(\log\log\log\frac{1}{p}\right)^{-1}I_p(W)$ or one of $i$ or $j$ equal $k$. Further call an important block $(i,j)$ \emph{somewhat important} if $w_{ij}<1$ or \emph{very important} if $w_{ij}=1$.
\item (Somewhat Important Blocks) Recall that a $K$-block $B$ is \emph{negligible} if $\Hom_B(K,W)=o(p^{e(K)})$, and \emph{non-negligible} otherwise. If $B$ is a non-negligible $K$-block, then all edges that $B$ sends into somewhat important blocks are disjoint (i.e. form a matching).
\item (Unimportant Blocks are Large) If $B$ is a non-negligible $K$-block and $(i,j)$ is an unimportant block containing the image of at least one edge under $B$, then $m(S_i)m(S_j)\geq (1+o(1))p^{-1}I_p(W)$ and $w_{ij}=p+o(p)$.
\item (High Degrees within Important Blocks) $w_{ij}m(S_i)\gg n^{-1}\log n$ for all $(i,j)$ important.
\item (Not Too Many Copies of $K$) $\Hom(K,W+p)=O(p^{e(K)}).$
\end{enumerate}
\end{conditions}
\begin{remark}
The upper bound $I_p(W)\ll p^{2e(K)}n$ of condition (4) is the strictest bound that we must deal with, in the sense that our ranges on $p$ in the main theorems come from having to satisfy this bound. Thus it is likely that one could improve our bounds on $p$ by proving a stronger version of Theorem \ref{ldplowerboundthm} that loosens this condition.
\end{remark}
For technical reasons related to the proof of Lemma \ref{swappinglemma} below, we would like to show that we can replace condition (2) with a stronger condition. We do so via the following lemma.
\begin{lemma}\label{stronger2lemma}
Let $W$ satisfy Conditions \ref{lowerboundconditions}. Then there is a block graphon $W'$ on the same number of parts as $W$ satisfying Conditions \ref{lowerboundconditions} such that
\begin{itemize}
\item $\Hom(K,W')\geq (1-o(1))\Hom(K,W)$
\item $I_p(W')\leq I_p(W)$
\item At least one interval of the block partition of $W'$ is of length at least $1-p+p\left(\log\log\frac{1}{p}\right)^{-1}$.
\end{itemize}
\end{lemma}
The proof of Lemma \ref{stronger2lemma} is quite technical and relatively unenlightening, and a reader looking to understand the broad strokes of the proof of Theorem \ref{ldplowerboundthm} may wish to skip it. The idea is that we slightly shrink the intervals $S_i$, $1\leq i\leq k-1$, and go through and check that all the desired conditions hold.
\begin{proof}[Proof of Lemma \ref{stronger2lemma}]
We will obtain $W'$ by shrinking all intervals $S_i$, $1\leq i\leq k-1$ by a factor of $1-\left(\log\log\frac{1}{p}\right)^{-1}$ to form intervals $S_i'$. Then take $S_k'=[0,1]\backslash (S_1'\cup\cdots\cup S_{k-1}')$.

Let $W'$ take value $w'_{ij}$ on $S_i'\times S_j'$, with $w'_{ij}=w_{ij}$ for $1\leq i,j\leq {k-1}$ and $w'_{ik}=w'_{ki}$ is the unique value such that
\[\displaystyle\sum_{i=1}^k w'_{ij}S_i'=p\]
for all $j\in [k]$. Notice that this is equivalent to (1) of Conditions \ref{lowerboundconditions} holding.

We now prove the three bullet points of the lemma, after which we will proceed to showing that $W'$ satisfies Conditions \ref{lowerboundconditions}.
\setcounter{claimcounter}{0}
\begin{claim}
$m(S_k)\geq 1-p+p\left(\log\log\frac{1}{p}\right)^{-1}$
\end{claim}
 Since $m(S_1)+\cdots+m(S_{k-1})\leq p$, $m(S_1')+\cdots+m(S_{k-1}')\leq p-p\left(\log\log\frac{1}{p}\right)^{-1}$, so $m(S_k)\geq 1-p+p\left(\log\log\frac{1}{p}\right)^{-1}$, proving the claim.

\begin{claim}\label{annoyingentropyclaim}
$I_p(W')\leq I_p(W)$
\end{claim}
For $1\leq i\leq k-1$,
\begin{align}
\label{Sk'Sk} m(S_k')(w'_{ik}-p) & =-\displaystyle\sum_{j=1}^{k-1}m(S_j')(w'_{ij}-p) \\ & \nonumber =-\left(1-\left(\log\log\frac{1}{p}\right)^{-1}\right)\displaystyle\sum_{j=1}^{k-1}m(S_j)(w_{ij}-p) \\ & \nonumber=\left(1-\left(\log\log\frac{1}{p}\right)^{-1}\right)m(S_k)(w_{ik}-p)
\end{align}
With $i=k$, we may apply a similar computation using (\ref{Sk'Sk}). In particular,
\begin{align}
\label{Skk'Skk} m(S_k')(w'_{kk}-p) & =-\displaystyle\sum_{j=1}^{k-1}m(S_j')(w'_{kj}-p) \\ & \nonumber =-\left(1-\left(\log\log\frac{1}{p}\right)^{-1}\right)\frac{m(S_k)}{m(S_k')}\displaystyle\sum_{j=1}^{k-1}m(S_j)(w_{kj}-p) \\ & \nonumber=\left(1-\left(\log\log\frac{1}{p}\right)^{-1}\right)\frac{m(S_k)^2}{m(S_k')}(w_{ik}-p) \\ & \nonumber=\left(1-\left(\log\log\frac{1}{p}\right)^{-1}\right)m(S_k)(w_{ik}-p)
\end{align}
since $m(S_k')\geq m(S_k)$.

Since $I_p(p+x)$ is a convex function with $I_p(p)=0$, we have that $a I_p(p+x)\leq I_p(p+ax)$ if $a\geq 1$. For $1\leq i\leq k-1$, letting $a=\frac{m(S_k')}{m(S_k)}$ and $x=w'_{ik}-p$, and applying (\ref{Sk'Sk}),
\[\frac{m(S_k')}{m(S_k)}I_p(w'_{ik})\leq I_p\left(p+\frac{m(S_k')}{m(S_k)}(w'_{ik}-p)\right) \leq I_p\left(p+\left(1-\left(\log\log\frac{1}{p}\right)^{-1}\right)w_{ik}\right) \leq I_p(w_{ik}).\]
If instead $a=\frac{m(S_k')^2}{m(S_k)^2}$ and $x=w'_{kk}-p$ and applying (\ref{Skk'Skk}),
\[\frac{m(S_k')^2}{m(S_k)^2}I_p(w'_{kk})\leq I_p\left(p+\frac{m(S_k')^2}{m(S_k)^2}(w'_{kk}-p)\right)\leq I_p\left(p+\left(1-\left(\log\log\frac{1}{p}\right)^{-1}\right)w_{kk}\right) \leq I_p(w_{kk}).\]
Since $m(S_i')\leq m(S_i)$ for all $1\leq i\leq k-1$, and $w'_{ij}=w_{ij}$ for $1\leq i,j\leq k-1$, the previous two equations yield
\[m(S_i')m(S_j')I_p(w'_{ij})\leq m(S_i)m(S_j)I_p(w_{ij})\]
for all $1\leq i,j\leq k$. Since $I_p(W)=\displaystyle\sum_{1\leq i,j\leq k}m(S_i)m(S_j)I_p(w_{ij})$ and similarly for $W'$, we have shown that $I_p(W')\leq I_p(W)$.
\begin{claim}\label{annoyinghomclaim}
$\Hom(K,W')\geq (1-o(1))\Hom(K,W)$
\end{claim}
By (\ref{Sk'Sk}) and (\ref{Skk'Skk}) and the fact that $m(S_k')\leq m(S_k)$,
\[|w'_{ik}-p|\leq |w_{ik}-p|.\]
Since $m(S_k')=(1+o(1))m(S_k)$, we also have $(w'_{ik}-p)=(1+o(1))(w_{ik}-p)$. Thus $w'_{ik}\geq (1-o(1))w_{ik}$ for all $i$.

Now, we have shown that $m(S_i')\geq (1-o(1))S_i$ for all $i$, and $w'_{ij}\geq (1-o(1))w_{ij}$ for all $i,j$. Since for each $B$, $\Hom_B(K,W)$ is a product of specific $m(S_i)$ and $w_{ij}$, and $\Hom_B(K,W')$ is the product of the corresponding $m(S_i')$ and $w'_{ij}$, $\Hom_B(K,W')\geq (1-o(1))\Hom_B(K,W)$ for all $B$. Summing,
\[\Hom(K,W')\geq (1-o(1))\Hom_B(K,W),\]
proving this claim as well.

Technically, we must show that $w'_{ij}\in [0,1]$ for all $i,j$ in order to verify that $W'$ is indeed a graphon. For $1\leq i,j\leq k$ this is clear, and by (\ref{Sk'Sk}) and (\ref{Skk'Skk}), $w'_{ik}-p$ has the same sign as and a smaller magnitude than $w_{ik}-p$ for all $i$, which proves that $w'_{ik}\in [0,1]$.

Now, we must show that $W'$ satisfies Conditions \ref{lowerboundconditions}. Condition (1) follows easily from the definition of $w'_{ik}$, and we in fact have our stronger version of condition (2) that $m(S_k')\geq 1-p+p\left(\log\log\frac{1}{p}\right)^{-1}$.

Conditions (5) and (9) follow easily from the definition of $W'$, and (3) follows from Claim \ref{annoyinghomclaim}.

Since $w'_{ij}-p=(1+o(1))(w_{ij}-p)$, $w'_{ij}-p=w_{ij}-p+o(|w_{ij}-p|)=w_{ij}-p+o(w_{ij}+p)$. Thus $w'_{ij}+p=(1+o(1))(w_{ij}+p)$. We also know that $m(S_i')=(1+o(1))m(S_i)$. Therefore, $\Hom_B(K,W'+p)=(1+o(1))\Hom_B(K,W+p)$ for all $B$ (each side is simply a product of terms of the form $m(S_i)$ and $w_{ij}+p$ or analogously for $W'$). Summing, $\Hom(K,W'+p)=(1+o(1))\Hom_B(K,W+p)$. This shows Condition (10).

Now, since $w'_{ik}-p=(1+o(1))(w_{ik}-p)$, $I_p(w'_{ik})=(1+o(1))I_p(w_{ik})$ (the $o(1)$ can be taken to be uniform as $I_p(p+x)$ behaves like $\frac{x^2}{2p}$ around $x=0$). Thus for all $i,j$, $m(S_i')m(S_j')I_p(w'_{ij})=(1+o(1))m(S_i)m(S_j)I_p(w_{ij})$, so summing,
\[I_p(W')=(1+o(1))I_p(W).\]
This proves Condition 4. Furthermore, since $w_{ij}$ is preserved for $i,j\leq k-1$ and $m(S_i')=(1+o(1))m(S_i)$ for all $i$, it also proves Condition (6). Now, we just must show conditions (7) and (8) still hold for $W'$. But $m(S_i')=(1+o(1))m(S_i)$ for all $i$ and $w'_{ij}-p=(1+o(1))(w_{ij}-p)$, so the properties we would like to show of non-negligible blocks are preserved when we go from $W$ to $W'$. Thus it suffices to show that $W$ and $W'$ have the exact same set of non-negligible blocks. This is accomplished via the following claim, which will therefore finish the proof of the lemma.

\begin{claim}
$\Hom_B(K,W)=\Hom_B(K,W')+o(p^{e(K)})$ for all $B$.
\end{claim}
By Condition (10), there is some $C$ such that $\Hom(K,W+p),\Hom(K,W'+p)\leq Cp^{e(K)}$. Take any $K$-block $B$ and any $\epsilon>0$.

If there is some $1\leq i,j\leq k$ such that $B$ `uses' the block $(i,j)$ (in the sense that there is some edge $vw\in E(K)$ such that $B(v)=i$ and $B(w)=j$), then if $w_{ij}\leq\epsilon p$, then $\Hom_B(K,W)\leq\frac{w_{ij}}{w_{ij}+p}\Hom_B(K,W+p)\leq\epsilon\Hom_B(K,W+p)\leq\epsilon Cp^{e(K)}$.

Similarly, since $w_{ij}-p=(1+o(1))(w'_{ij}-p)$, $w'_{ij}\leq (\epsilon+o(1))p$, so $\Hom_B(K,W')\leq(\epsilon+o(1))Cp^{e(K)}$ by the same logic. Thus
\[\left|\Hom_B(K,W')-\Hom_B(K,W)\right|\leq (2\epsilon+o(1))Cp^{e(K)}\]
by the triangle inequality.\

If instead $w_{ij}>\epsilon p$ for every $(i,j)$ used by $B$, then $w'_{ij}=p+(1+o(1))(w_{ij}-p)=(1+o(1))w_{ij}+o(p)=(1+o(1))w_{ij}$ for each such $(i,j)$, since $w_{ij}=\Omega(p)$. Since $m(S_i')=(1+o(1))m(S_i)$ for each $i$, this implies that
\[\Hom_B(K,W')=(1+o(1))\Hom_B(K,W).\]
Since $\Hom_B(K,W)\leq\Hom_B(K,W+p)=O(p^{e(K)})$, we thus have
\[\left|\Hom_B(K,W')-\Hom_B(K,W)\right|=o(\Hom_B(K,W))=o(p^{e(K)})\]
in this case. Thus in all cases, $\left|\Hom_B(K,W')-\Hom_B(K,W)\right|\leq (2\epsilon+o(1))Cp^{e(K)}$, and since this holds for any constant $\epsilon>0$ we have the desired conclusion.
\end{proof}
\section{Overview of Proof of Theorem \ref{ldplowerboundthm}}\label{ldplowerboundsection2}
For the following four definitions, we will take $n\to\infty$, $d=d(n)$, $p=p(n):=\frac{d}{n}$. We will further take $W$ to be a block graphon on the partition $[0,1]=S_1\cup\cdots\cup S_k$ and taking value $w_{ij}$ on $S_i\times S_j$, satisfying Conditions \ref{lowerboundconditions} with the improved condition (2) that $m(S_k)\geq 1-p+p\left(\log\log\frac{1}{p}\right)^{-1}$. By Lemma \ref{stronger2lemma}, under these conditions, to prove Theorem \ref{ldplowerboundthm} it will suffice to show that
\begin{equation}\label{HomKGnd}
-\log\left(\Pr\left[\Hom(K,G_n^d)\geq (1-o(1))\Hom(K,W)n^{v(K)}\right]\right)\leq \left(\frac{1}{2}+o(1)\right)I_p(W)n^2.
\end{equation}
We will prove this statement over this and the next two sections. We first make several definitions.
\begin{definition}[Blocks]\label{blocksdef}
We call a pair $(i,j)$, $1\leq i,j\leq k$ a \emph{block} of $W$, and think of it as referring to the rectangle $S_i\times S_j$. Recall from Conditions \ref{lowerboundconditions} that we call a block $(i,j)$ \emph{important} if $i,j\leq k-1$ and $m(S_i)m(S_j)w_{ij}\ll \left(\log\log\log\frac{1}{p}\right)^{-1}I_p(W)$.

Define the sets $V_i$, $1\leq i\leq k$, by $V_i:=\mathbb{Z}^+\cap nS_i$; that is, the interval of $[n]$ corresponding to $S_i$.

(We will often refer to blocks $S_i\times S_j$ or $V_i\times V_j$ as important; this simply refers to the underlying pair $(i,j)$.)
\end{definition}
We now consider several subsets of the set of graphs on $[n]$, which we will use as events under various probability distributions on that set.
\begin{definition}[Events]\label{eventsdef}
Let $\mathcal{K}_{n,d}$ be the set of $d$-regular graphs on $[n]$.

For any constant $\delta'>0$, define
\[\mathcal{H}_{\delta'}=\left\{T\in\mathcal{K}_{n,d}:\Hom(K,T)\geq (1+\delta')p^{e(K)}n^{v(K)}\right\}.\]
\end{definition}
We now define the probability distributions that we will use. Recall from Conditions \ref{lowerboundconditions} that a block $(i,j)$ (thought of as specifying a block $S_i\times S_j$ of our graphon $W$) is called \emph{important} if $1\leq i,j\leq k-1$ and $m(S_i)m(S_j)w_{ij}\ll \left(\log\log\log\frac{1}{p}\right)^{-1}I_p(W)$.
\begin{definition}[Probability Distributions]\label{distributionsdef}
Let $\mathbb{P}_{G_n^d}$ be the probability measure of a random $d$-regular graph on $[n]$, and let $\mathbb{P}_p$ be the probability measure of the Erd\H{o}s-R\'{e}nyi random graph on $[n]$ with edge probability $p$.

Furthermore, let $\mathbb{P}_{\star}$ be the inhomogeneous Erd\H{o}s-R\'{e}nyi model where we sample edges with probability corresponding to $W$ in the important blocks of $W$, and probability $p$ otherwise. In particular, we sample edge $uv$ with probability $w_{ij}$ if $(u,v)$ is contained in some important $V_i\times V_j$, and probability $p$ otherwise.
\end{definition}

By Condition (3) of Conditions \ref{lowerboundconditions}, we may take $\delta>0$ such that $(1+\delta)p^{e(K)}=\Hom(K,W)$.

Using Definitions \ref{eventsdef} and \ref{distributionsdef}, we may rephrase (\ref{HomKGnd}) as stating that
\begin{equation}\label{pgndhomequation}
\mathbb{P}_{G_n^d}(\mathcal{H}_{\delta'})\geq\exp\left(-\left(\frac{1}{2}+o(1)\right)I_p(W)n^2\right)
\end{equation}
for all $0<\delta'<\delta$.

We largely parallel the proof in Section 2.3 of \cite{BD}, although substantial modification will be necessary. A main idea is to note that for any event $\mathcal{B}$,
\begin{align*}
\mathbb{P}_{G_n^d}(\mathcal{H}_{\delta'}) & \geq\mathbb{P}_{G_n^d}(\mathcal{B}\land\mathcal{H}_{\delta'}) \\ & =\mathbb{P}_{G_n^d}(\mathcal{B})-\mathbb{P}_{G_n^d}(\mathcal{B}\land\neg\mathcal{H}_{\delta'}).
\end{align*}
\begin{remark}
We would like to choose $\mathcal{B}$ such that both of these terms may be bounded in the appropriate direction. The idea is to make the event $\mathcal{B}$ (informally) the event that a graph `looks like' the graphon $W$. The point is that given the event $\mathcal{B}$, we should expect $\Hom(K,G_n^d)$ to be large, so $\neg\mathcal{H}_{\delta'}$ should be a lower-tail-type event whose probability can be bounded above using Janson's inequality. There is a tension between choosing $\mathcal{B}$ simple enough so that $\mathbb{P}_{G_n^d}(\mathcal{B})$ can be approximately computed easily and choosing it specific enough to close the `loopholes' by which $B\land\neg\mathcal{H}_{\delta'}$ might easily occur. These loopholes do not occur in \cite{BD}, so our choice of event $\mathcal{B}$ will be correspondingly more complicated.
\end{remark}
We now specify the event we will be using as the event $\mathcal{B}$ above.
\begin{definition}\label{andegdef}
For $i\neq j$, $1\leq i,j\leq k$, define $a_{ij}:=\left\lfloor w_{ij}|V_i||V_j|+\frac{1}{2}\right\rfloor$, and for $1\leq i\leq k$, define $a_{ii}:=2\left\lfloor w_{ii}\binom{|V_i|}{2}+\frac{1}{2}\right\rfloor$.

Let $\mathcal{A}_n$ be the set of graphs $T\in\mathcal{K}_{n,d}$ such that for all $1\leq i,j\leq k-1$ with $(i,j)$ important, there are exactly $a_{ij}$ pairs $(v_i,v_j)\in V_i\times V_j$ such that $uv$ is an edge of $T$.

Let $\mathcal{A}_n^{\deg}\subseteq\mathcal{A}_n$ be the set of graphs $T\subseteq\mathcal{A}_n$ such that for all $(i,j)$ important, then $v_i\in V_i$ has $|N_{T}(v_i)\cap V_j|\leq 2\frac{a_{ij}}{|V_i|}$.
\end{definition}
\begin{remark}
Note that $a_{ij}$ rounds $w_{ij}|V_i||V_j|$ to the nearest integer, or nearest even integer if $i=j$ (this is necessary because if $i=j$ each relevant edge is counted twice). Thus $\mathcal{A}_n$ may be thought of as stipulating that a graph has as close as possible to the appropriate density $w_{ij}$ in each important block $V_i\times V_j$.

The subset $\mathcal{A}_n^{\deg}$ can be thought of as additionally specifying that in each important block $V_i\times V_j$, the degrees are not too much more than expected.
\end{remark}
We will end up taking our event $\mathcal{B}$ to be $\mathcal{A}_n^{\deg}$. In summary, we have the following setup, which we will be using for the next three sections (and thus contains some further definitions that we will use in future sections).
\begin{setup}\label{ldpsetup}
Let $n\to\infty$ and take $d=d(n)\in\mathbb{Z}^+$, $p=\frac{d}{n}$ with $n^{-1}\log\log n\ll p\ll 1$. Let $k=k(n)=O(1)$ be a constant and let $W=W(n)$ be a block graphon on $k$ intervals $S_1,\ldots,S_k$ forming a partition of $[0,1]$, such that $W=w_{ij}$ on $S_i\times S_j$. (Note that $S_i$, $w_{ij}$ depend on $n$.)

Further suppose that $W$ satisfies Conditions \ref{lowerboundconditions} (with all asymptotics taken as $n\to\infty$), and satisfies a stronger version of condition (2) stating that $m(S_k)\geq 1-p+p\left(\log\log\frac{1}{p}\right)^{-1}$. Take $\delta=\delta(n)>0$ such that $\Hom(K,W)=(1+\delta)p^{e(K)}$.

Let $V_i=\mathbb{Z}^+\cap nS_i$ for all $i\in [k]$. Recall that a block $(i,j)$ is \emph{important} if $1\leq i,j\leq k-1$ and $m(S_i)m(S_j)w_{ij}\ll \left(\log\log\log\frac{1}{p}\right)^{-1}I_p(W)$ and define $a_{ij}$ for $1\leq i,j\leq k$, $\mathcal{A}_n$ and $\mathcal{A}_n^{\deg}$ as in Definition \ref{andegdef}. Further define $\mathcal{H}_{\delta'}$ as in Definition \ref{eventsdef} for any $\delta>0$.

Let $Imp_n=Imp_n(W)\subset E(K_n)$ be the set of edges that are contained in important blocks $V_i\times V_j$. For any subset $S\subseteq Imp_n$, let $\mathcal{B}^S=\mathcal{B}^S(n,W)$ be the set of all graphs $T$ on $[n]$ such that $E(T)\cup Imp_n=S$.
\end{setup}
We have effectively reduced Theorem \ref{ldplowerboundthm} to proving the following two propositions.
\begin{prop}\label{firsttermprop}
Under Setup \ref{ldpsetup},
\[\mathbb{P}_{G_n^d}(\mathcal{A}_n^{\deg})\geq\exp\left(-\left(\frac{1}{2}+o(1)\right)I_p(W)n^2\right).\]
\end{prop}
\begin{prop}\label{secondtermprop}
Under Setup \ref{ldpsetup}, for any $\delta'\leq\delta-\Omega(1)$,
\[\mathbb{P}_{G_n^d}(\mathcal{A}_n^{\deg}\land\neg\mathcal{H}_{\delta'})\leq\exp\left(-\left(\frac{1}{2}+\Omega(1)\right)I_p(W)n^2\right),\]
recalling the definition of $\mathcal{H}_{\delta'}$ from Definition \ref{eventsdef}.
\end{prop}
\begin{proof}[Proof of Theorem \ref{ldplowerboundthm} given Propositions \ref{firsttermprop} and \ref{secondtermprop}]
Take $p=p(n)$ with $n^{-1}\log\log n\ll p\ll 1$, and take $W=W(n)$ a block graphon on $k=k(n)=O(1)$ parts satisfying Conditions \ref{lowerboundconditions}.

By Lemma \ref{stronger2lemma}, there is some $W'=W'(n)$ a block graphon on $k$ parts satisfying Conditions \ref{lowerboundconditions} with $\Hom(K,W')\geq (1-o(1))\Hom(K,W)$, $I_p(W')\leq I_p(W)$. Furthermore, $W'$ has at least one interval in its partition of length at least $1-p+p\left(\log\log\frac{1}{p}\right)^{-1}$.

Since $\Hom(K,W)=(1+\Omega(1))p^{e(K)}$ by (3) of Conditions \ref{lowerboundconditions}, $\Hom(K,W')\geq (1-o(1))\Hom(K,W)=(1+\Omega(1))p^{e(K)}$. Thus we may take $\delta=\delta(n)>0$ such that $\Hom(K,W')=(1+\delta)p^{e(K)}$.

We are now in the situation of Setup \ref{ldpsetup}, with $W'$ taking the place of the graphon $W$. Take any constant $\epsilon>0$ independent of $n$ and set $\delta'=\delta-\epsilon$. By Propositions \ref{firsttermprop} and \ref{secondtermprop},
\begin{align*}
\mathbb{P}_{G_n^d}(\mathcal{H}_{\delta'}) & \geq\mathbb{P}_{G_n^d}(\mathcal{A}_n^{\deg}\land\mathcal{H}_{\delta'}) \\ & =\mathbb{P}_{G_n^d}(\mathcal{A}_n^{\deg})-\mathbb{P}_{G_n^d}(\mathcal{A}_n^{\deg}\land\neg\mathcal{H}_{\delta'}) \\ & \geq\exp\left(-\left(\frac{1}{2}+o(1)\right)I_p(W)n^2\right).
\end{align*}
Unrolling the definition of $\mathbb{P}_{G_n^d}(\mathcal{H}_{\delta'})$, we see that it is equal to
\[\Pr[\Hom(K,G_n^d)\geq (1+\delta')p^{e(K)}n^{v(K)}],\]
and since $1+\delta'=1+\delta-\epsilon=\frac{1+\delta-\epsilon}{1+\delta}p^{-e(K)}\Hom(K,W')\geq (1-o(1))\frac{1+\delta-\epsilon}{1+\delta}p^{-e(K)}\Hom(K,W)$, we see that
\begin{align*}
\Pr\left[\Hom(K,G_n^d)\geq (1-o(1))\frac{1+\delta-\epsilon}{1+\delta}\Hom(K,W)n^{v(K)}\right] & \geq\exp\left(-\left(\frac{1}{2}+o(1)\right)I_p(W')n^2\right) \\ & \geq\exp\left(-\left(\frac{1}{2}+o(1)\right)I_p(W)n^2\right),
\end{align*}
as $I_p(W')\leq I_p(W)$.

Since this holds for any $\epsilon>0$, we may take $\epsilon\to 0$ sufficiently slowly that this lower bound still holds, so
\[\Pr\left[\Hom(K,G_n^d)\geq (1-o(1))\Hom(K,W)n^{v(K)}\right]\geq\exp\left(-\left(\frac{1}{2}+o(1)\right)I_p(W)n^2\right).\]
Taking logarithms,
\[-\log\left(\Pr\left[\Hom(K,G_n^d)\geq (1-o(1))\Hom(K,W)n^{v(K)}\right]\right)\leq\left(\frac{1}{2}+o(1)\right)I_p(W)n^2,\]
as desired.
\end{proof}
We will prove these Propositions \ref{firsttermprop} and \ref{secondtermprop} in the following two sections. We finish this section with several useful observations.
\begin{lemma}\label{conditionsdeductionlemma}
Under Setup \ref{ldpsetup}, the following statements hold.
\begin{itemize}
\item $|V_i|=m(S_i)n+O(1)=(1+o(1))m(S_i)n\gg 1$ for all $i\in [k]$.
\item $a_{ij}=w_{ij}|V_i||V_j|+O(1)$ for all $i,j\in [k]$.
\item $a_{ij}=(1+o(1))w_{ij}|V_i||V_j|$ if $(i,j)$ is important.
\end{itemize}
\end{lemma}
\begin{proof}
To prove the first statement, note that since $nS_i$ is an interval of length $m(S_i)n$, it contains between $m(S_i)n-1$ and $m(S_i)n+1$ positive integers. Since $V_i=nS_i\cap\mathbb{Z}^+$, we must have $|V_i|=m(S_i)n+O(1)$. By (5) of Conditions \ref{lowerboundconditions}, $m(S_i)n\gg 1$. Thus $|V_i|=(1+o(1))m(S_i)n\gg 1$.

To prove the second statement, note that $a_{ij}$ is $w_{ij}|V_i||V_j|$ rounded to the nearest (or nearest even if $i=j$) integer, so it differs from $w_{ij}|V_i||V_j|$ by $O(1)$.

To prove the third statement, given the second statement it suffices to show that $w_{ij}|V_i||V_j|\gg 1$ when $(i,j)$ is important. Applying the first statement along with (9) of Conditions \ref{lowerboundconditions} yields
\[w_{ij}|V_i||V_j|=(1+o(1))(w_{ij}m(S_i)n)|V_j|\gg (1+o(1))|V_j|\log n\gg 1,\]
exactly what we wanted to show.
\end{proof}
The following lemma shows that important blocks dominate the entropy.
\begin{lemma}\label{importantblocklemma}
Under Setup \ref{ldpsetup},
\[I_p(W)=(1+o(1))n^{-2}\displaystyle\sum_{\substack{1\leq i,j\leq k \\ (i,j)\text{ important}}}a_{ij}\log\frac{w_{ij}}{p}.\]
\end{lemma}
\begin{proof}
By (1) of Conditions \ref{lowerboundconditions}, $W$ is $p$-regular, so
\[\displaystyle\int_{S_k}(W(x_0,y)-p)dy=-\displaystyle\int_{[0,1]\backslash S_k}(W(x_0,y)-p)dy\]
for all $x_0$. Take $i$ such that $x_0\in S_i$. Now, $W(x_0,y)-p=w_{ij}-p$ when $y\in S_j$, so we may write the left side as $m(S_k)(w_{ik}-p)$.

Now, let $f(x)=I_p(p+x)$. The function $f$ is convex, as $I_p$ is, so by Jensen's inequality,
\begin{align*}
\displaystyle\int_{[0,1]\backslash S_k}I_p(W(x_0,y))dy & =\displaystyle\int_{[0,1]\backslash S_k}f(W(x_0,y)-p)dy \\ & \geq (1-m(S_k))f\left(\frac{\displaystyle\int_{[0,1]\backslash S_k}(W(x_0,y)-p)dy}{(1-m(S_k))}\right) \\ & =(1-m(S_k))f\left(\frac{-\displaystyle\int_{S_k}(W(x_0,y)-p)dy}{(1-m(S_k))}\right) \\ & =(1-m(S_k))f\left(\frac{-m(S_k)}{(1-m(S_k))}(w_{ik}-p)\right).
\end{align*}
More simply, we also have
\[\displaystyle\int_{S_k}I_p(W(x_0,y))dy=\displaystyle\int_{S_k}f(w_{ik}-p)dy=m(S_k)f(w_{ik}-p).\]
Thus letting $q=\frac{m(S_k)}{1-m(S_k)}$, we see that
\[\frac{\displaystyle\int_{S_k}I_p(W(x_0,y))dy}{\displaystyle\int_{[0,1]\backslash S_k}I_p(W(x_0,y))dy}\leq \frac{q\cdot f(w_{ik}-p)}{f(q\cdot (w_{ik}-p))}.\]
Notice that $q\geq\frac{1-p}{p}\geq\frac{1}{2p}$ by condition (2) of Conditions \ref{lowerboundconditions}. Since $f(x)=\Theta\left(x\log\frac{x}{p}\right)$ for $p\ll |x|\leq 1$ and $f(x)=\Theta\left(\frac{x^2}{p}\right)$ for $|x|=O(p)$, it is easy to check that for $q=\Omega(p^{-1})$ and all $x$ in $[-p,1-p]$, $f(x)\ll \frac{f(qx)}{q}$.

Therefore, we see that
\[\displaystyle\int_{[0,1]\backslash S_k}I_p(W(x_0,y))dy\gg\displaystyle\int_{S_k}I_p(W(x_0,y))dy\]
for all $x_0$, so we must have
\begin{equation}\label{eliminatesk}
\displaystyle\int_{[0,1]^2}I_p(W(x,y))dx dy=(1+o(1))\displaystyle\int_{[0,1]\times [0,1]\backslash S_k}I_p(W(x,y))dx dy=(1+o(1))\displaystyle\int_{([0,1]\backslash S_k)^2}I_p(W(x,y))dx dy.
\end{equation}
Finally, by condition (6) of Conditions \ref{lowerboundconditions}, if some block $S_i\times S_j\subset ([0,1]\backslash S_k)^2$ is not important, then $w_{ij}=p$ and so the integral of $I_p(W)$ over that block vanishes. Thus in the right hand side of (\ref{eliminatesk}) we may simply integrate over important blocks, yielding
\begin{align*}
I_p(W) & =(1+o(1))\displaystyle\sum_{\substack{1\leq i,j\leq k \\ (i,j)\text{ important}}}\displaystyle\int_{S_i\times S_j}I_p(W(x,y))dx dy \\ & =(1+o(1))\displaystyle\sum_{\substack{1\leq i,j\leq k \\ (i,j)\text{ important}}}m(S_i)m(S_j) I_p(w_{ij}).
\end{align*}
Multiplying by $n^2$, and applying Lemma \ref{conditionsdeductionlemma}, we obtain
\begin{align*}
n^2I_p(W) & =(1+o(1))\displaystyle\sum_{\substack{1\leq i,j\leq k \\ (i,j)\text{ important}}}(m(S_i)n)(m(S_j)n)I_p(w_{ij}) \\ & =(1+o(1))\displaystyle\sum_{\substack{1\leq i,j\leq k \\ (i,j)\text{ important}}}|V_i||V_j|I_p(w_{ij}).
\end{align*}
Now, for important $(i,j)$, $w_{ij}\gg p$, so $I_p(w_{ij})=(1+o(1))w_{ij}\log\frac{w_{ij}}{p}$ by Lemma \ref{entropyapproxlemma}. Applying Lemma \ref{conditionsdeductionlemma}, $a_{ij}=(1+o(1))|V_i||V_j|w_{ij}$ for $(i,j)$ important, and thus
\[n^2I_p(W)=(1+o(1))\displaystyle\sum_{\substack{1\leq i,j\leq k \\ (i,j)\text{ important}}}a_{ij}\log\frac{w_{ij}}{p},\]
as desired.
\end{proof}
\section{Proof of Proposition \ref{firsttermprop}}\label{firsttermsection}
We will prove three lemmas that together easily imply the Proposition.
\begin{lemma}\label{Gndvsplemma}
Under Setup \ref{ldpsetup},
\[\mathbb{P}_{G_n^d}(\mathcal{A}_n^{\deg})\geq\exp\left(-o(I_p(W)n^2)\right)\mathbb{P}_{p}(\mathcal{A}_n^{\deg}).\]
\end{lemma}
\begin{lemma}\label{pvsstarlemma}
Under setup \ref{ldpsetup}, the Radon-Nikodym derivative $\frac{d\mathbb{P_*}}{d\mathbb{P}_p}$ is constant on $\mathcal{A}_n$ and equals
\[\exp\left(-\left(\frac{1}{2}+o(1)\right)I_p(W)n^2\right).\]
\end{lemma}
\begin{lemma}\label{starlargelemma}
Under setup \ref{ldpsetup},
\[\mathbb{P}_{\star}(\mathcal{A}_n^{\deg})\geq\exp\left(-o(I_p(W)n^2)\right).\]
\end{lemma}
\begin{proof}[Deduction of Proposition \ref{firsttermprop} from Lemmas \ref{Gndvsplemma}, \ref{pvsstarlemma}, and \ref{starlargelemma}]
Since $\mathcal{A}_n^{\deg}\subseteq\mathcal{A}_n$, Lemma \ref{pvsstarlemma} implies that
\[\mathbb{P}_p(\mathcal{A}_n^{\deg})=\exp\left(-\left(\frac{1}{2}+o(1)\right)I_p(W)n^2\right)\mathbb{P}_{\star}(\mathcal{A}_n^{\deg}).\]
Thus by Lemmas \ref{Gndvsplemma} and \ref{starlargelemma},
\begin{align*}
\mathbb{P}_{G_n^d}(\mathcal{A}_n^{\deg}) & \geq\exp\left(-o(I_p(W)n^2)\right)\mathbb{P}_{p}(\mathcal{A}_n^{\deg}) \\ & =\exp\left(-\left(\frac{1}{2}+o(1)\right)I_p(W)n^2\right)\mathbb{P}_{\star}(\mathcal{A}_n^{\deg}) \\ & \geq\exp\left(-\left(\frac{1}{2}+o(1)\right)I_p(W)n^2\right).
\end{align*}
\end{proof}
We prove the three lemmas over the next three subsections.
\subsection{Proof of Lemma \ref{Gndvsplemma}}
This bound is the most involved of the three. The main idea is the swapping argument used in Lemma 2.5 of \cite{BD}.

Recall from Setup \ref{ldpsetup} the definition of $Imp_n$ and $\mathcal{B}_S$. Notice that the $\mathcal{B}_S$ partition the set of all graphs on $[n]$ as $S$ ranges over all subsets of $Imp_n$. Furthermore, since $\mathcal{A}_n^{\deg}$ only imposes restrictions on edges in important blocks, it can be expressed as the disjoint union of certain $\mathcal{B}_S$.

We would therefore like to show that $\mathbb{P}_{G_n^d}(\mathcal{B}^S)$ is (almost) at least as probable as $\mathbb{P}_p(\mathcal{B}^S)$ for all $S$. We accomplish this by showing that adding a single edge to $S$ multiplies $\mathbb{P}_{G_n^d}(\mathcal{B}^S)$ by approximately $p$.
\begin{lemma}\label{swappinglemma}
Assume Setup \ref{ldpsetup}. Take some $S\subseteq Imp_n$, and let $e\in Imp_n\backslash S$ be an edge. Then
\[(1-o(1))\left(\log\log\frac{1}{p}\right)^{-2}p\leq\frac{\mathbb{P}_{G_n^d}[\mathcal{B}^{S\cup e}]}{\mathbb{P}_{G_n^d}[\mathcal{B}^S]}\leq(1+o(1))p.\]
\end{lemma}
\begin{remark}
It will be vital to the proof of Lemma \ref{swappinglemma} that $Imp_n$ spans few vertices; namely, it only spans those in $V_1\cup\cdots\cup V_{k-1}$. Since $m(S_k)\geq 1-p+p\left(\log\log\frac{1}{p}\right)^{-1}$, we expect that $|V_1\cup\cdots\cup V_{k-1}$ should be approximately at most $d\left(1-\log\log\frac{1}{p}\right)^{-1}$. That this set is at most of size $d$ is clearly important, as otherwise we might be able to choose $S$ to force some vertex of degree greater than $d$. However, the extra $1-\left(\log\log\frac{1}{p}\right)^{-1}$ factor that we gained by proving Lemma \ref{stronger2lemma} will also be important in order for us to obtain good bounds. This is the primary location in the proof that we use this stronger bound on $m(S_k)$.
\end{remark}
\begin{proof}[Proof of Lemma \ref{swappinglemma}]
Paralleling the proof in Lemma 2.5 of \cite{BD}, let $\mathcal{C}_1$ and $\mathcal{C}_0$ be the collections of $d$-regular graphs on $[n]$ that satisfy $\mathcal{B}^{S\cup e}$ and $\mathcal{B}^S$, respectively. We would like to bound $\frac{|\mathcal{C}_1|}{|\mathcal{C}_0|}$. Let $e=uw$, $u,w\in [n]$.

Suppose $G_1\in\mathcal{C}_1$. Any set of four vertices $u_1,w_1,u_2,w_2$ such that $u_1,w_1,u_2,w_2\in V_k$ and $u_1w_1,u_2w_2\in E(G_1)$ but $u_1w,u_2w_1,uw_2\notin E(G_1)$ define a `forward switching' wherein the edges $uw,u_1w_1,u_2w_2$ are replaced with the edges $u_1w,u_2w_1,uw_2$ to yield a graph. This resulting graph is in $\mathcal{C}_0$, because as blocks involving $S_k$ cannot be important, the only change to the important blocks is that the edge $e=uw$ is removed.

Similarly, in the other direction, given $G_0\in\mathcal{C}_0$, any four vertices $u_1,w_1,u_2,w_2$ such that $u\in V_i$, $w\in V_j$, $u_1,w_1,u_2,w_2\in V_k$ and $u_1w,u_2w_1,uw_2\in E(G_0)$ but $u_1w_1,u_2w_2\notin E(G_0)$ define a `reverse switching' wherein the edges $u_1w,u_2w_1,uw_2$ are replaced with the edges $uw,u_1w_1,u_2w_2$ to yield a graph in $\mathcal{C}_1$.

It is clear that forward switching and reverse switching are inverses of each other, so by counting the total number of switchings in two ways, we see that
\begin{equation}\label{swappingbound}\frac{\displaystyle\min_{G_0\in\mathcal{C}_0}(\#\text{ reverse switchings from }G_0)}{\displaystyle\max_{G_1\in\mathcal{C}_1}(\#\text{ forward switchings from }G_1)}\leq\frac{|\mathcal{C}_1|}{|\mathcal{C}_0|}\leq\frac{\displaystyle\max_{G_0\in\mathcal{C}_0}(\#\text{ reverse switchings from }G_0)}{\displaystyle\min_{G_1\in\mathcal{C}_1}(\#\text{ forward switchings from }G_1)}.\end{equation}
Bounding the maximums is quite easy. The number of forward switchings from a particular $G_1$ is at most $d^2n^2$, as there are at most $dn$ ways to choose each of the edges $u_1w_1$ and $u_2w_2$. Similarly, the number of reverse switchings from a particular $G_0$ is at most $d^3n$. This is because there are at most $d$ choices of $u_1\in V\backslash S$ adjacent to $w$, and similarly for $w_2$, and at most $dn$ choices for the edge $u_2w_1$.

We now bound the minimums. Consider reverse switchings from some $G_0\in\mathcal{C}_0$. Since $u_1$ is only restricted to be an element of $S_k$ adjacent to $w$ in $G_0$, the number of potential possibilities for it is at least $d-|V_1|-\cdots-|V_{k-1}|=|V_k|+d-n$, but $|V_k|\geq n\cdot m(S_k)-1\geq n-d+d\left(\log\log\frac{1}{p}\right)^{-1}-1$, so there are at least $d\left(\log\log\frac{1}{p}\right)^{-1}-1$ choices. Similarly, there are at least $d\left(\log\log\frac{1}{p}\right)^{-1}-2$ possibilities for $w_2$ (we must also avoid having $w_2=u_1$).

By the conditions of Setup \ref{ldpsetup}, $p\gg n^{-1}\log\log n$, so $d=pn\gg\log\log n\geq\log\log\frac{1}{p}$, so the expressions in the last paragraph are both $(1+o(1))\frac{d}{\log\log\frac{1}{p}}$. Given $u_1$ and $w_2$, we may choose $u_2w_1$ to be any edge contained in $V_k\backslash\{u_1,w_2\}$. Since there are at most $2(n-|V_k|+2)d\leq 2d^2$ choices that do not satisfy this, we get at least $d(n-2d)$ possibilities for $u_2$ and $w_1$. The only remaining restriction is that neither $u_1w_1$ nor $u_2w_2$ may be edges in $G_0$. But there are only at most $d^4$ possilibities when $u_1w_1$ is also an edge of $G_0$, since we have at most $d$ choices for $u_1$ (as it is adjacent to $w$), and then at most $d$ choices for $w_1$ (adjacent to $u_1$), at most $d$ choices for $u_2$ (adjacent to $w_1$), and at most $d$ choices for $w_2$ (adjacent to $u$). The same analysis occurs when $u_2w_2$ is an edge of $G_0$, so we obtain at least
\[\frac{d^3(n-2d)}{\left(\log\log\frac{1}{p}\right)^2}-2d^4=(1+o(1))d^3n\left(\log\log\frac{1}{p}\right)^{-2}.\]
reverse switchings, as $\left(\log\log\frac{1}{p}\right)^2\ll \frac{1}{p}=\frac{n}{d}$.

The case of minimizing forward switchings from $G_1$ is similar. First we choose $u_1w_1$ to be some edge in $S_k\times S_k$. There are $dn$ edges in the graph (double-counting each edge because we can switch $u_1$ and $w_1$), of which at most $2d(n-|V_k|)\leq 2d^2$ are not contained in $V_k\times V_k$, so again we obtain at least $d(n-2d)$ choices for $u_1$ and $w_1$, and then similarly at least $d(n-2d)$ ways to choose $u_2w_2$ contained in $S_k\backslash\{u_1,w_1\}$. Now, we only must check that $u_1w$, $u_2w_1$, and $uw_2$ are not edges. There are at most $d^3n$ possibilities where $u_1w\in E(G_1)$, as there are at most $d$ choices for each of $u_1$ and $w_1$ and at most $dn$ choices for the edge $u_2w_2$. The same holds for $uw_2$. If $u_2w_1\in G_1$, then $u_1u_2w_1w_2$ is a path of length $3$, so there are again at most $d^3n$ choices by a similar argument. Thus the minimum number of forward switchings from $G_1$ is at least
\[d^2(n-2d)^2-3d^3n=(1+o(1))d^2n^2.\]
Substituting into (\ref{swappingbound}), we obtain that
\[(1+o(1))\frac{d}{\left(\log\log\frac{1}{p}\right)^2n}\leq\frac{|\mathcal{C}_1|}{|\mathcal{C}_0|}\leq (1+o(1))\frac{d}{n},\]
and noting that $p=\frac{d}{n}$ we yield the statement of the Lemma.
\end{proof}
We will also need a bound on $\mathbb{P}_{G_n^d}(\mathcal{B}^{\emptyset})$, after which we will be able to add edges one at a time using Lemma \ref{swappinglemma}.
\begin{lemma}\label{emptysetlemma}
Under Setup \ref{ldpsetup},
\[\mathbb{P}_{G_n^d}(\mathcal{B}^{\emptyset})\geq (1-p-o(p))^{|Imp_n|}.\]
\end{lemma}
\begin{proof}
Adding one edge at a time to the empty set and applying the upper bound of Lemma \ref{swappinglemma} repeatedly yields that
\[\mathbb{P}_{G_n^d}(\mathcal{B}^S)\leq (p+o(p))^{|S|}\mathbb{P}_{G_n^d}(\mathcal{B}^{\emptyset}).\]
The idea is to use this bound while summing over all $S$. We have that
\begin{align*}
1 & =\displaystyle\sum_{S\subseteq Imp_n}Pr[G_n^d\cap Imp_n=S] \\ & =\displaystyle\sum_{S\subseteq Imp_n}\mathbb{P}_{G_n^d}(\mathcal{B}^S) \\ & \leq\displaystyle\sum_{S\subseteq Imp_n}(p+o(p))^{|S|}\mathbb{P}_{G_n^d}(\mathcal{B}^{\emptyset}) \\ & =\displaystyle\sum_{i=0}^{|Imp_n|}\binom{|Imp_n|}{i}(p+o(p))^i\mathbb{P}_{G_n^d}(\mathcal{B}^{\emptyset}) \\ & =(1+p+o(p))^{|Imp_n|}\mathbb{P}_{G_n^d}(\mathcal{B}^{\emptyset}).
\end{align*}
Since $\frac{1}{1+p+o(p)}=1-p+o(p)$, the Lemma follows.
\end{proof}
These two lemmas together imply a lower bound on $\mathbb{P}_{G_n^d}(\mathcal{B}^S)$ for all $S$.
\begin{cor}\label{Bsprobcor}
Under Setup \ref{ldpsetup}, for any $S\subseteq Imp_n$,
\[\mathbb{P}_{G_n^d}(\mathcal{B}^S)\geq \exp\left(-O\left(|S|\log\log\log\frac{1}{p}+p|Imp_n|\right)\right)\mathbb{P}_{p}(\mathcal{B}^S).\]
\end{cor}
\begin{proof}
Adding one edge at a time to the empty set and applying Lemmas \ref{swappinglemma} and \ref{emptysetlemma}, we see that for all $S\subseteq Imp_n$,
\[\mathbb{P}_{G_n^d}(\mathcal{B}^{S})\geq \left((1-o(1))\left(\log\log\frac{1}{p}\right)^{-2}p\right)^{|S|}(1-p-o(p))^{|Imp_n|}.\]
We may also easily compute that
\[\mathbb{P}_{p}(\mathcal{B}^{S})=p^{|S|}(1-p)^{|Imp_n|-|S|},\]
so
\begin{align*}
\frac{\mathbb{P}_{G_n^d}(\mathcal{B}^{S})}{\mathbb{P}_{p}(\mathcal{B}^{S})} & \geq\left((1-o(1))\left(\log\log\frac{1}{p}\right)^{-2}\right)^{|S|}(1-o(p))^{|Imp_n|} \\ & =\exp\left(-\left(2|S|\log\log\log\frac{1}{p}+o(|S|)+o(p|Imp_n|)\right)\right),
\end{align*}
so we have proven the Corollary.
\end{proof}
\begin{proof}[Proof of Lemma \ref{Gndvsplemma}]
As mentioned, $\mathcal{A}_n^{\deg}$ is a union of $\mathcal{B}^S$ over the $S$ that satisfy the conditions of $\mathcal{A}_n^{\deg}$; in particular, for all important $(i,j)$, all such $S$ must have $a_{ij}$ elements in block $V_i\times V_j$ if $i\neq j$ and $\frac{a_{ij}}{2}$ if $i=j$ (if $i=j$ each edge is counted twice). Therefore, for all such $S$, $|S|=\displaystyle\sum_{(i,j)\text{ important}}\frac{a_{ij}}{2}$. By Corollary \ref{Bsprobcor}, for all such $S$ we have that
\[\mathbb{P}_{G_n^d}(\mathcal{B}^S)\geq \exp\left(-O\left(\log\log\log\frac{1}{p}\displaystyle\sum_{(i,j)\text{ important}}\frac{a_{ij}}{2}+p|Imp_n|\right)\right)\mathbb{P}_{p}(\mathcal{B}^S),\]
so to prove Lemma \ref{Gndvsplemma} it suffices to show that
\begin{equation}\label{logloglog1}
\log\log\log\frac{1}{p}\displaystyle\sum_{(i,j)\text{ important}}a_{ij}+p|Imp_n|=o(I_p(W)n^2).
\end{equation}
Now, we compute $|Imp_n|$. The contribution from the block $(i,j)$ is $|V_i||V_j|$ if $i=j$ and $\binom{|V_i|}{2}=\left(\frac{1}{2}+o(1)\right)|V_i||V_j|$ if $i\neq j$ (the last equality is because $|V_i|\gg 1$ by (5) of Conditions \ref{lowerboundconditions} and the definition of $V_i$). Since the edges in blocks $(i,j)$ and $(j,i)$ are the same edges (but there is no other overlap), we see that
\[|Imp_n|=(1+o(1))\displaystyle\sum_{(i,j)\text{ important}}\frac{|V_i||V_j|}{2}.\]
Thus substituting into (\ref{logloglog1}), we see that we must show
\begin{equation}\label{logloglog2}
a_{ij}\log\log\log\frac{1}{p}+p|V_i||V_j|=o\left(I_p(W)n^2\right)
\end{equation}
for all $(i,j)$ important.

Now, by condition (6) of Conditions \ref{lowerboundconditions}, $w_{ij}\geq p$ if $(i,j)$ is important. Thus by Lemma \ref{conditionsdeductionlemma},
\[p|V_i||V_j|\leq w_{ij}|V_i||V_j|=(1+o(1))a_{ij}\ll a_{ij}\log\log\log\frac{1}{p}.\]
Finally, (6) of Conditions \ref{lowerboundconditions} and Lemma \ref{conditionsdeductionlemma} also together imply that for (i,j) important,
\[a_{ij}\log\log\log\frac{1}{p}=(1+o(1))w_{ij}m(S_i)m(S_j)n^2\log\log\log\frac{1}{p}\ll I_p(W)n^2.\]
So we have shown that $p|V_i||V_j|\ll a_{ij}\log\log\log\frac{1}{p}\ll I_p(W)n^2$ for $(i,j)$ important, so (\ref{logloglog2}) is proven. This completes the proof of Lemma \ref{Gndvsplemma}.
\end{proof}
\subsection{Proof of Lemma \ref{pvsstarlemma}}
The fact that the Radon-Nikodym derivative is constant on $\mathcal{A}_n$ is due to the fact that the edge probabilities in $\mathbb{P}_{\star}$ and $\mathbb{P}_p$ only differ in the important blocks, and are constant on each important block. Thus the Radon-Nikodym derivative only depends on the number of edges in each important block $(i,j)$, and on $\mathcal{A}_n$ this value is fixed at $a_{ij}$. We now compute the value of this derivative.

We multiply the contributions from each important block. If $(i,j)$ is important and $i\neq j$, we obtain the contribution
\[\left(\frac{w_{ij}}{p}\right)^{a_{ij}}\left(\frac{1-w_{ij}}{1-p}\right)^{|V_i||V_j|-a_{ij}}.\]
Since $|V_i||V_j|-a_{ij}$ is simply $(1-w_{ij})|V_i||V_j|$ rounded to the nearest integer, it is at most $2(1-w_{ij})|V_i||V_j|$. Therefore,
\begin{align*}
\left|\log\left(\frac{1-w_{ij}}{1-p}\right)^{|V_i||V_j|-a_{ij}}\right| & =(|V_i||V_j|-a_{ij})\log\frac{1-w_{ij}}{1-p} \\ & \leq \left(2(1-w_{ij})\log\frac{1-w_{ij}}{1-p}\right)|V_i||V_j| \\ & \leq \left(2\log\frac{1}{1-p}\right)|V_i||V_j|,
\end{align*}
as the function $x\log\frac{x}{1-p}$ is convex on $[0,1]$ and thus maximized on the interval $[0,1]$ at one of the endpoints. Since $p\ll 1$, $\log\frac{1}{1-p}<2p$, so this is bounded by $4p|V_i||V_j|$.

For an important diagonal block $(i,i)$, the contribution is
\[\left(\frac{w_{ii}}{p}\right)^{\frac{a_{ii}}{2}}\left(\frac{1-w_{ii}}{1-p}\right)^{\binom{|V_i|}{2}-\frac{a_{ii}}{2}},\]
and a similar argument shows that the logarithm of the second factor is bounded in absolute value by $4p\binom{|V_i|}{2}$.

Again, when multiplying these contributions we double count those not on the diagonal, so we must divide the exponents in those terms by two. Therefore (using that $|V_i|\gg 1$),
\begin{align*}
\frac{d\mathbb{P}_{\star}}{d\mathbb{P}_p} & =\exp\left(\displaystyle\sum_{\substack{1\leq i,j\leq k \\ (i,j)\text{ important}}}\frac{a_{ij}}{2}\log\frac{w_{ij}}{p}+O(p|V_i||V_j|)\right) \\ & =\exp\left(\left(\frac{1}{2}+o(1)\right)I_p(W)n^2+O(p)\displaystyle\sum_{\substack{1\leq i,j\leq k \\ (i,j)\text{ important}}}|V_i||V_j|\right)
\end{align*}
by Lemma \ref{importantblocklemma}.

So to prove Lemma \ref{pvsstarlemma}, it suffices to show that
$|V_i||V_j|\ll p^{-1}n^2I_p(W)$ for all $(i,j)$ important.

But $|V_i|=m(S_i)n+O(1)=(1+o(1))m(S_i)n$, by (5) of Conditions \ref{lowerboundconditions}. So we must show that
\[m(S_i)m(S_j)\ll p^{-1}I_p(W).\]
But $I_p(W)$ is at least the entropy on the block $S_i\times S_j$, which is equal to $m(S_i)m(S_j)I_p(w_{ij})$. It thus suffices to observe that $I_p(w_{ij})\gg p$, which is true as $w_{ij}\gg p$ by the definition of an important block. Thus Lemma \ref{pvsstarlemma} is proven.
\subsection{Proof of Lemma \ref{starlargelemma}}
We begin by writing $\mathbb{P}_{\star}(\mathcal{A}_n^{\deg})=\mathbb{P}_{\star}(\mathcal{A}_n)-\mathbb{P}_{\star}(\mathcal{A}_n\backslash\mathcal{A}_n^{\deg})$. We will show that the former probability is large and the latter is small, via the following two lemmas.
\begin{lemma}\label{P*An}
Under Setup \ref{ldpsetup},
\[\mathbb{P}_{\star}(\mathcal{A}_n)\geq n^{-O(1)}.\]
\end{lemma}
\begin{proof}
Let $G_*$ be a graph sampled from $\mathbb{P}_{\star}$. To compute $\mathbb{P}_{\star}(\mathcal{A}_n)$, we would like to find the probability that for all $(i,j)$ important, $|E(G_*)\cap (V_i\times V_j)|=\begin{cases}a_{ij} & i\neq j \\ \frac{a_{ij}}{2} & i=j\end{cases}$

These events for different blocks are independent, as long as we are not comparing $(i,j)$ and $(j,i)$, and the probability for block $(i,j)$ is
\[\binom{|V_i||V_j|}{a_{ij}}w_{ij}^{a_{ij}}(1-w_{ij})^{|V_i||V_j|-a_{ij}}\]
when $i\neq j$ and
\[\binom{\binom{|V_i|}{2}}{a_{ij}}w_{ij}^{\frac{a_{ij}}{2}}(1-w_{ij})^{\binom{|V_i|}{2}-\frac{a_{ij}}{2}}.\]
For $(i,j)$ important and $i\neq j$ we know that $a_{ij}$ is the closest integer to $w_{ij}|V_i||V_j|$, so it is within $1$ of being the optimum of
\[\binom{|V_i||V_j|}{a}w_{ij}^{a}(1-w_{ij})^{|V_i||V_j|-a}\]
over all integers $a$, $0\leq a\leq |V_i||V_j|$. This maximum value must be at least $\frac{1}{|V_i||V_j|+1}$ (as summing this expression over all $|V_i||V_j|+1$ values of $a$ gives $1$), and changing $a$ by $1$ multiplies the result by at least $\frac{w_{ij}(1-w_{ij})}{(a+1)(|V_i||V_j|-a+1)}$. Thus
\[\binom{|V_i||V_j|}{a_{ij}}w_{ij}^{a_{ij}}(1-w_{ij})^{|V_i||V_j|-a_{ij}}\geq\frac{w_{ij}(1-w_{ij})}{(a+1)(|V_i||V_j|-a+1)(|V_i||V_j|+1)}\]
Now, $w_{ij}\gg n^{-1}/m(S_i)=(1+o(1))/|V_i|$ by (9) of Conditions \ref{lowerboundconditions} and Lemma \ref{conditionsdeductionlemma}, so as long as $1-w_{ij}\geq\frac{1}{2|V_i||V_j|}$ we have that $\binom{|V_i||V_j|}{a_{ij}}w_{ij}^{a_{ij}}(1-w_{ij})^{|V_i||V_j|-a_{ij}}\geq\Omega\left((|V_i||V_j|)^{-5}\right)$.

If instead $1-w_{ij}<\frac{1}{2|V_i||V_j|}$, then $w_{ij}|V_i||V_j|>|V_i||V_j|-\frac{1}{2}$, so since $a_{ij}$ is the closest integer to $w_{ij}|V_i||V_j$ we must have $a_{ij}=|V_i||V_j|$. In this case,
\[\binom{|V_i||V_j|}{a_{ij}}w_{ij}^{a_{ij}}(1-w_{ij})^{|V_i||V_j|-a_{ij}}=w_{ij}^{|V_i||V_j|}\geq\left(1-\frac{1}{2|V_i||V_j|}\right)^{|V_i||V_j|}=e^{-1/2}+o(1)=\Omega(1),\]
as $|V_i|,|V_j|\gg 1$ by (5) of Conditions \ref{lowerboundconditions}.

Thus in both cases, we have obtained a lower bound of $\Omega((|V_i||V_j|)^{-5})=\Omega(n^{-10})$ when $i\neq j$. When $i=j$ a similar argument yields the same bound. Since there are fewer than $k^2$ important blocks, we have that
\[\mathbb{P}_{\star}(\mathcal{A}_n)=\Omega\left(n^{-10k^2}\right)=\Omega\left(n^{O(1)}\right),\]
as desired.
\end{proof}
\begin{lemma}\label{P*Andeg}
Under Setup \ref{ldpsetup},
\[\mathbb{P}_{\star}(\mathcal{A}_n\backslash\mathcal{A}_n^{\deg})\leq n^{-\omega(1)}.\]
\end{lemma}
\begin{proof}
Again, let $G_*$ be a graph sampled from $\mathbb{P}_{\star}$. If $G_*\in\mathcal{A}_n\backslash\mathcal{A}_n^{\deg}$, one of the degree conditions given in $\mathcal{A}_n^{\deg}$ must not hold. In particular, there is an important block $V_i\times V_j$ and a vertex $v\in V_i$ such that $|N_{V_j}(v)|\geq 2\frac{a_{ij}}{|V_i|}$. By Lemma \ref{conditionsdeductionlemma}, $2\frac{a_{ij}}{|V_i|}=(2+o(1))w_{ij}|V_j|$.

But $|N_{V_j}(v)|$, as a random variable under $\mathbb{P}_{\star}$, is simply a sum of $|V_j|$ (or $|V_j|-1$ if $i=j$) independent Bernoulli random variables with probability $w_{ij}$. Therefore, since the sum has mean $(1+o(1))w_{ij}|V_j|$ we may apply a Chernoff bound to see that
\[\text{Pr}\left[|N_{V_j}(v)|\geq (2+o(1))w_{ij}|V_j|\right]\leq\exp(-\Omega(w_{ij}|V_j|)).\]
which is $n^{-\omega(1)}$, as $w_{ij}|V_j|\gg\log n$ by (9) of Conditions \ref{lowerboundconditions}.

Initially we chose some important block $(i,j)$ and vertex $v$. There are at most $k^2n$ such choices, so a union bound yields
\[\mathbb{P}_{\star}(\mathcal{A}_n\backslash\mathcal{A}_n^{\deg})\leq k^2n\cdot n^{-\omega(1)}=n^{-\omega(1)}.\]
\end{proof}
Now it is easy to complete the proof of Lemma \ref{starlargelemma}.
\begin{proof}[Proof of Lemma \ref{starlargelemma}]
Lemmas \ref{P*An} and \ref{P*Andeg} together show that
\[\mathbb{P}_{\star}(\mathcal{A}_n^{\deg})\geq n^{-O(1)}=\exp(-O(\log n)).\]
Since $I_p(W)\gg n^{-1}\log n\gg n^{-2}\log n$ by (4) of Conditions \ref{lowerboundconditions}
\[\mathbb{P}_{\star}(\mathcal{A}_n^{\deg})=\exp\left(-o(I_p(W)n^2)\right),\]
as desired.
\end{proof}
\section{Proof of Proposition \ref{secondtermprop}}\label{secondtermsection}
The general idea of this section is to reduce the proposition to proving a lower tail bound, given by Proposition \ref{lowertailboundprop}. We will then prove Proposition \ref{lowertailboundprop}. There are several difficulties not faced in \cite{BD}, which we will address.

\subsection{Reduction to Lower Tail Bound}
We begin with a change of measure to $\mathbb{P}_{\star}$. This is accomplished by the following lemma.
\begin{lemma}\label{changeofmeasurelemma}
Under Setup \ref{ldpsetup},
\[\mathbb{P}_{G_n^d}(\mathcal{A}_n^{\deg}\land\neg\mathcal{H}_{\delta'})\leq\exp\left(-\left(\frac{1}{2}+o(1)\right)I_p(W)n^2\right)\mathbb{P}_{\star}(\mathcal{A}_n^{\deg}\land\neg\mathcal{H}_{\delta'}).\]
\end{lemma}
\begin{proof}
First, recall the definition of $\mathcal{K}_{n,d}$ from Definition \ref{eventsdef}. Note that since $\mathbb{P}_{G_n^d}$ is simply $\mathbb{P}_p$ restricted to $\mathcal{K}_{n,d}$ and renormalized,
\begin{equation}\label{Ppfirst}
\mathbb{P}_{G_n^d}(\mathcal{A}_n^{\deg}\land\neg\mathcal{H}_{\delta'})=\frac{\mathbb{P}_{p}(\mathcal{K}_{n,d}\land\mathcal{A}_n^{\deg}\land\neg\mathcal{H}_{\delta'})}{\mathbb{P}_p(\mathcal{K}_{n,d})}.
\end{equation}
Now, (2.38) of \cite{BD} states that
\[\log\mathbb{P}_p(\mathcal{K}_{n,d})\sim-\frac{1}{2}n\log d\]
whenever $1\ll d\ll n$, which occurs by as $p\gg n^{-1}$. Also, by (4) of Conditions \ref{lowerboundconditions}, $n\log d=o(I_p(W)n^2)$, so $\mathbb{P}_p(\mathcal{K}_{n,d})=\exp(-o(I_p(W)n^2))$. Subsituting into (\ref{Ppfirst}),
\[\mathbb{P}_{G_n^d}(\mathcal{A}_n^{\deg}\land\neg\mathcal{H}_{\delta'})=\exp(-o(I_p(W)n^2))\mathbb{P}_{p}(\mathcal{K}_{n,d}\land\mathcal{A}_n^{\deg}\land\neg\mathcal{H}_{\delta'}).\]
Now, Lemma \ref{pvsstarlemma} allows us to change measure to $\mathbb{P}_{\star}$, showing that
\[\mathbb{P}_{p}(\mathcal{K}_{n,d}\land\mathcal{A}_n^{\deg}\land\neg\mathcal{H}_{\delta'})=\exp\left(-\left(\frac{1}{2}+o(1)\right)I_p(W)n^2\right)\mathbb{P}_{\star}(\mathcal{K}_{n,d}\land\mathcal{A}_n^{\deg}\land\neg\mathcal{H}_{\delta'}),\]
since $\mathcal{K}_{n,d}\land\mathcal{A}_n^{\deg}\land\neg\mathcal{H}_{\delta'}\subseteq\mathcal{A}_n$.
Thus
\begin{align*}
\mathbb{P}_{G_n^d}(\mathcal{A}_n^{\deg}\land\neg\mathcal{H}_{\delta'}) & =\exp\left(-\left(\frac{1}{2}+o(1)\right)I_p(W)n^2\right)\mathbb{P}_{\star}(\mathcal{K}_{n,d}\land\mathcal{A}_n^{\deg}\land\neg\mathcal{H}_{\delta'}) \\ & \leq\exp\left(-\left(\frac{1}{2}+o(1)\right)I_p(W)n^2\right)\mathbb{P}_{\star}(\mathcal{A}_n^{\deg}\land\neg\mathcal{H}_{\delta'}).
\end{align*}
\end{proof}
Lemma \ref{changeofmeasurelemma} means that it suffices to show that
\begin{equation}\label{P*secondtermequation}
\mathbb{P}_{\star}(\mathcal{A}_n^{\deg}\land\neg\mathcal{H}_{\delta'})=\exp(-\Omega(I_p(W)n^2)).
\end{equation}
\begin{remark}
In \cite{BD}, at this point the bound was shown by dropping the $\mathcal{A}_n^{\deg}$ altogether and showing that $\mathbb{P}_{\star}(\neg\mathcal{H}_{\delta'})=\exp(-\Omega(I_p(W)n^2))$. However, for us this statement will not necessarily even be true. When $K=K_0$ is our graph formed by adding an edge to $K_{2,4}$ and $W$ is the optimum graphon from Figure \ref{K24graphonfigure}, $I_p(W)\gtrsim p^3 \left(\log\frac{1}{p}\right)^{\frac{2}{3}}\left(\log\log\frac{1}{p}\right)^{\frac{1}{3}}$, but we should expect that $\mathbb{P}_{\star}(\neg\mathcal{H}_{\delta'})\geq\exp\left(-O\left(p^3\left(\log\frac{1}{p}\right)^{\frac{2}{3}}\left(\log\log\frac{1}{p}\right)^{-\frac{2}{3}}\right)\right)$. Informally, this is because if we sample a graph $G_*$ according to $\mathbb{P}_{\star}$, this is the probability that it has many fewer edges than expected in the unique somewhat important block. We get around the problem that somewhat important blocks could have many fewer edges than expected with too high a probability by enforcing the event $\mathcal{A}_n^{\deg}$, which prevents that from happening.

In general, as this example shows, the important blocks are small enough that to allow them to be sampled from randomly would raise the probability of a lower tail event too high. However, (8) and (9) of Conditions \ref{lowerboundconditions} will guarantee that the unimportant blocks are large in exactly the way we need to properly apply Janson's inequality. We thus would like to fix the set of edges in important blocks while choosing the edges in unimportant blocks at random. It turns out that (7) of Conditions \ref{lowerboundconditions} will imply that no matter how we fix the set of edges in the important blocks, the expected number of homomorphisms from $K$ stays constant.
\end{remark}
Recall that $\mathcal{B}^S$ is the set of graphs $T$ with $E(T)\cap Imp_n=S$.
\begin{definition}\label{psdef}
Under Setup \ref{ldpsetup}, for any $S\subset Imp_n$, let $\mathbb{P}_S$ be the probability distribution given by restricting $\mathbb{P}_{\star}$ (or $\mathbb{P}_p$) to $\mathcal{B}^S$ and renormalizing. In other words, it is the distribution given by taking a graph where we take each edge in $S$ with probability $1$, each edge in $Imp_n\backslash S$ with probability $0$, and each edge not in $Imp_n$ with probability $p$.

Call an $S\subseteq Imp_n$ \emph{allowed} if $\mathcal{B}^S\subseteq\mathbb{A}_n^{\deg}$; that is, if for all important $(i,j)$ there are $a_{ij}$ pairs $(v_i,v_j)$ such that $v_i\in V_i$, $v_j\in V_j$, and $v_iv_j\in S$, and for all $v_i\in V-i$, $|N_S(v_i)\cap V_j|\leq 2\frac{a_{ij}}{|V_i|}$.

We now change measure again to $\mathbb{P}_S$.
\end{definition}
\begin{lemma}\label{easyproblemma}
Assuming Setup \ref{ldpsetup},
\[\mathbb{P}_{\star}(\mathcal{A}_n^{\deg}\land\neg\mathcal{H}_{\delta'})\leq\displaystyle\max_{\text{allowed }S}\mathbb{P}_S(\neg\mathcal{H}_{\delta'}).\]
\end{lemma}
\begin{proof}
Since $\mathcal{P}_S$ is the restriction of $\mathbb{P}_{\star}$ to $\mathcal{B}_S$, we must have that for any $S$ and any event $\mathcal{A}$,
\[\mathbb{P}_{\star}(\mathcal{B}_S\land\mathcal{A})=\mathbb{P}_{\star}(\mathcal{B}_S)\mathbb{P}_{\star}((\mathcal{B}_S\land\mathcal{A})|\mathcal{B}_S)=\mathbb{P}_{\star}(\mathcal{B}_S)\mathbb{P}_S(\mathcal{A}).\]

By the definition of $\mathcal{A}_n^{\deg}$ and allowed $S$,
\[\mathcal{A}_n^{\deg}=\displaystyle\bigvee_{\text{allowed }S}\mathcal{B}_S.\]
Since the $\mathcal{B}_S$ are disjoint,
\begin{align*}
\mathbb{P}_{\star}(\mathcal{A}_n^{\deg}\land\neg\mathcal{H}_{\delta'}) & =\displaystyle\sum_{\text{allowed }S}\mathbb{P}_{\star}(\mathcal{B}_S\land\neg\mathcal{H}_{\delta'}) \\ & =\displaystyle\sum_{\text{allowed }S}\mathbb{P}_{\star}(\mathcal{B}_S)\mathbb{P}_S(\neg\mathcal{H}_{\delta'}) \\ & \leq\displaystyle\max_{\text{allowed }S}\mathbb{P}_S(\neg\mathcal{H}_{\delta'})\displaystyle\sum_{\text{allowed }S}\mathbb{P}_{\star}(\mathcal{B}_S) \\ & =\mathbb{P}_{\star}(\mathcal{A}_n^{\deg})\displaystyle\max_{\text{allowed }S}\mathbb{P}_S(\neg\mathcal{H}_{\delta'}) \\ & \leq\displaystyle\max_{\text{allowed }S}\mathbb{P}_S(\neg\mathcal{H}_{\delta'}).
\end{align*}
\end{proof}
The importance of Lemma \ref{easyproblemma} is that it allows us to fix the set of edges we are choosing from important blocks. For each allowed $S\subseteq Imp_n$, we now would like to bound above $\mathbb{P}_S(\mathcal{H}_{\delta'})$.

The following definition and lemma will be useful.
\begin{definition}\label{hombgtdef}
Take $V_1,\ldots,V_k$ to be the partition of $[n]$ from Setup \ref{ldpsetup}. Let $T$ be any graph on $[n]$. For any $B:V(K)\to k$ we define $\Hom_B(K,T)$ to be the number of homomorphisms from $K$ to $T$ where each $v\in V(K)$ is sent into $V_{B(v)}\subseteq [n]$.
\end{definition}
\begin{lemma}\label{GSexpectationlemma}
Assume Setup \ref{ldpsetup}. Let $\epsilon>0$ be any constant (independent of $n$). Take some allowed $S\in\text{Imp}_n$ and and a $K$-block $B$ with $\Hom_B(K,W)\geq\epsilon p^{e(K)}$. Recall the definition of $\mathbb{P}_S$ from Definition \ref{psdef} and let $G_S$ be a random graph sampled from $\mathbb{P}_S$. Then
\[\mathbb{E}(\Hom_B(K,G_S))\geq (1-o(1))\Hom_B(K,W)n^{v(K)}.\]
\end{lemma}
\begin{proof}
Note that $\Hom_B(K,W)=\displaystyle\prod_{v\in V(K)}m(S_{B(v)})\displaystyle\prod_{uv\in E(K)}w_{B(u)B(v)}$ by definition. Since that $|V_i|=(1+o(1))m(S_i)n$ for all $i$ by Lemma \ref{conditionsdeductionlemma}, it suffices to show that
\begin{equation}\label{wequation}
\mathbb{E}(\Hom_B(K,G_S))\geq (1-o(1))\displaystyle\prod_{v\in V(K)}|V_{B(v)}|\displaystyle\prod_{uv\in E(K)}w_{B(u)B(v)}.
\end{equation}

We use (7) of Conditions \ref{lowerboundconditions}, which implies that all edges of $K$ that $B$ sends to somewhat important blocks are disjoint. Let $D\subseteq E(K)$ be this set of disjoint edges, and let $V(D)$ be the set of their vertices.

We count homomorphisms where all vertices in $K$ have distinct image, and we will count them by the position of the vertices in $V(D)$. These are uniquely determined by the (ordered) edges that we choose to send the edges in $D$ to, as $D$ is simply a union of disjoint edges. The number of choices for the image of some edge $uv\in D$ is at least $a_{B(u)B(v)}-e(K)(|V_{B(u)}|+|V_{B(v)}|)$, as $(B(u),B(v))$ must be important by the definition of $D$ and $S$ has exactly $a_{ij}$ (ordered) edges in each important block $V_i\times V_j$, and there are at most $e(K)(|V_{B(u)}|+|V_{B(v)}|)$ edges of the block that intersect an edge already chosen.

After this, send the vertices in $V(K)\backslash V(D)$ to arbitrary distinct vertices of the appropriate block. Now, notice that in any very important block $(i,j)$, $w_{ij}=1$ by definition, so we have that $S|_{V_i\times V_j}=K(V_i,V_j)$ for all valid $S$ (as $a_{ij}=|V_i||V_j|$ if $i\neq j$ and $\binom{|V_i|}{2}$ if $i=j$). Therefore, the image of any edge of $K$ sent to a very important block will always be contained in $G_S$.

It suffices to deal with the unimportant blocks. Let $uv$ be an edge of $K$ sent to an unimportant block by our proposed homomorphism. Its image is contained in $G_S$ with probability $p$ (as unimportant blocks are sampled with probability $p$). Since we are sending all vertices of $K$ to distinct vertices, all of these probabilities from the edges of $K$ are independent. Thus the probability that this is a valid homomorphism to $G_S$ is
\[p^{|\{uv\in E(K): (B(u),B(v))\text{ unimportant}\}|}.\]
By linearity of expectation, since there were $\displaystyle\prod_{uv\in D}(a_{B(u)B(v)}-e(K)(|V_{B(u)}|+|V_{B(v)}|))$ choices of where to send the vertices in $V(D)$ and at least $|V_{B(v)}|-v(K)$ choices of where to send the vertices in $V(K)\backslash V(D)$ (similarly, we subtract the $v(K)$ to account for the fact that we are choosing the images to be distinct), we have that $\mathbb{E}(\Hom_B(K,G_S))$ is at least

\[p^{|\{uv\in E(K): (B(u),B(v))\text{ unimportant}\}|}\displaystyle\prod_{uv\in D}(a_{B(u)B(v)}-e(K)(|V_{B(u)}|+|V_{B(v)}|))\displaystyle\prod_{v\in V(K)\backslash V(D)}(|V_{B(v)}|-v(K)).\]
Now, $|V_i|\gg 1$ by Lemma \ref{conditionsdeductionlemma}, so we may write $|V_{B(v)}|-v(K)=(1+o(1))|V_{B(v)}|$. Similarly, we would like to bound the factors of the first product. For $uv\in D$, $(B(u),B(v))$ is (somewhat) important, so we may apply Lemma \ref{conditionsdeductionlemma} and (9) of Conditions \ref{lowerboundconditions} to obtain
\[w_{B(u)B(v)}|V_{B(u)}|=(1+o(1))w_{B(u)B(v)}m(S_{B_u})n\gg 1\]
Therefore, since $a_{B(u)B(v)}$ is the closest integer to $w_{B(u)B(v)}|V_{B(u)}||V_{B(v)}|\gg 1$, we have that $a_{B(u)B(v)}\gg |V_{B(u)}|$ and thus
\[a_{B(u)B(v)}-e(K)(|V_{B(u)}|+|V_{B(v)}|)=(1+o(1))w_{B(u)B(v)}|V_{B(u)}||V_{B(v)}|.\]
Substituting in, we see that
\begin{align}
\mathbb{E}(\Hom_B(K,G_S)) & \geq (1+o(1))p^{|\{uv\in E(K): (B(u),B(v))\text{ unimportant}\}|}\displaystyle\prod_{v\in V(K)}|V_{B(v)}|\displaystyle\prod_{uv\in D}w_{B(u)B(v)} \\ & \label{w'equation} =(1+o(1))\displaystyle\prod_{v\in V(K)}|V_{B(v)}|\displaystyle\prod_{uv\in E(K)}w'_{B(u)B(v)},
\end{align}
where we define $w'_{ij}=w_{ij}$ if $(i,j)$ is important and $w'_{ij}=p$ otherwise. The last equality follows from the fact that $w'_{ij}=1$ for all very important blocks, and that the edges in $D$ are the only ones sent into somewhat important blocks.
Comparing (\ref{wequation}) with (\ref{w'equation}), we see that it suffices to show that $w'_{ij}\geq (1+o(1))w_{ij}$ for all $1\leq i,j\leq k$.

But when $(i,j)$ is important then $w'_{ij}=w_{ij}$, and when $(i,j)$ is unimportant, then $w'_{ij}=p$ by definition and $w_{ij}=p+o(p)$ by (8) of Condition \ref{lowerboundconditions}. This completes the proof.
\end{proof}
Using the lemma above, we can bound $\neg\mathcal{H}_{\delta'}$ by something that looks like a lower tail event.
\begin{cor}\label{lowertailcor}
Assume Setup \ref{ldpsetup}. Let $S\in\text{Imp}_n$ be allowed and let $G_S$ be a random graph sampled from $\mathbb{P}_S$, with $\mathbb{P}_S$ as in Definition \ref{psdef}. Take $\delta'=\delta'(n)$ with $\delta'\leq\delta-\Omega(1)$. There exists $\epsilon>0$ (not depending on $n$ or $S$) such that for all sufficiently large $n$, if $G_S\in\neg\mathcal{H}_{\delta'}$, there exists some $K$-block $B$ such that $\Hom_B(K,W)\geq\epsilon p^{e(K)}$ and
\[\Hom_B(K,G_S)\leq (1-\epsilon)\mathbb{E}(\Hom_B(K,G_S)),\]
where $\Hom_B(K,G_S)$ is defined as in Definition \ref{hombgtdef}.
\end{cor}
\begin{proof}
Since $\Hom(K,W)=(1+\delta)p^{e(K)}$ by the definintion of $\delta$ in Setup \ref{ldpsetup}, if $G_S\in\neg\mathcal{H}_{\delta'}$, we have that
\begin{align*}
\displaystyle\sum_{B\text{ a }K\text{-block}}\Hom_B(K,G_S) & =\Hom(K,G_S) \\ & \leq (1+\delta')p^{e(K)}n^{v(K)} \\ & \leq\Hom(K,W)n^{v(K)}-(\delta-\delta')p^{e(K)}n^{v(K)} \\ & \leq\displaystyle\sum_{B\text{ a }K\text{-block}}\Hom_B(K,W)n^{v(K)}-(\delta-\delta')p^{e(K)}n^{v(K)}.
\end{align*}
Since there are at most $k^{v(K)}$ $K$-blocks, there is some $K$-block $B$ with
\begin{equation}\label{Kblockboundequation}
\Hom_B(K,G_S)\leq\Hom_B(K,W)n^{v(K)}-\frac{\delta-\delta'}{k^{v(K)}}p^{e(K)}n^{v(K)}.
\end{equation}
Thus we must have $\Hom_B(K,W)\geq\frac{\delta-\delta'}{k^{v(K)}}p^{e(K)}$, so for any choice of $\epsilon\leq\frac{\delta-\delta'}{k^{v(K)}}=\Omega(1)$ we have satisfied the desired condition $\Hom_B(K,W)\geq\epsilon p^{e(K)}$.

Furthermore, $\Hom_B(K,W)\leq\Hom(K,W+p)=O(p^{e(K)})$ by (10) of Conditions \ref{lowerboundconditions}, so (\ref{Kblockboundequation}) implies that
\[\Hom_B(K,G_S)\leq (1-\Omega(1))\Hom_B(K,W)n^{v(K)}\leq (1-2\epsilon)\Hom_B(K,W)n^{v(K)}\]
for some sufficiently small $\epsilon$. Thus by Lemma \ref{GSexpectationlemma},
\[\Hom_B(K,G_S)\leq (1-2\epsilon+o(1))\mathbb{E}(\Hom_B(K,G_S))\leq (1-\epsilon)\mathbb{E}(\Hom_B(K,G_S))\]
for sufficiently large $n$. Since our choice of $\epsilon$ in no way depended on $S$, this proves the corollary.
\end{proof}
The following proposition will bound the desired lower tail event.
\begin{prop}\label{lowertailboundprop}
Assume Setup \ref{ldpsetup}. Let $\epsilon>0$ be any constant (not depending on $n$). There exists a constant $C>0$ such that for all sufficiently large $n$ and any allowed $S\subset\text{Imp}_n$, and for all $K$-blocks $B$ with $\Hom_B(K,W)\geq\epsilon p^{e(K)}$,
\[\text{Pr}\left[\Hom_B(K,G_S)\leq (1-\epsilon)\mathbb{E}\left(\Hom_B(K,G_S)\right)\right]\leq\exp\left(-C\cdot I_p(W)n^2\right),\]
where $\mathbb{P}_S$ is defined as in Definition \ref{psdef}, $G_S$ is a random graph sampled from $\mathbb{P}_S$, and $\Hom_B(K,G_S)$ is defined as in Definition \ref{hombgtdef}.
\end{prop}
We will prove this proposition in the next subsection. For now, we conclude this section with the proof of Proposition \ref{secondtermprop}.
\begin{proof}[Proof of Proposition \ref{secondtermprop} given Proposition \ref{lowertailboundprop}]
Take $\delta'\leq\delta-\Omega(1)$. By Lemma \ref{changeofmeasurelemma}, Lemma \ref{easyproblemma}, and Corollary \ref{lowertailcor},
\begin{align*}
\mathbb{P}_{G_n^d}(\mathcal{A}_n^{\deg}\land\neg\mathcal{H}_{\delta'}) & \leq\exp\left(-\left(\frac{1}{2}+o(1)\right)I_p(W)n^2\right)\mathbb{P}_{\star}(\mathcal{A}_n^{\deg}\land\neg\mathcal{H}_{\delta'}) \\ & \leq\exp\left(-\left(\frac{1}{2}+o(1)\right)I_p(W)n^2\right)\displaystyle\max_{\text{allowed }S}\mathbb{P}_S(\neg\mathcal{H}_{\delta'}) \\ & \leq\exp\left(-\left(\frac{1}{2}+o(1)\right)I_p(W)n^2\right) \\ & \cdot\displaystyle\max_{\text{allowed }S}\displaystyle\sum_{\substack{B\text{ a }K\text{-block} \\ \Hom_B(K,W)\geq\epsilon p^{e(K)}}}\left(\Pr\left[\Hom_B(K,G_S)\leq (1-\epsilon')\mathbb{E}(\Hom_B(K,G_S))\right]\right)
\end{align*}
for all sufficiently large $n$ and some $\epsilon'>0$ not depending on $n$. But by Proposition \ref{lowertailboundprop}, there exists $C>0$ such that
\[\Pr\left[\Hom_B(K,G_S)\leq (1-\epsilon')\mathbb{E}(\Hom_B(K,G_S))\right]\leq\exp(-C\cdot I_p(W)n^2)\]
for all $S\subseteq Imp_n$, all $K$-blocks $B$ with $\Hom_B(K,W)\geq\epsilon p^{e(K)}$, and all sufficiently large $n$. Thus (since the number of $K$-blocks is $O(1)$),
\[\mathbb{P}_{G_n^d}(\mathcal{A}_n^{\deg}\land\neg\mathcal{H}_{\delta'})\leq\exp\left(\left(-\frac{1}{2}+\Omega(1)\right)I_p(W)n^2\right),\]
so Proposition \ref{secondtermprop} is proven.
\end{proof}
\subsection{Bounding the Lower Tail}
The goal of this subsection is to prove Proposition \ref{lowertailboundprop}, whereupon the work done in the previous section implies Proposition \ref{secondtermprop}. We will need to use Janson's inequality, which we now state.
\begin{thm}[Janson \cite{Janson}]\label{Jansoninequality}
Let $N,m\in\mathbb{Z}^+$. Let $J\subset [N]$ be a random subset where each element is chosen independently at random. Let $Q_i$, $1\leq i\leq m$ be fixed (not necesarily distinct) subsets of $[N]$, and define the random indicator variable $D_i:=1_{Q_i\in J}$ and $D:=\displaystyle\sum_{i=1}^m D_i$. Then
\[\Pr[D\leq (1-\epsilon)\mu]\leq\exp\left(-\frac{\epsilon^2\mu^2}{2\left(\mu+\displaystyle\sum_{\substack{1\leq i,j\leq m \\ i\sim j}}\mathbb{E}(D_iD_j)\right)}\right),\]
where $\mu=\mathbb{E}D$ and we say that $i\sim j$ if $i\neq j$ and $Q_i\cap Q_j\neq\emptyset$ (in other words, $D_i$ and $D_j$ are correlated).
\end{thm}
We now prove the proposition.
\begin{proof}[Proof of Proposition \ref{lowertailboundprop}]
Assume Setup \ref{ldpsetup}. Take some absolute constant $\epsilon>0$. Further take some $S\subseteq Imp_n$ and some $K$-block $B$ with $\Hom_B(K,W)\geq\epsilon p^{e(K)}$. We must show that $\text{Pr}\left[\Hom_B(K,G_S)\leq (1-\epsilon)\mathbb{E}\left(\Hom_B(K,G_S)\right)\right]\leq\exp\left(-C\cdot I_p(W)n^2\right)$.

Call a homomorphism from $K$ to $K_n$ a $B$-homomorphism if every $v\in K$ is sent into $V_{B(v)}\subseteq [n]$. Call a $B$-homomorphism $\phi$ from $K$ to $K_n$ \emph{possible} if its image does not intersect $Imp_n\backslash S$ (as then it can never be a valid homomorphism), or in other words, if it is a homomorphism from $K$ to $S\cup Unimp_n$. 

We apply Theorem \ref{Jansoninequality} with $[N]\subset E(K_n)$ being the set of edges in unimportant blocks $E(K_n)\backslash Imp_n$, which we call $Unimp_n$. For each possible $B$-homomorphism $\phi$ we create a subset $Q_{\phi}=\text{Im}(\phi)\backslash S\subset Unimp_n$. (Note that many of the $Q_{\phi}$ may be the same, as the image of several homomorphisms may differ only within $S$.)

Then if we sample a set $J\subseteq Unimp_n$ where each element is selected with probability $p$, we are in essence sampling our random graph $G_S$ (which will have edge set $J\cup S$), and a possible homomorphism $\phi$ from $K$ to $G_S$ is valid if and only if $Q_{\phi}\subseteq J$. Thus letting $D_{\phi}=1_{Q_{\phi}\subseteq J}$, then $\displaystyle\sum_{\phi}D_{\phi}=\Hom_B(G,G_S)$. Thus Theorem \ref{Jansoninequality} states that
\[\Pr[\Hom_B(K,G_S)\leq (1-\epsilon)\mathbb{E}(\Hom_B(K,G_S))]\leq\exp\left(\frac{-\epsilon^2(\mathbb{E}(\Hom_B(K,G_S)))^2}{2\left(\displaystyle\sum_{\substack{\phi,\phi'\\ \phi\sim \phi'}}\Pr\left(\text{Im}(\phi)\cup\text{Im}(\phi')\subseteq G_S\right)\right)}\right),\]
where the sum runs over pairs of possible $B$-homomorphisms from $K$ to $G_S$, and $\phi\sim\phi'$ if $\text{Im}(\phi)$ and $\text{Im}(\phi')$ intersect in at least one edge in an unimportant block. (The $\mu$ was eliminated from the denominator of Theorem \ref{Jansoninequality} because we redefined $\sim$ to be possible when $\phi=\phi'$, and under the new definition we always have $\phi\sim\phi$ because $K$ contains at least one edge sent by $B$ into an unimportant block.)

The fact that we only consider pairs $\phi,\phi'$ which share an edge in an \emph{unimportant} block (as opposed to any edge) is the point of fixing the edge set $S$.
Now, by Lemma \ref{GSexpectationlemma}, since $\Hom_B(K,W)\geq\epsilon p^{e(K)}$, we have that $\mathbb{E}(\Hom_B(K,G_S))\geq (1-o(1))\Hom_B(K,W)n^{v(K)}\geq\frac{\epsilon}{2}p^{e(K)}n^{v(K)}$ for sufficiently large $n$, so to prove Proposition \ref{lowertailboundprop} it suffices to show that
\begin{equation}\label{necessarydeltabound}
\displaystyle\sum_{\substack{\phi,\phi'\text{ possible }B\text{-homomorphisms}\\ \phi\sim \phi'}}\Pr\left(\text{Im}(\phi)\cup\text{Im}(\phi')\subseteq G_S\right)\leq C'\cdot\frac{n^{2v(K)-2}p^{2e(K)}}{I_p(W)}
\end{equation}
for some $C'>0$ not depending on $S$ and all sufficiently large $n$.

We bound the left side of (\ref{necessarydeltabound}) by breaking into sums based on the isomorphism class of $\text{Im}(\phi)\cup\text{Im}(\phi')$.

First suppose $v(\text{Im}(\phi)\cup\text{Im}(\phi')=2v(K)-2$. Then since $\phi\sim\phi'$, $\text{Im}(\phi)$ and $\text{Im}(\phi')$ must be injective copies of $K$ joined along a single unimportant edge; say $uv,u'v'\in E(K)$ with $\phi(u)=\phi'(u')$ and $\phi(v)=\phi'(v')$.

The sum of all terms in the sum where $\phi$ and $\phi'$ satisfy these relations---that $\phi(u)=\phi'(u')$ and $\phi(v)=\phi'(v')$, $\phi$ and $\phi'$ are injective, and there are no other overlaps in the values of $\phi$ and $\phi'$---is simply equal to the expected number of injective $B$-homomorphisms from $H$ to $G_S$, where $H$ is the graph given by joining $K$ to itself by identifying $u$ and $v$ in one copy of $K$ with $u'$ and $v'$ in the other, respectively. (Note that there is a natural definition of a $B$-homomorphism from $H$ to $G_S$, as there is a natural map $V(H)\to [k]$ induced by the two maps $V(K)\to [k]$ on both copies of $K$ that form $H$. They must agree on the overlap or else this sum is trivially $0$.)

We upper bound the number of such injective $B$-homomorphisms. Call edges of $H$ very important/somewhat important/unimportant if their image blocks under $B$ are. First, we determine where we are sending the somewhat important edges of $H$. Let $H_{SI}$ be the subgraph of $H$ given by the somewhat important edges. Since the edge that the two copies of $K$ making up $H$ intersect in is unimportant, and the somewhat important edges in each copy of $K$ form a matching by (7) of Conditions \ref{lowerboundconditions}, $H_{SI}$ is the union of a matching and up to two copies of $K_{1,2}$.

For each edge $e=uv\in E(H_{SI})$ that is not part of a $K_{1,2}$, since we must send $u$ into $V_{B(u)}$ and $v$ into $V_{B(v)}$, there are $a_{B(u)B(v)}$ choices for where to send the pair $(u,v)$ so that $uv$ maps into an edge, by the definition of $a_{ij}$. By Lemma \ref{conditionsdeductionlemma}, $a_{B(u)B(v)}=(1+o(1)) w_{B(u)B(v)}|V_{B(u)}||V_{B(v)}|$ since $(B(u),B(v))$ is (somewhat) important by assumption.

Now, consider a $K_{1,2}$ formed by $u,v,w\in V(H_{SI})$, $uv,uw\in E(H_{SI})$. We use the degree condition in the validity of $S$. In particular, for any $u'\in V_{B(u)}$, if our homomorphism sends $u$ to $u'$, we must send $v$ into $N_S(u')\cap V_{B(v)}$, which by the validity of $S$ (which in turn relies on the definition of $\mathcal{A}_n^{\deg}$) has cardinality at most $2\frac{a_{B(u)B(v)}}{|V_{B(u)}|}$. Similarly, we have at most $2\frac{a_{B(u)B(w)}}{|V_{B(u)}|}$ choices for where to send $w$. Since we have $|V_{B(u)}|$ choices for where to send $u$ initially, our total number of choices is at most
\begin{align*}
4\frac{a_{B(u)B(v)}}{|V_{B(u)}|}\frac{a_{B(u)B(w)}}{|V_{B(u)}|}|V_{B(u)}| & =4\frac{a_{B(u)B(v)}a_{B(u)B(w)}}{|V_{B(u)}|} \\ & =(4+o(1))w_{B(u)B(v)}w_{B(u)B(w)}|V_{B(u)}||V_{B(v)}||V_{B(w)}|
\end{align*}

Thus the total number of ways to embed $H_{SI}$ is at most
\[(16+o(1))\displaystyle\prod_{v\in V(H_{SI})}|V_{B(v)}|\displaystyle\prod_{uv\in E(H_{SI})}w_{B(u)B(v)}.\]

Embedding the rest of the vertices arbitrarily in the appropriate $V_i$ (and sending them to distinct vertices), the image of each very important edge of $H$ is in $G_S$ with probability $1$ (since $S\in G_S$ covers the entirety of all very important blocks) and the image of each unimportant edge of $H$ is in $G_S$ with probability $p$. Since all vertices of $H$ are sent to distinct vertices in $[n]$, all of these probabilities are independent. Thus letting $w'_{ij}=w_{ij}$ if $(i,j)$ is important and $w'_{ij}=p$, the expected number of $B$-homomorphisms from $H$ is at most
\[(16+o(1))\displaystyle\prod_{v\in V(H)}|V_{B(v)}|\displaystyle\prod_{uv\in V(H)}w'_{B(u)B(v)}.\]
By (8) of Conditions \ref{lowerboundconditions}, $w'_{ij}=(1+o(1))w_{ij}$ for all pairs $(i,j)$ appearing in our product (since all blocks appearing must be images of some edge in $H$ and thus the image of some edge in $K$). So noting that $|V_i|=(1+o(1))m(S_i)n$ by Lemma \ref{conditionsdeductionlemma}, our upper bound becomes
\begin{align*}
(16+o(1))n^{v(H)}\displaystyle\prod_{v\in V(H)}m(S_{B(v)})\displaystyle\prod_{uv\in V(H)}w_{B(u)B(v)} & =(16+o(1))n^{2v(K)-2}\Hom_B(H,W) \\ & =(16+o(1))n^{2v(K)-2}\frac{\Hom_B(K,W)^2}{m(S_{B(u_0)})m(S_{B(v_0)})w_{B(u_0)B(v_0)}},
\end{align*}
where $u_0v_0$ is the edge of $H$ where the two copies of $K$ overlap. By (10) of Conditions \ref{lowerboundconditions}, $\Hom_B(K,W)\leq\Hom(K,W)\leq\Hom(K,W+p)=O(p^{e(K)})$, so we have obtained an upper bound of
\[O\left(\frac{n^{2v(K)-2}p^{2e(K)}}{m(S_{B(u_0)})m(S_{B(v_0)})w_{B(u_0)B(v_0)}}\right).\]
By (8) of Conditions \ref{lowerboundconditions}, since $(B(u_0),B(v_0))$ is unimportant, $m(S_{B(u_0)})m(S_{B(v_0)})w_{B(u_0)B(v_0)}\geq (1+o(1))I_p(W)$, so we have proven the desired $O\left(\frac{n^{2v(K)-2}p^{2e(K)}}{I_p(W)}\right)$ bound for the terms where $v(\text{Im}(\phi)\cup\text{Im}(\phi'))=2v(K)-2$.

We now tackle the terms where $v(\text{Im}(\phi)\cup\text{Im}(\phi'))<2v(K)-2$. (We must have $v(\text{Im}(\phi)\cup\text{Im}(\phi'))\leq 2v(K)-2$ as the two homomorphisms $\phi$ and $\phi'$ overlap on an edge.) However, here there is a trivial bound of $n^{2v(K)-3}$ for the number of homomorphisms, so since $n\gg p^{-2e(K)}I_p(W)$ by (4) of Conditions \ref{lowerboundconditions} we have the desired bound in this case too, proving Proposition \ref{lowertailboundprop}.
\end{proof}
We have now completed the proof of Theorem \ref{ldplowerboundthm}. The only remaining loose ends, besides the proofs of the main theorems, are the proofs of Proposition \ref{thesegraphonsworkprop}, Proposition \ref{mostgraphsprop}, and Theorem \ref{cycleunionthm} as well as the deduction of Corollary \ref{correctexponentcor} from Theorem \ref{correctexponent}. We will deal with these over the next three sections.
\section{Proof of Proposition \ref{thesegraphonsworkprop}}\label{thesegraphonsworksection}
Take $n\to\infty$ and $p=p(n)\ll 1$.

We must check all conditions of Theorem \ref{ldplowerboundthm} in both of the following two cases.
\begin{case}
Let $K$ be an arbitrary graph with $\gamma=\gamma(K)>0$. Let $W=W(n)$ be set to any of the graphons in Figures \ref{constantgraphonfigure}, and suppose that $(n^{-1}\log n)^{\frac{1}{2e(K)-2-\gamma}}\ll p\ll 1$.
\end{case}
\begin{case}
Let $K=K_0$ be the graph from Figure \ref{K24plusanedgefigure}. Let $W=W(n)$ be the graphon from Figure \ref{K24graphonfigure} with $a(p)=\Theta\left(p^2\left(\log\frac{1}{p}\right)^{-\frac{1}{3}}\left(\log\log\frac{1}{p}\right)^{\frac{1}{3}}\right)$ and $b(p)=\Theta\left(p\left(\log\frac{1}{p}\right)^{\frac{2}{3}}\left(\log\log\frac{1}{p}\right)^{-\frac{2}{3}}\right)$, and suppose that $(n^{-1}\log n)^{\frac{1}{15}}\ll p\ll 1$.
\end{case}
 In both cases, it is clear that $W$ is a block graphon on a fixed number of blocks. Furthermore, $n^{-1}\log\log n\ll p\ll 1$, because for all graphs $K$ with $e(K)\geq 3$ we may take some $H\subseteq K$ such that $2e(K)-2-\gamma=2e(K)-2-\frac{e(H)-v(H)}{c(H)}\geq 2e(K)-2-e(H)\geq e(K)-2\geq 1$, and if $e(K)\leq 2$ $K$ must be a forest.

What remains is to check that all of conditions (1) through (10) of Conditions \ref{lowerboundconditions} hold in both cases.

Condition (1) holds as we constructed our graphons to be regular. (2) holds by inspection. (6) holds for any block $(i,j)$ with $\log\frac{w_{ij}}{p}\gg\log\log\log\frac{1}{p}$ as then by Lemma \ref{entropyapproxlemma}, $I_p(W)\geq (1-o(1))m(S_i)m(S_j)w_{ij}\log\frac{w_{ij}}{p}\gg m(S_i)m(S_j)w_{ij}\log\log\log\frac{1}{p}$. This clearly holds for all blocks $(i,j)$ with $w_{ij}=1$, so (6) easily holds in Case 1, and in Case 2 it suffices to note that $\log\frac{b(p)}{p}\gg\log\log\log\frac{1}{p}$. (3) holds by Lemma \ref{constructionenoughhoms} and the argument in the proof of the upper bound of Theorem \ref{K24correctconstant}.

The remaining conditions are (4), (5), (7), (8), (9), and (10).
\subsection{Proof of (4), (5), (7), (8), (9), and (10) for the Graphons in Case $1$}\label{checkingfigure2subsection}
Notice that for any $H\subseteq K$, $c(H)\leq\frac{v(H)}{2}$, as setting all vertices to have weight $\frac{1}{2}$ is a valid fractional cover. Thus $2+\frac{e(H)-v(H)}{c(H)}\leq\frac{e(H)}{c(H)}\leq e(H)\leq e(K)$ for all nonempty $H\subseteq K$, so $2+\gamma\leq e(K)$.

Therefore $2e(K)-2-\gamma\geq 2+\gamma$, so $p\gg (n^{-1}\log n)^{\frac{1}{2+\gamma}}$.

Since $w_{ij}\geq p$ for all important blocks $(i,j)$, (5) and (9) are satisfied as long as $p\cdot m(S_i)\gg n^{-1}\log n$ for all $i$. Since the smallest $S_i$ is of size $\Theta(p^{1+\gamma})$, and $p\gg (n^{-1}\log n)^{\frac{1}{2+\gamma}}$, this holds.

By Lemma \ref{constructionlowentropy}, $I_p(W)=(1+o(1))(2z+w)p^{2+\gamma}\log\frac{1}{p}$ (with $z=0$ or $w=0$ if we are looking at one of the first two graphons from the figure). Thus $I_p(W)=\Theta\left(p^{2+\gamma}\log\frac{1}{p}\right)$.

So since $(n^{-1}\log n)^{\frac{1}{2+\gamma}}\ll p$, $I_p(W)\gg n^{-1}\log n$, proving the left inequality of (4). For the right inequality, since $\log\frac{1}{p}\leq\log n$, we must just show that $n^{-1}\log n\ll p^{2e(K)-2-\gamma}$, which follows from our conditions on $p$.

Notice that all blocks with value $0$ must be negligible by definition. Thus all the non-negligible blocks $(i,k)$ in the last column (or last row) of the graphons in Figure \ref{constantgraphonfigure} have $m(S_i)m(S_k)=\Omega\left(p^{1+\frac{\gamma}{2}}\right)\gg p^{1+\gamma}\log\frac{1}{p}=\Theta(p^{-1}I_p(W))$, since $\gamma>0$.

So to prove (8) (by using (6)) it suffices to show that all blocks $(i,j)$ with $1\leq i,j\leq k-1$, $w_{ij}=p$ are either important or have $m(S_i)m(S_j)\geq p^{1+\gamma}\log\frac{1}{p}$. Notice that due to the structure of $W$, all such blocks have $m(S_i)m(S_j)=\Theta(p^c)$ for some constant $c$. If $c<1+\gamma$, we are done, as $\log\frac{1}{p}=p^{-o(1)}$. Otherwise, $m(S_i)m(S_j)w_{ij}=p^{c+1}\leq p^{2+\gamma}\ll\left(\log\log\log\frac{1}{p}\right)^{-1}I_p(W)$, as $I_p(W)=\Theta\left(p^{2+\gamma}\log\frac{1}{p}\right)$, so $(i,j)$ is important and we are also done in this case, proving (8).

We now show the more difficult statements (7) and (10). First we will assume (10) and prove (7), and then prove (10).

For (7), note that the second graphon in Figure \ref{constantgraphonfigure} has no somewhat important blocks. Thus we may assume we are in the first or third case, where there is a `hub' of size $\Theta(p^{1+\gamma})$. Suppose for the sake of contradiction that there is some non-negligible $K$-block $B$ that sends two non-disjoint edges into somewhat important blocks of $W$. Then there is some $u,v,w\in V(K)$ such that $uv$ and $uw$ are edges of $K$, and $B$ sends $uv$ and $uw$ to somewhat important blocks. Let $B(u)=i$, $B(v)=j_1$, and $B(w)=j_2$.

The only somewhat important blocks have value $p$, so $w_{ij_1}=w_{ij_2}=p$. Therefore, since $(i,j_1)$ and $(i,j_2)$ are somewhat important, $m(S_i)m(S_{j_1})\ll p^{-1}\left(\log\log\log\frac{1}{p}\right)^{-1}I_p(W)$, and similarly for $j_2$. By the construction of $W$, $m(S_i)m(S_{j_1})$ must be $\Theta(p^c)$ for some constant $c$, and since $I_p(W)=\Theta\left(p^{2+\gamma}\log\frac{1}{p}\right)$, this implies that
\[m(S_i)m(S_{j_1})=O(p^{1+\gamma}),\]
and similarly for $S_{j_2}$. Since the only blocks with value $p$ have area at least $p^{2+\frac{\gamma}{2}}$, this implies that $\gamma\leq 2$.

Now, suppose we modify $B$ to some $K$-block $B'$, where $u$ is now sent into the hub $S_1$ instead of $S_i$. We further modify so that all vertices adjacent to $u$ in $K$ that were sent into the large interval (of size $1-p$) are now sent into the interval of size $p+o(p)$. We compute $\frac{\Hom_{B'}(K,W)}{\Hom_B(K,W)}$.

In going from $B$ to $B'$, we lose a factor of $p$ for both the edges $uv$ and $uw$, as those are now sent into a block with value $1$. We also gain a factor of $\Theta\left(\frac{p^{1+\gamma}}{m(S_i)}\right)$, as we are sending $u$ into a smaller block.

When moving each neighbor of $u$ (for example, call one of them $u_0$) that was originally sent into the large block, we gained a factor of $p+o(p)$ from sending $u_0$ into a $p+o(p)$-times smaller block, but lost at least a factor of $p+o(p)$ as now $uu_0$ is sent into a block with value $1$ instead of $p+o(p)$. Thus
\[\frac{\Hom_{B'}(K,W)}{\Hom_B(K,W)}=\Omega\left(\frac{p^{\gamma-1}}{m(S_i)}\right).\]
By (10), $\Hom_{B'}(K,W)\leq\Hom(K,W+p)=O(p^{e(K)})$. Thus if $m(S_i)\ll p^{\gamma-1}$, $\Hom_B(K,W)\ll p^{e(K)}$ and $B$ is negligible, a contradiction. Thus since $m(S_i)=\Omega(p^c)$ (by inspection of $W$), $m(S_i)=\Omega(p^{\gamma-1})$.

Looking at $W$, we see that we must have $m(S_i)\leq p$. Thus $\gamma-1\geq 1$ and $\gamma\geq 2$. Combining this with our argument earlier, we must have that $\gamma=2$, and so $S_i$ must be the interval of size $p-o(p)$. Since we must have $m(S_i)m(S_{j_1})=O(p^{1+\gamma})=O(p^3)$, $S_{j_1}$ must be the interval of size $(1+o(1))p^{1+\gamma/2}=(1+o(1))p^2$ in the third graphon of Figure \ref{constantgraphonfigure}.

In summary, we have shown that $\gamma=2$ and (labelling the four intervals of this third graphon $S_1$ through $S_4$ in the natural way) that $i=3$ and $j_1=j_2=2$. However, this will again cause an issue. Create a new $K$-block $B''$ which is identical to $B$ except that it sends $u$ into $S_2$ instead of $S_3$. This gains one factor of $p+o(p)$ (since $m(S_2)=(p+o(p))m(S_3)$ but loses two factors of $p$ since $uv$ and $uw$ are now sent into blocks with value $1$. (There are several irrelevant factors of $1+o(1)$ coming from edges $uu_0$ where $u_0$ was sent into $S_4$.) Thus $\Hom_B(K,W)\leq (p+o(p))\Hom_{B'}(K,W)=O(p^{e(K)+1})$, again by (10). Thus $B$ is negligible, again a contradiction. This proves (7) given (10).

For the graphons in Figure \ref{constantgraphonfigure}, it only remains to prove (10); that is, that $\Hom(K,W+p)=O(p^{e(K)})$. Since $W+p\leq 2\max(W,p)$, it suffices to show that $\Hom(K,\max(W,p))=O(p^{e(K)})$.

We show that for any $K$-block $B$, $\Hom_B(K,\max(W,p))=O(p^{e(K)})$. Call the intervals in the third diagram of Figure \ref{constantgraphonfigure} $S_1,S_2,S_3,S_4$ in that order. (If we are in one of the other two diagrams, one of $S_1$ or $S_2$ may be empty.) Let $A_i=B^{-1}(i)$, so that all vertices in $A_i$ are sent into $S_i$.

If $H\subseteq K$ is the subgraph with $E(H)=E(A_1,A_1\cup A_2\cup A_3)\cup E(A_2)$, so that the edges of $H$ are exactly those sent into blocks of value $1$, we can easily see that
\[\Hom_B(K,\max(W,p))=O\left(p^{(1+\gamma)|A_1|+(1+\gamma/2)|A_2|+|A_3|}+e(K)-e(H)\right).\]
Since $|A_1|+|A_2|+|A_3|=v(H)$, we thus must show that $e(H)\leq v(H)+\gamma(|A_1|+|A_2|/2)$. But $e(H)\leq v(H)+\gamma c(H)$ by the definition of $\gamma$, and $|A_1|+|A_2|/2\geq c(H)$ because giving all vertices in $A_1$ weight $1$, all vertices in $A_2$ weight $\frac{1}{2}$, and all vertices in $A_3$ weight $0$ is a fractional vertex cover of $H$. This completes the proof of (10) and the proof of the first half of Proposition \ref{thesegraphonsworkprop}.
\subsection{Proof of (4), (5), (7), (8), (9), and (10) in Case 2}
We first show (4). We know that $I_p(W)=\Theta\left(p^3\left(\log\frac{1}{p}\right)^{\frac{2}{3}}\left(\log\log\frac{1}{p}\right)^{\frac{1}{3}}\right)$ by the argument in the proof of the upper bound of Theorem \ref{K24correctconstant}. Since in this case $e(K_0)=9$, we must show $n^{-1}\ll I_p(W)\ll p^{18}n$. These both hold as long as
\[n^{-1}\ll p^{15}\left(\log\frac{1}{p}\right)^{-\frac{2}{3}}\left(\log\log\frac{1}{p}\right)^{-\frac{1}{3}},\]
which holds since $p\gg (n^{-1}\log n)^{\frac{1}{15}}$ by the conditions given.

We now show (5) and (9). We must show that $a(p),pb(p)\gg n^{-1}\log n$. Since $a(p)\lesssim pb(p)$, we just must show $n^{-1}\log n\ll a(p)$. But by the given bounds on $p$, $n^{-1}\log n\ll p^{15}\ll a(p)$, since $a(p)=\Theta\left(p^2\left(\log\frac{1}{p}\right)^{-\frac{1}{3}}\left(\log\log\frac{1}{p}\right)^{\frac{1}{3}}\right)$.

To show (8), first note (as before) that it is impossible for any non-negligible $K$-block to send any edge to a block with value $0$. Since the only unimportant blocks with nonzero value have area $p+o(p)\gg p^{-1}I_p(W)$, (8) is proven.

We are left to show (7) and (10). Label the intervals $S_1$, $S_2$, and $S_3$ in that order, so that $m(S_1)=a(p)$, $m(S_2)=p-a(p)$, and $m(S_3)=1-p$. The only somewhat important block is $(2,2)$. Suppose for the sake of contradiction that some non-negligible $K_0$-block $B$ sends two adjacent edges $uv,uw\in E(K_0)$ into $(2,2)$; that is, $B(u)=B(v)=B(w)=2$. Then modifying $B$ by sending $u$ to $S_1$ and sending all vertices adjacent to $u$ originally sent to $S_3$ into $S_2$, by a similar analysis to the last subsection, we have that $\Hom_B(K_0,W)\leq\frac{b(p)^2}{p}\Hom_{B'}(K_0,W)$. Assuming (10), $\Hom_{B'}(K_0,W)\leq\Hom(K_0,W+p)=O(p^{e(K_0)})$, so since $b(p)^2\ll p$, $\Hom_B(K_0,W)\ll p^{e(K_0)}$, so $B$ is negligible, a contradiction. Thus it now only suffices to show (10).

Similarly to the last subsection, to prove (10) it suffices to show that for all $K_0$-blocks $B$, $\Hom_B(K_0,\max(W,p))=O(p^{e(K_0)})$. Let $A_i=B^{-1}(i)$ and $H=E(A_1,A_1\cup A_2)$. Then since $a(p)=p^{2+o(1)}$ and $b(p)=p^{1+o(1)}$, and since $e(K_0)=9$, it is easy to see that
\[\Hom_B(K_0,\max(W,p))=p^{(1+o(1))(2|A_1|+|A_2|+9-e(H))}.\]
Thus if $2|A_1|+|A_2|>e(H)$, we are done.

We now consider when we can have $2|A_1|+|A_2|\leq e(H)$, while keeping in mind that $E(H)=E(A_1,A_1\cup A_2)$. We must have that $A_1$ is a vertex cover of $H$, so $|A_1|\geq c(H)$. Since $|A_1|+|A_2|\geq v(H)$, we have $e(H)\geq c(H)+v(H)$. By the argument in Lemma \ref{K0contributingsubgraphslemma}, this means that $H\subseteq\{\emptyset,K_{2,4},K_0\}$, and we also must have $A_1\cup A_2=V(H)$. Since $A_1$ must be a minimum vertex cover of $H$, in the case $H=K_{2,4}$ we must have that $A_1$ consists of the two vertices of degree $4$, and in the case $H=K_0$ we must have that $A_1$ consists of the two vertices of degree $4$ and one vertex of degree $3$.

In the case where $H=\emptyset$, since $A_1\cup A_2=V(H)$, $A_1=A_2=\emptyset$. Thus $B(v)=3$ for all $v\in V(K)$, so $\Hom_B(K_0,\max(W,p))=(1+o(1))p^9$. In the case where $H=K_0$ , all edges are sent into blocks of value $1$. Since three vertices are sent into each of $S_1$ and $S_2$, $\Hom_B(K_0,\max(W,p))=(1+o(1))p^3a(p)^3\ll p^9$.

This only leaves the case where $H=K_{2,4}$. This case corresponds to when the two vertices of degree $4$ are sent into $S_1$ and the other four are sent into $S_2$. In this case, $\Hom_B(K,\max(W,p))=(1+o(1))p^4a(p)^2b(p)=\Theta(p^9)$, by the definitions of $a(p)$ and $b(p)$. This proves (10), and thus completes this section and the proof of Proposition \ref{thesegraphonsworkprop}.
\section{Finishing the Log Gap}\label{correctexponentcorsection}
In this section, we will show how to modify the argument of \cite{BD} to prove Theorem \ref{cycleunionthm}, and then deduce Corollary \ref{correctexponentcor} from Theorems \ref{correctexponent} and \ref{cycleunionthm}.
\begin{proof}[Proof of Theorem \ref{cycleunionthm}]
Technically, the argument of Bhattacharya and Dembo \cite{BD} only deals with the case where the $2$-core of $K$ is a single cycle, instead of a disjoint union of cycles. However, the same proof goes through in the disjoint union of cycles case almost identically. We will largely just describe the slight changes that must be made in the Bhattacharya-Dembo proof in order for it to apply here.

Notice that if we remove a leaf from $K$, then both $\Hom(K,G_n^d)$ and $n^{v(K)}p^{e(K)}$ change by a factor of $d=np$, so $-\log\Pr\left[\Hom(K,G_n^d)\geq (1+\delta)p^{e(K)}n^{v(K)}\right]$ stays constant. Thus we may assume without loss of generality that $K$ itself (not just its $2$-core) is a disjoint union of cycles.

Recall that $c(i_1,\ldots,i_{\ell};\delta)$ is the positive value of $c$ such that $\displaystyle\prod_{j=1}^{\ell}(1+\lfloor c\rfloor+\{c\}^{i_j/2})=1+\delta$. This exists because $\displaystyle\prod_{j=1}^{\ell}(1+\lfloor x\rfloor+\{x\}^{i_j/2})$ is a continuous increasing function on $(0,\infty)$, so it has a well-defined inverse on $(1,\infty)$.

Now, since $K=\bigcup_{1\leq j\leq\ell}C_{i_j}$,
\[\Hom(K,G)=\displaystyle\prod_{j=1}^{\ell}\Hom(C_{i_j},G)\]
for any graph $G$. Therefore, if $\Hom(K,G_n^d)\geq (1+\delta)n^{v(K)}p^{e(K)}$, there must be some $j$, $1\leq j\leq\ell$, such that $\Hom(C_{i_j},G_n^d)\geq (1+\lfloor c\rfloor+\{c\}^{i_j/2})n^{i_j}p^{i_j}$. This gives a bound
\[\Pr\left[\Hom(K,G_n^d)\geq (1+\delta)p^{e(K)}n^{v(K)}\right]\leq\displaystyle\sum_{j=1}^{\ell}\Pr\left[\Hom(C_{i_j},G_n^d)\geq (1+\lfloor c\rfloor+\{c\}^{i_j/2})n^{i_j}p^{i_j}\right].\]
So to prove the upper bound of Theorem \ref{cycleunionthm}, it suffices to show that
\[\Pr\left[\Hom(C_{i_j},G_n^d)\geq (1+\lfloor c\rfloor+\{c\}^{i_j/2})n^{i_j}p^{i_j}\right]\leq\exp\left(-\left(\frac{c}{2}-o(1)\right)n^2p^2\log\frac{1}{p}\right).\]
Having reduced to a single cycle, noting that $\lfloor x\rfloor+\{x\}^{i_j/2}$ and $\lfloor x\rfloor+\{x\}^{2/i_j}$ are inverse functions on $[0,\infty)$, this follows from Theorems 1.5 (a) and 1.1 of \cite{BD}.

We now show the lower bound of Theorem \ref{cycleunionthm}. We parallel the cycle argument in Section 2.3 of \cite{BD}. Define $X_n^{\star}$ as in (2.2) of \cite{BD}, with $\lfloor c\rfloor$ cliques of size $d+1$ and one clique of size $s_1\sim\{c\}^{\frac{1}{2}}d$, and define $P_{\star}$ in the same way. It is easy to compute that $I_p(X_n)=(c+o(1))d^2\log\frac{1}{p}=(c+o(1))n^2p^2\log\frac{1}{p}$, so we have the correct entropy and
\[\frac{d\mathbb{P}_p}{d\mathbb{P}_{\star}}=\exp\left(-\left(\frac{c}{2}+o(1)\right)n^2p^2\log\frac{1}{p}\right).\]
Following the argument in \cite{BD} up to (2.46), we must show that
\[\mathbb{P}_{\star}(\Hom(K,G_n)\leq (1+\delta-\Omega(1))n^{v(K)}p^{e(K)})\ll\mathbb{P}_{\star}(\mathcal{K}_n^d)=\exp(-o(n^2p^2\log(1/p))),\]
where $\mathcal{K}_n^d$ is the event that a graph is $d$-regular. The second bound is already proved for this $\mathbb{P}_{\star}$ in \cite{BD} (see the analysis of Case 1 after (2.51)).

To finish the proof, we upper bound $\mathbb{P}_{\star}(\Hom(K,G_n)\leq (1+\delta-\Omega(1))n^{v(K)}p^{e(K)})$. Since
\[1+\delta=\displaystyle\prod_{j=1}^{\ell}(1+\lfloor c\rfloor+\{c\}^{i_j/2}),\]
it suffices to upper bound
\[\mathbb{P}_{\star}(\Hom(C_{i_j},G_n)\leq (1+\lfloor c\rfloor+\{c\}^{i_j/2}-\Omega(1))n^{i_j}p^{i_j}).\]
But there are $(1+o(1))(\lfloor c\rfloor+\{c\}^{i_j/2})n^{i_j}p^{i_j}$ homomorphisms from $C_{i_j}$ into the $\lceil c\rceil$ planted cliques, so again we are looking at a lower tail probability for the number of homomorphisms of a cycle into $G(n-o(n),p)$, and the $\exp(-\Theta(n^2p))$ upper bound holds as in \cite{BD}. Since $n^2p\gg n^2p^2\log(1/p)$ for our range of $p$, we are done.
\end{proof}
\begin{proof}[Deduction of \ref{correctexponentcor} from Theorems \ref{correctexponent} and \ref{cycleunionthm}]
\setcounter{casecounter}{0}
Take any nonforest graph $K$ and fix $\delta>0$.
\begin{case}
The $2$-core of $K$ is not a disjoint union of cycles.
\end{case}
In this case, by Theorem \ref{correctexponent} it suffices to show that $\rho(K,\delta):=\displaystyle\min_{\substack{z,w\geq 0 \\ P_K(z,w)\geq 1+\delta}}\left(z+\frac{w}{2}\right)$ is not $\infty$. In other words, we must show that $P_K$ contains some nonconstant monomial (since all coefficients of $P_K$ are positive by definition). Thus we must show that $K$ has at least one contributing subgraph $H$ with a valid subset $A\subseteq V(H)$. But the vertices of the fractional vertex cover of any $H$ have half-integer coordinates (see for example Theorem 64.11 of \cite{Sch}), so there is some half-integer-valued minimum fractional vertex cover of $H$. Thus any $H$ has at least one valid subset. So it suffices to show that $K$ has at last one contributing subgraph $H$.

The $2$-core of $K$ is not a disjoint union of cycles, so it has more edges than vertices. Thus $\gamma(K)>0$. Take any $H$ such that $\frac{e(H)-v(H)}{c(H)}=\gamma(K)$. $H$ cannot be a forest, since we must have $e(H)-v(H)=\gamma(K)c(H)>0$.

If $H$ has a leaf, we may remove it and its single edge and keep $e(H)-v(H)$ constant while not increasing $c(H)$. Thus the new graph $H'$ given by removing the leaf also has $\frac{e(H')-v(H')}{c(H')}=\gamma(K)$. Thus repeatedly removing the leaves one by one, we arrive at the $2$-core of $H$, call it $H_2$, and we have shown that $\frac{e(H_2)-v(H_2)}{c(H_2)}=\gamma(K)$. Since $\delta(H_2)\geq 2$, $H_2$ is a contributing subgraph of $K$ and we are done.
\begin{case}
The $2$-core of $K$ is a disjoint union of cycles.
\end{case}
Since $K$ is not a forest, the $2$-core of $K$ contains at least one cycle. Let $i_1,\ldots,i_{\ell}$ be the cycle lengths in the $2$-core of $K$. Recall the definition of $c(i_1,\ldots,i_{\ell};\delta)$ from Theorem \ref{cycleunionthm}. We can take $c$ sufficiently large such that $\lfloor c\rfloor\geq\delta$, and since $\ell\neq 0$ we thus have
\[\displaystyle\prod_{j=1}^{\ell}(1+\lfloor c\rfloor+\{c\}^{i_j/2})\geq 1+\delta.\]
Thus $c(i_1,\ldots,i_{\ell};\delta)\neq\infty$. Since $\delta>0$, we also have $c(i_1,\ldots,i_{\ell};\delta)>0$. Thus by Theorem \ref{cycleunionthm}
\[-\log\Pr\left[\Hom(K,G_n^d)\geq (1+\delta)p^{e(K)}n^{v(K)}\right]=\Theta\left(n^2p^2\log\frac{1}{p}\right).\]

Note that $\gamma(K)=0$. This is because the $2$-core $H\subseteq K$ is a disjoint union of cycles and thus $\frac{e(H)-v(H)}{c(H)}=0$, and for any $H'\subseteq K$ the $2$-core of $H'$ must also be a (possibly empty) disjoint union of cycles and thus $e(H')<v(H')$.

Thus the only thing remaining to prove is that we have covered the entire desired range of $p$. In other words, we must show that
\[n^{-\frac{1}{3}}\lesssim (n^{-1}\log n)^{\frac{1}{2e(K)-2-\gamma(K)}},\]
so it suffices to show that $2e(K)-2-\gamma(K)\geq 3$. But $\gamma(K)=0$, and since $K$ contains at least one cycle, $e(K)\geq 3$, so $2e(K)-2-\gamma(K)\geq 4$. This finishes the proof of Corollary \ref{correctexponentcor}.
\end{proof}
\section{Proof of Proposition \ref{mostgraphsprop}}\label{mostgraphssection}
Let $K$ be any graph and suppose $\gamma=\gamma(K)>2$. We prove the conditions of Theorem \ref{correctconstant} hold, namely that $K$ is not a forest and that no contributing subgraphs of $K$ have bad edges.

Firstly, if $K$ is a forest, then $e(H)<v(H)$ for all nonempty subgraphs $H\subseteq K$, so $\gamma<0$, a contradiction. Thus $K$ is not a forest.

Take any contributing subgraph $H\subseteq G$. By the definition of contributing, $\frac{e(H)-v(H)}{c(H)}=\gamma>2$. Take some edge $e_0\in H$ and let $H'=H\backslash e_0$. We must have $\frac{e(H)-v(H)-1}{c(H')}\leq\frac{e(H')-v(H')}{c(H')}\leq\gamma=\frac{e(H)-v(H)}{c(H)}$. Using the simple inequality that $\frac{a+b}{c+d}$ is between $\frac{a}{b}$ and $\frac{c}{d}$ for $a,c$ nonnegative and $b,d$ positive, we must have that $\frac{1}{c(H)-c(H')}\geq\gamma>2$ or $c(H)=c(H')$. Thus $c(H)-c(H')<\frac{1}{2}$.

But $c(H)$ and $c(H')$ are both half-integer, as the vertices of the fractional vertex cover polytope all have half-integer coordinates (as in the previous section). Thus $c(H)=c(H')$.

Since the minimum fractional vertex cover number is the same as the maximum fractional matching number by (3) of Lemma \ref{edgeweightslemma}, we may construct a maximum fractional matching $(w_e)_{e\in E(H')}$ of $H'$ with $\sum w_e=c(H')=c(H)$. Setting $w_{e_0}=0$, we thus obtain a maximum fractional matching on $H$ with $e_0$ having weight $0$. Thus $e_0$ is not a bad edge. Since $e_0$ was arbitrary, we have proven that $H$ has no bad edges, and we have proven the conditions of Theorem \ref{correctconstant}.

To complete the proof of the first part of Proposition \ref{mostgraphsprop}, we must show that $\gamma>2$ as long as $K$ or any subgraph of $K$ has average degree greater than $4$. But since giving all vertices weight $\frac{1}{2}$ is always a valid way to generate a fractional vertex cover of any graph, $c(H)\leq\frac{v(H)}{2}$ for any $H\subseteq K$. Thus if $H\subseteq K$ has average degree greater than $4$, then
\[2<\frac{e(H)}{v(H)}=1+\frac{e(H)-v(H)}{v(H)}\leq 1+\frac{e(H)-v(H)}{2c(H)},\]
so $\gamma\geq\frac{e(H)-v(H)}{c(H)}>2$, as desired.

To prove the second part, now take a nonforest $K$ with $v(K)\leq 5$. Take any $H\subseteq K$ contributing. Then $v(H)\geq 5$ and $H$ has minimum degree at least $2$. We show that $H$ has no bad edges.

If $H$ has a Hamiltonian cycle, $H$ has no bad edges, because assigning all edges in the Hamiltonian cycle weight $\frac{1}{2}$ is a fractional perfect matching and thus a fractional vertex cover. The only graphs $H$ on at most $5$ vertices with $\delta(H)\geq 2$ and no Hamiltonian cycle are $K_{2,3}$, $K_{1,1,3}$, and the butterfly graph (two triangles joined at a vertex as in Example \ref{butterflyexample}). For the butterfly graph, we may obtain a fractional perfect matching by assigning all edges containing the vertex of degree $4$ weight $\frac{1}{4}$, and the other two edges weight $\frac{1}{2}$ (this is the same matching we used in Example \ref{butterflyexample}).

Both $H=K_{2,3}$ and $H=K_{1,1,3}$ have $c(H)=2$, so we may take a fractional matching by giving all edges in $K_{2,3}$ weight $\frac{1}{3}$ (and similarly for $K_{1,1,3}$, as it contains a copy of $K_{2,3}$). Thus for each of the three graphs we have constructed a maximum fractional matching with no edges of weight $1$. Thus none of them have any bad edges. This completes the proof of the Proposition.
\section{Proofs of Main Theorems}\label{maintheoremsection}
In this section, we prove Theorems \ref{correctexponent}, \ref{correctconstant}, and \ref{K24thm}.
\begin{proof}[Proof of Upper Bound of Theorems \ref{correctexponent} and \ref{correctconstant}]
Take any nonforest graph $K$ with the $2$-core of $K$ not a disjoint union of cycles. Then the $2$-core of $K$ must have more edges than vertices, and thus $\gamma:=\gamma(K)>0$. Take $n\to\infty$ and $d=d(n)$, $p:=\frac{d}{n}$ with $(n^{-1}\log n)^{\frac{1}{2e(K)-2-\gamma(K)}}\ll p\ll 1$. Finally, fix some constant $\delta>0$.

By the upper bound of Theorem \ref{varcorrectexponent}, since $p\to 0$, there exists $W=W(n)$ such that $\Hom(K,W)\geq (1+\delta)p^{e(K)}$, $I_p(W)\leq (2+o(1))\rho(K,\delta)p^{2+\gamma(K)}\log\frac{1}{p}$, and $W$ satisfies the conditions of Theorem \ref{ldplowerboundthm}.

Thus applying Theorem \ref{ldplowerboundthm},
\[-\log\left(\Pr\left[\Hom(K,G_n^d)\geq (1-o(1))\Hom(K,W)n^{v(K)}\right]\right)\leq\left(\frac{1}{2}+o(1)\right)I_p(W)n^2.\]
Substituting our bounds on $\Hom(K,W)$ and $I_p(W)$, we see that
\[-\log\left(\Pr\left[\Hom(K,G_n^d)\geq (1+\delta-o(1))p^{e(K)}n^{v(K)}\right]\right)\leq\left(1+o(1)\right)\rho(K,\delta)n^2p^{2+\gamma(K)}\left(\log\frac{1}{p}\right).\]
This is exactly the statement of the upper bounds of Theorems \ref{correctexponent} and \ref{correctconstant} except that $1+\delta$ is replaced by $1+\delta-o(1)$ on the left side. However, since $\rho(K,(1+o(1))\delta)=(1+o(1))\rho(K,\delta)$ (as we showed in the proof of the upper bound of Theorem \ref{varcorrectexponent} in Section \ref{constructionssection}), we may absorb this $o(1)$ into the $1+o(1)$ factor on the right hand side by increasing $\delta$ by $o(1)$, finishing the proof.
\end{proof}
\begin{proof}[Proof of Lower Bound of Theorems \ref{correctexponent} and \ref{correctconstant}]
Fix any nonforest graph $K$ with $2$-core not a disjoint union of cycles. We have that $\gamma:=\gamma(K)>0$ for the same reason as in the previous proof. Take $n\to\infty$, $d=d(n)$ and $p:=\frac{d}{n}$ with $(n^{-1}\log n)^{\frac{1}{2e(K)-2-\gamma(K)}}\ll p\ll 1$. Fix a constant $\delta>0$.

Recall the definition of $\Phi_n^d(K,t)$ from Definition \ref{phidef}. By the lower bounds of Theorems \ref{varcorrectexponent} and \ref{varcorrectconstant},
\[\Phi_n^d(K,1+\delta)\gtrsim p^{2+\gamma}n^2,\]
and if no contributing subgraphs of $K$ have bad edges,
\[\Phi_n^d(K,1+\delta)\geq (1-o(1))\rho(K,\delta)n^2p^{2+\gamma(K)}\log\frac{1}{p}.\]
Thus we may apply Theorem \ref{ldpupperboundthm} to finish the argument (applying the same method as the previous proof to remove the $o(1)$ from the $1+\delta+o(1)$ in the result), provided that we can show that
\[(n^{-1}\log n)^{\frac{1}{2e(K)-2-\gamma(K)}}\geq (n^{-1}\log n)^{\frac{1}{2\Delta_*(K)}},\]
or in other words, that
\[2\Delta_*(K)\leq 2e(K)-2-\gamma(K).\]
Recalling from Definition \ref{deltastardef} the definition of $\Delta_*$, $2\Delta_*(K)=\displaystyle\max_{vw\in E(K)}(\deg v+\deg w)$. Here all edges of $K$ are counted at most once except $vw$, which is counted twice. So $2\Delta_*(K)=e(K)+1$, so it suffices to prove that
\[\gamma(K)\leq e(K)-3.\]
For any $H$ with $v(H)<3$, $e(H)-v(H)<0$, so since $\gamma>0$, when computing $\gamma$ we may take the maximum only over subgraphs with at least $3$ vertices. But for any $H$ with $v(H)\geq 3$, $\frac{e(H)-v(H)}{c(H)}\leq e(H)-v(H)\leq e(H)-3$. Thus $\gamma(K)\leq e(K)-3$, finishing the proof.
\end{proof}
\begin{proof}[Proof of Theorem \ref{K24thm}]
Let $K=K_0$ be the graph from Theorem \ref{K24thm} (appearing in Figure \ref{K24plusanedgefigure}). Take $n\to\infty$, $d=d(n)$ and $p:=\frac{d}{n}$ with $(n^{-1}\log n)^{\frac{1}{15}}\ll p\ll 1$. Fix a constant $\delta>0$.

Let $W=W(n)$ be a $p$-regular graphon with $\Hom(K_0,W)\geq (1+\delta)p^9$ and $I_p(W)=(1+o(1))(18\delta)^{\frac{1}{3}}p^3\left(\log\frac{1}{p}\right)^{\frac{2}{3}}\left(\log\log\frac{1}{p}\right)^{\frac{1}{3}}$ that satisfies the conditions of Theorem \ref{ldplowerboundthm}, as guaranteed by Theorem \ref{K24correctconstant}.

Applying Theorem \ref{ldplowerboundthm} to $W$, we have that
\begin{align*}
& -\log\left(\Pr\left[\Hom(K_0,G_n^d)\geq (1+\delta-o(1))p^9n^6\right]\right) \\ & \leq (1+o(1))\frac{(18\delta)^{\frac{1}{3}}}{2}n^2p^3\left(\log\frac{1}{p}\right)^{\frac{2}{3}}\left(\log\log\frac{1}{p}\right)^{\frac{1}{3}}.
\end{align*}
Multiplying $\delta$ by a $1+o(1)$ factor to cancel the $o(1)$ inside the probability, we have proven the upper bound of Theorem \ref{K24thm}.

For the lower bound, Theorem \ref{K24correctconstant} implies that
\[\Phi_n^d(K_0,1+\delta)\geq (1+o(1))\frac{(18\delta)^{\frac{1}{3}}}{2}n^2p^3\left(\log\frac{1}{p}\right)^{\frac{2}{3}}\left(\log\log\frac{1}{p}\right)^{\frac{1}{3}}.\]
We may easily compute $2\Delta_*(K_0)=7<15$, so $p\gg (n^{-1}\log n)^{\frac{1}{2\Delta_*(K_0)}}$. Thus Theorem \ref{ldpupperboundthm} implies that
\begin{align*}
& -\log\left(\Pr\left[\Hom(K_0,G_n^d)\geq (1+\delta+o(1))p^9n^6\right]\right) \\ & \geq (1+o(1))\frac{(18\delta)^{\frac{1}{3}}}{2}n^2p^3\left(\log\frac{1}{p}\right)^{\frac{2}{3}}\left(\log\log\frac{1}{p}\right)^{\frac{1}{3}}.
\end{align*}
Multiplying the $\delta$ by $(1-o(1))$ to cancel out the $o(1)$ inside the probability, we have proven the lower bound of Theorem \ref{K24thm}. This finishes the proof.
\end{proof}
\section{Acknowledgements}
The author would like to thank his PhD advisor Yufei Zhao for introducing him to this problem, as well as helpful input and advice throughout the process. The author would also like to thank Nick Cook, Lutz Warnke, Xiaoyu He, and Wojtek Samotij for their helpful comments on earlier drafts.

\end{document}